\documentclass[11pt]{book}
\usepackage{amssymb,latexsym,amsmath,amsthm,amscd}
\usepackage{setspace}
\usepackage[active]{srcltx}
\usepackage[all]{xy} \xyoption{arc}
\usepackage[top = 1in, bottom=1.25in, left=1in, right=1in]{geometry}
\usepackage{graphicx}
\usepackage{makeidx}
\usepackage{fancyhdr}
\usepackage{rotating}
\usepackage{float}
\usepackage{tocloft}
\usepackage{charter}
\usepackage{amsthm}

\makeatletter
\newtheorem*{rep@theorem}{\rep@title}
\newcommand{\newreptheorem}[2]{%
\newenvironment{rep#1}[1]{%
 \def\rep@title{#2 \ref{##1}}%
 \begin{rep@theorem}}%
 {\end{rep@theorem}}}
\makeatother

\newtheorem{theorem}{Theorem}[section]

\newtheorem{lemma}[theorem]{Lemma}
\newtheorem{proposition}[theorem]{Proposition}
\newtheorem*{proposition*}{Proposition}

\newtheorem{corollary}[theorem]{Corollary}
\newtheorem{conjecture}[theorem]{Conjecture}

\theoremstyle{definition}
\newtheorem{definition}[theorem]{Definition}
\newtheorem{example}[theorem]{Example}
\newtheorem*{example*}{Example}

\theoremstyle{remark}
\newtheorem*{remark}{Remark}
\newtheorem*{remarks}{Remarks}

\theoremstyle{algorithm}
\newtheorem*{algorithm}{Algorithm}

\newcommand{\bA}{\mathbb{A}}

\newcommand{\bC}{\mathbb{C}}

\newcommand{\bG}{\mathbb{G}}

\newcommand{\bN}{\mathbb{N}}

\newcommand{\bP}{\mathbb{P}}
\newcommand{\bQ}{\mathbb{Q}}

\newcommand{\bZ}{\mathbb{Z}}

\DeclareMathOperator{\Tr}{Tr}
\DeclareMathOperator{\Gal}{Gal}

\DeclareMathOperator{\Spec}{Spec}

\newcommand{\Hilb}{\textrm{Hilb}}
\newcommand{\Hilbn}{\textrm{Hilb}^n(\mathbb{A}_k^2)}

\newcommand{\cellsize}{5}
\newlength{\cellsz} 
\newsavebox{\cell}
\sbox{\cell}{\begin{picture}(\cellsize,\cellsize)
\put(0,0){\line(1,0){\cellsize}}
\put(0,0){\line(0,1){\cellsize}}
\put(\cellsize,0){\line(0,1){\cellsize}}
\put(0,\cellsize){\line(1,0){\cellsize}}
\end{picture}}
\newcommand\cellify[1]{\def\thearg{#1}\def\nothing{}%
\ifx\thearg\nothing
\vrule width0pt height\cellsz depth0pt\else
\hbox to 0pt{\usebox{\cell} \hss}\fi%
\vbox to \cellsz{
\vss
\hbox to \cellsz{\hss$#1$\hss}
\vss}}
\newcommand\tableau[1]{\vtop{\let\\\cr
\baselineskip -16000pt \lineskiplimit 16000pt \lineskip 0pt
\ialign{&\cellify{##}\cr#1\crcr}}}
\newcommand{\kellsize}{17}
\newlength{\kellsz} 
\newsavebox{\kell}
\sbox{\kell}{\begin{picture}(\kellsize,\kellsize)
\put(0,0){\line(1,0){\kellsize}}
\put(0,0){\line(0,1){\kellsize}}
\put(\kellsize,0){\line(0,1){\kellsize}}
\put(0,\kellsize){\line(1,0){\kellsize}}
\end{picture}}
\newcommand\kellify[1]{\def\thearg{#1}\def\nothing{}%
\ifx\thearg\nothing
\vrule width0pt height\kellsz depth0pt\else
\hbox to 0pt{\usebox{\kell} \hss}\fi%
\vbox to \kellsz{
\vss
\hbox to \kellsz{\hss$#1$\hss}
\vss}}
\newcommand\ktableau[1]{\vtop{\let\\\cr
\baselineskip -16000pt \lineskiplimit 16000pt \lineskip 0pt
\ialign{&\kellify{##}\cr#1\crcr}}}

\newcommand{\sellsize}{36}
\newlength{\sellsz} 
\newsavebox{\sell}
\sbox{\sell}{\begin{picture}(\sellsize,20)
\put(0,0){\line(1,0){\sellsize}}
\put(0,0){\line(0,1){\sellsize}}
\put(\sellsize,0){\line(0,1){\sellsize}}
\put(0,\sellsize){\line(1,0){\sellsize}}
\end{picture}}
\newcommand\sellify[1]{\def\thearg{#1}\def\nothing{}%
\ifx\thearg\nothing
\vrule width0pt height\sellsz depth0pt\else
\hbox to 0pt{\usebox{\sell} \hss}\fi%
\vbox to \sellsz{
\vss
\hbox to \sellsz{\hss$#1$\hss}
\vss}}
\newcommand\stableau[1]{\vtop{\let\\\cr
\baselineskip -16000pt \lineskiplimit 16000pt \lineskip 0pt
\ialign{&\sellify{##}\cr#1\crcr}}}

\newcommand{\smellsize}{7}
\newlength{\smellsz} 
\newsavebox{\smell}
\sbox{\smell}{\begin{picture}(\smellsize,\smellsize)
\put(0,0){\line(1,0){\smellsize}}
\put(0,0){\line(0,1){\smellsize}}
\put(\smellsize,0){\line(0,1){\smellsize}}
\put(0,\smellsize){\line(1,0){\smellsize}}
\end{picture}}
\newcommand\smellify[1]{\def\thearg{#1}\def\nothing{}%
\ifx\thearg\nothing
\vrule width0pt height\smellsz depth0pt\else
\hbox to 0pt{\usebox{\smell} \hss}\fi%
\vbox to \smellsz{
\vss
\hbox to \smellsz{\hss$#1$\hss}
\vss}}
\newcommand\smtableau[1]{\vtop{\let\\\cr
\baselineskip -16000pt \lineskiplimit 16000pt \lineskip 0pt
\ialign{&\smellify{##}\cr#1\crcr}}}

\numberwithin{equation}{chapter}
\makeindex

\begin{document}
\frontmatter
\begin{titlepage}

\begin{center}
\vspace{1cm}
{\huge COMPATIBLY SPLIT SUBVARIETIES OF THE HILBERT SCHEME OF POINTS IN THE PLANE}
\vspace{5cm}

A Dissertation\\
Presented to the Faculty of the Graduate School\\
of Cornell University\\
in Partial Fulfillment of the Requirements for the Degree of \\
Doctor of Philosophy\\
\vspace{5cm}

by\\
Jenna Beth Rajchgot\\
January 2013
\end{center}
\end{titlepage}









\newpage
\thispagestyle{empty}
\vspace*{\fill}
\vspace{-2cm}
\begin{center}
{\large $\copyright$ 2013 Jenna Beth Rajchgot\\ ALL RIGHTS RESERVED}
\end{center}
\vspace*{\fill}
\newpage


\newpage
\thispagestyle{empty}
\addtocounter{page}{-1}
\vspace{0.1in}
\begin{center}
COMPATIBLY SPLIT SUBVARIETIES OF THE HILBERT SCHEME OF POINTS IN THE PLANE\\
Jenna Beth Rajchgot, Ph.D.\\
Cornell University 2013
\end{center}
Let $k$ be an algebraically closed field of characteristic $p>2$. By a result of Kumar and Thomsen (see \cite{KT}), the standard Frobenius splitting of $\bA^2_k$ induces a Frobenius splitting of $\Hilbn$. In this thesis, we investigate the question, ``what is the stratification of $\Hilbn$ by all compatibly Frobenius split subvarieties?'' 

We provide the answer to this question when $n\leq 4$ and give a conjectural answer when $n=5$. We prove that this conjectural answer is correct up to the possible inclusion of one particular one-dimensional subvariety of $\Hilb^5(\bA^2_k)$, and we show that this particular one-dimensional subvariety is not compatibly split for at least those primes $p$ satisfying $2 <p \leq 23$. 

Next, we restrict the splitting of $\Hilbn$ (now for arbitrary $n$) to the affine open patch $U_{\langle x,y^n\rangle}$ and describe all compatibly split subvarieties of this patch and their defining ideals. We find degenerations of these subvarieties to Stanley-Reisner schemes, explicitly describe the associated simplicial complexes, and use these complexes to prove that certain compatibly split subvarieties of $U_{\langle x,y^n\rangle}$ are Cohen-Macaulay.



\newpage
\begin{center}
BIOGRAPHICAL SKETCH
\end{center}
Jenna Rajchgot was raised in Thornhill, Ontario, Canada, a suburb of Toronto. She graduated from Thornhill Secondary School in 2003 and went on to study mathematics at Queen's University in Kingston, Ontario. She graduated from Queen's with a B.Sc.H. degree in the spring of 2007. The following August, she entered the math Ph.D. program at Cornell University. She completed her dissertation in August 2012 under the supervision of Professor Allen Knutson.

\addcontentsline{toc}{section}{{\bf Biographical sketch}}

\newpage
\begin{center}
ACKNOWLEDGEMENTS
\end{center}
This thesis could not have been written without the constant guidance of my advisor, Allen Knutson.  Allen helped me every step of the way; he provided the motivating problem, patiently explained relevant background material, and offered numeous helpful suggestions at our frequent meetings. I learned more mathematics from Allen than from any other source and I have a greater appreciation of the subject thanks to him. I feel very fortunate to have worked under Allen's supervision and I'm pleased to thank him here.

I would also like to thank my other committee members, Tara Holm and Mike Stillman, for helpful discussions, for interesting courses and seminars, and for commenting on previous drafts of this thesis. 

I am grateful to David Speyer for graciously allowing an unpublished theorem and proof of his to be included in this thesis. I also wish to thank Allen Knutson, Thomas Lam, and David Speyer for permitting me to include an unpublished version of their algorithm for finding compatibly split subvarieties.

Over my five years at Cornell, I had the pleasure of working with a number of students and postdocs in the mathematics department. I would especially like to acknowledge Mathias Lederer as I had many enjoyable and enlightening conversations with him. In addition, Mathias kindly taught me to use a program that he wrote for computing on open patches of Hilbert schemes of points. This program was very useful when performing computations relevant to this thesis.

On a personal note, I would like to thank my friends, especially Amy, Joeun, Kathryn, Michelle, and Voula, for providing me a great deal of encouragement while I was writing this thesis, and for making my five years in Ithaca very enjoyable. 

Finally, I would like to thank my parents Percy and Rochelle, my brother Jason, and my partner Cameron Franc for all of their love and support.  
\\

I was supported by the Cornell mathematics department and by the Natural Sciences and Engineering Research Council of Canada (NSERC).

\addcontentsline{toc}{section}{{\bf Acknowledgements}}

\setcounter{tocdepth}{2}
\renewcommand{\cfttoctitlefont}{\hfill\Large\bfseries}
\renewcommand{\cftaftertoctitle}{\hfill}
\renewcommand{\contentsname}{TABLE OF CONTENTS}
\tableofcontents


\newpage
\begin{center}
INTRODUCTION
\end{center}
Let $k$ be an algebraically closed field of characteristic $p>2$, and let $X$ be a quasiprojective smooth surface defined over $k$. The Hilbert scheme of $n$ points on $X$, denoted $\Hilb^n(X)$, is the $2n$-dimensional smooth scheme (see \cite{Fog}) which parametrizes all dimension-$0$, degree-$n$ subschemes of $X$. If $X$ is Frobenius split, then $\Hilb^n(X)$ is Frobenius split (see \cite{KT}). More specifically, if $X$ is Frobenius split compatibly with the anticanonical divisor $D$, then $\Hilb^n(X)$ is Frobenius split compatibly with the anticanonical divisor described set-theoretically by \[\{q\in\Hilb^n(X)~|~\textrm{Supp}(q)\cap D\neq\emptyset\}.\] 

In this thesis, we are interested in the case $X = \bA_k^2$ and $D = \{xy=0\}$. Our motivating problem is to understand the stratification of $\Hilbn$ by all of its finitely many (see \cite{KM}, \cite{Sch}) compatibly split subvarieties (with respect to this particular splitting).

\begin{figure}[h]
\begin{center}
\includegraphics[scale = 0.32]{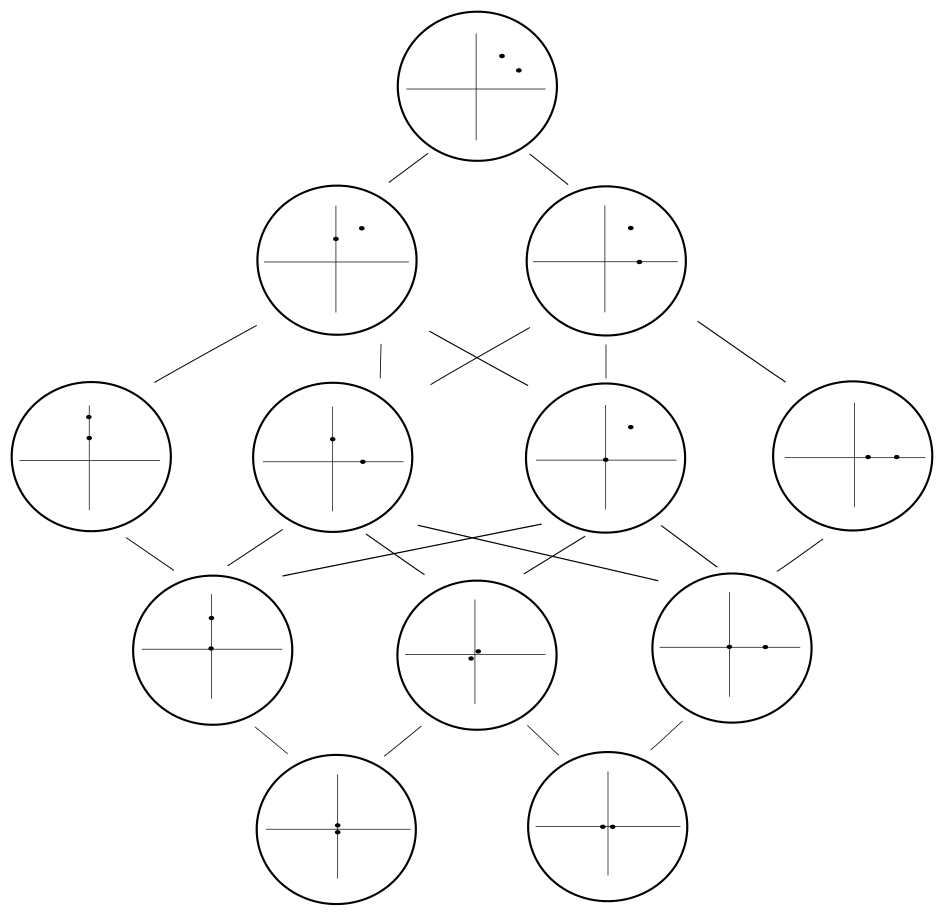}
\label{fig;2ptsposetintro}
\caption{The compatibly split subvarieties of $\Hilb^2(\bA^2_k)$}
\end{center}
\end{figure}

To gain some intuition, we consider the $n=2$ case. The stratification of $\Hilb^2(\bA^2_k)$ by all compatibly split subvarieties is shown in Figure 1. We have drawn a ``stratum representative'' (see Subsection \ref{s;stratrep}) to correspond to an entire compatibly split subvariety. For example, the picture where the two points are at the origin and on a diagonal line represents the punctual Hilbert scheme, denoted $\Hilb^2_0(\bA^2_k)$, which parametrizes dimension-$0$, length-$2$ subschemes of $\bA^2_k$ that are supported at the origin. Note that $\Hilb^2_0(\bA^2_k)\cong \bP^1_k$.

To gain further intuition into our problem, we remark (see Proposition \ref{p;intuition}) that finding all compatibly split subvarieties of $\Hilbn$ is equivalent to finding all compatibly split subvarieties $Z\subseteq \Hilb^m(\bA^2_k)$ for $m\leq n$, where $Z$ is a subvariety the punctual Hilbert scheme $\Hilb^m_0(\bA^2_k)$. More precisely, for sufficiently large primes $p$, $Y\subseteq \Hilbn$ is a compatibly split subvariety if and only if $Y$ is the closure in $\Hilbn$ of the (set-theoretic) image of a map
\[i:\Hilb^r(\textrm{``punctured y-axis''})\times \Hilb^s(\textrm{``punctured x-axis''})\times \Hilb^t(\mathbb{A}^2_k\setminus \{xy=0\})\times Z\]
\[\rightarrow \Hilbn\]
\[(I_1,I_2,I_3,I_4)\mapsto I_1\cap I_2\cap I_3\cap I_4\] for some $r,s,t\geq 0$ with $r+s+t\leq n$ and for some compatibly split $Z\subseteq \Hilb^{n-r-s-t}(\mathbb{A}^2_k)$ with the property that $Z\subseteq \Hilb^{n-r-s-t}_0(\mathbb{A}^2_k)$. Of course, in order to understand all compatibly split subvarieties of $\Hilbn$, we must still compute all of these $Z\subseteq \Hilb_0^m(\bA^2_k)$, $m\leq n$. 

One method for finding all compatibly split subvarieties for small $n$ is to utilize an algorithm of Allen Knutson, Thomas Lam, and David Speyer. We explain this algorithm in the Frobenius splitting background section and use it in Subsection \ref{s;algorithm2pts} to produce the stratification of $\Hilb^2(\bA^2_k)$ pictured above. In Subsection \ref{s;nleq5}, we use other tools (primarily the moment polyhedron of $\Hilbn$) to produce all compatibly split subvarieties of $\Hilbn$ for $n\leq 4$, as well as provide a conjectural list of all compatibly split subvarieties of $\Hilb^5(\bA^2_k)$. We prove that this conjectural list is correct up to the possible inclusion of one particular one-dimensional subvariety of $\Hilb^5(\bA^2_k)$, and we show that this particular one-dimensional subvariety is not compatibly split for at least those primes $p$ satisfying $2 <p \leq 23$.

After the discussion of the $n=5$ case, we restict our attention to the open affine patch \[U_{\langle x,y^n\rangle} = \{I\in \Hilbn~|~\{1,y,\dots,y^{n-1}\} \textrm{ is a }k \textrm{-vector space basis of }k[x,y]/I\}.\] This patch has an induced Frobenius splitting and so we may ask the question, ``what are all of the compatibly split subvarieties of $U_{\langle x,y^n\rangle}$ with the induced splitting?'' In Section \ref{s;goodPatch}, we answer this question for general $n$. We describe the defining ideals of these subvarieties and we find degenerations to Stanley-Reisner schemes. We provide an explicit description of the simplicial complexes associated to these Stanley-Reisner schemes and use these complexes to prove that certain compatibly split subvarieties of $U_{\langle x,y^n\rangle}$ are Cohen-Macaulay.


This thesis has two chapters. Chapter 1 discusses relevant background material, an unpublished result of David E Speyer, and the unpublished Knutson-Lam-Speyer algorithm for finding compatibly split subvarieties. Chapter 2 addresses the material described above. 

Macaulay 2 \cite{M2} was used to perform numerous computations leading to the results appearing in Chapter 2 of this thesis. Sage \cite{sage} (with the FLINT library) was used to perform some of the computations appearing in the discussion of the compatibly split subvarieties of $\Hilb^5(\bA^2_k)$.

\addcontentsline{toc}{section}{{\bf Introduction}}

\mainmatter

\pagestyle{fancy}
\chapter{Background}

\section{The Hilbert scheme of points in the plane}

\subsection{Basic notions}
Material in this subsection is taken primarily from \cite[Chapter 1]{Nak}.

Let $k$ be an algebraically closed field. The \textbf{Hilbert scheme of} $\mathbf{n}$ \textbf{points in the affine plane}, $\Hilbn$, is the scheme parametrizing all zero-dimensional, length-$n$ subschemes of $\bA^2_k$. More precisely, $\Hilbn$ represents the contravariant functor
\begin{center}
$\mathcal{H}ilb^n_{\bA_k^2}:\textrm{Schemes}\rightarrow \textrm{Sets}$
\[\mathcal{H}ilb^n_{\bA_k^2}(U) =\left\{ Z\subseteq U\times \bA_k^2~\begin{array}{|c}$Z$~\textrm{is a closed subscheme and}\\ \pi:Z\rightarrow U~\textrm{is flat with dimension-}0\textrm{,}~\textrm{length-}n~\textrm{fibers}\\\end{array}\right\}.\]
\end{center}

\begin{example}
\begin{enumerate}
\item The Hilbert scheme of $1$ point in the plane is isomorphic to $\bA^2_k$. 
\item A zero-dimensional, length-$2$ subscheme of $\bA^2_k$ is determined by either (i) two distinct \emph{unordered} locations $(x_1,y_1), (x_2,y_2)\in \bA^2_k$, or (ii) one location $(x,y)\in \bA^2_k$ along with a projective vector $v\in \bP^1_k$. Indeed, $\Hilb^2(\bA^2_k)\cong Bl_\Delta({\bA^2_k\times \bA^2_k})/S_2$ where $Bl_\Delta({\bA^2_k\times \bA^2_k})$ denotes the blow-up of $\bA^2_k\times \bA^2_k$ along the diagonal locus $\Delta := \{((x,y),(x,y))~|~(x,y)\in \bA^2_k\}$.
\end{enumerate}
\end{example}

Consider the action of the symmetric group $S_n$ on $(\bA^2_k)^n$ by permuting the various copies of $\bA^2$. Taking the quotient of $(\bA^2_k)^n$ by $S_n$ yields the \textbf{Chow variety}, the scheme of $n$ unlabelled locations in the plane. This scheme is often denoted by $S^n(\bA^2_k)$. An element of $S^n(\bA^2_k)$ is written as a sum \[\sum_{(x,y)\in\bA^2_k}a_{(x,y)}[(x,y)]\] where the coefficients $\{a_{(x,y)}~|~(x,y)\in\bA^2_k\}$ are non-negative integers that sum to $n$.

\begin{definition}\label{d;hilbchow}
The \textbf{Hilbert-Chow morphism} \[\Psi:\Hilbn\rightarrow S^n(\bA_k^2)\] is the map which takes each zero-dimensional, length-$n$ subscheme of $\bA^2_k$ to its support with multiplicities. That is, \[\Psi: Z\mapsto \sum_{x\in \bA^2_k}\textrm{length}(Z_x)[x].\]
\end{definition}

Notice that $S_n$ acts freely on the open set $V\subseteq (\bA^2_k)^n$ where all of the $n$ locations are distinct. Thus, the Chow variety is smooth along $V/S_n$. $S^n(\bA^2_k)$ is singular along the closed subvariety where at least two points collide. 

\begin{theorem}\cite{Fog}
\begin{enumerate}
\item $\Psi:\Hilbn\rightarrow (\bA_k^2)^n/S_n$ defined in Definition \ref{d;hilbchow} is a birational, projective morphism of schemes. It is an isomorphism over the open set where all points are distinct. 
\item $\Hilbn$ is connected and smooth. Thus, the Hilbert-Chow morphism is a resolution of singularities.
\end{enumerate}
\end{theorem}

As a corollary, $\Hilbn$ is irreducible. In addition, because the Chow variety is $2n$-dimensional and the Hilbert-Chow morphism is birational, $\Hilbn$ is also $2n$-dimensional.

\subsection{An open cover}\label{s;opencover}

In this subsection, we describe an open cover of $\Hilbn$ indexed by the colength-$n$ monomial ideals. Our main references for this material are \cite{H2}, \cite[Chapter 18]{MS}, and \cite{Led}.

\begin{definition}
Let $\lambda$ be a colength-$n$ monomial ideal. Define $U_\lambda\subseteq \Hilbn$ to be the set of ideals $I$ such that the monomials outside of $\lambda$ form a $k$-vector space basis for $k[x,y]/I$.
\end{definition}

Notice that the set of all $U_{\lambda}$ cover $\Hilbn$. Indeed, if $I\in \Hilbn$ and $>$ is a term order, the initial ideal $\textrm{init}_{>}(I)$ (see Subsection \ref{s;gb} for relevant definitions) is a colength-$n$ monomial ideal and the $n$ monomials outside of $\textrm{init}_{>}(I)$ form a $k$-vector space basis for $k[x,y]/I$. Thus, $I\in U_{\textrm{init}_{>}(I)}$.

Any ideal $I\in U_{\lambda}$ has a unique set of generators of the form \[\left\{x^ry^s-\sum_{x^hy^l\notin\lambda}c^{r,s}_{h,l}x^hy^l~|~x^ry^s\in \lambda, c^{r,s}_{h,l}\in k\right\}.\] Thus, the $c^{r,s}_{h,l}$ are coordinates on $U_\lambda$. Furthermore, if $g\in I$, then $x\cdot g\in I$ and $y\cdot g\in I$. These conditions yield relations in the various $c^{r,s}_{h,l}$. To be precise, suppose that $g = x^ry^s-\sum_{x^hy^l\notin\lambda}c^{r,s}_{h,l}x^hy^l$ is an element of $I$. Then, \[x\cdot g = x^{r+1}y^s - \sum_{x^hy^l\notin\lambda}c^{r,s}_{h,l}x^{h+1}y^l~~~(*)\] is also in $I$. Notice that some of the $x^{h+1}y^l$ appearing in $(*)$ are in $\lambda$. We can therefore replace these $x^{h+1}y^l\in\lambda$ using $x^{h+1}y^l-\sum_{x^{h'}y^{l'}}c^{h+1,l}_{h',l'}x^{h'}y^{l'}\in I$ to see that \[x^{r+1}y^s-\left(\sum_{x^{h+1}y^l\notin \lambda}c^{r,s}_{h,l}x^{h+1}y^l+\sum_{x^{h+1}y^l\in \lambda}c^{r,s}_{h,l}\sum_{x^{h'}y^{l'}}c^{h+1,l}_{h',l'}x^{h'}y^{l'}\right)~~~(**)\] is an element of $I$. On the other hand, \[x^{r+1}y^s-\sum_{x^hy^l\notin\lambda}c^{r+1,s}_{h,l}x^hy^l~~~(***)\] is also in $I$. Equating coefficients in $(**)$ with those in $(***)$ yield relations amongst various $c^{r,s}_{h,l}$. Let $J_{\lambda}$ denote the ideal generated by these relations, along with the relations obtained by interchanging the roles of $x$ and $y$ in the above computation.

\begin{proposition}\label{p;hilbbasic}
The sets $U_{\lambda}$ are open affine subvarieties that cover $\Hilbn$. The affine coordinate ring is given by $k[\{c^{r,s}_{h,l}~|~x^ry^s\in \lambda,~x^hy^l\notin \lambda\}]/J_\lambda$.
\end{proposition}

\begin{example}
Let $\lambda = \langle x,y^n\rangle$. Then $U_{\langle x,y^n\rangle}$ consists of all $I\in \Hilbn$ such that $\{1,y,\dots,y^{n-1}\}$ is a $k$-vector space basis for $k[x,y]/I$. In addition, each $I\in U_{\langle x,y^n\rangle}$ is generated by the set of polynomials \[\{x^ry^s- c^{r,s}_{0,n-1}y^{n-1}-\cdots - c^{r,s}_{0,1}y-c^{r,s}_{0,0}~|~x^ry^s\in \langle x,y^n\rangle\}.\] Notice that not all of these generators are necessary. In fact, each $I\in U_{\langle x,y^n\rangle}$ is minimally generated by a set of the form \[G = \{y^n-c^{0,n}_{0,n-1}y^{n-1}-\cdots - c^{0,n}_{0,1}y-c^{0,n}_{0,0},~ x-c^{1,0}_{0,n-1}y^{n-1}-\cdots- c^{1,0}_{0,1}y-c^{1,0}_{0,0}\}.\] Indeed, with respect to the Lex term order with $x\gg y$, $G$ is a Gr\"{o}bner basis with initial ideal $\langle x,y^n\rangle$. Thus, $I\in U_{\langle x,y^n\rangle}$ if and only if $\textrm{init}(I) = \langle x,y^n\rangle$.
\end{example}

The number of generators of the coordinate ring of $U_{\lambda}$ can be reduced significantly from that in Proposition \ref{p;hilbbasic}.

\begin{definition}
Let $\lambda$ be a colength-$n$ monomial ideal. The \textbf{standard set of} $\mathbf{\lambda}$, which we denote by $S_{\lambda}$, is the set of all $(i,j)\in \bN^2$ such that $x^iy^j$ is a monomial outside of $\lambda$. The \textbf{border of} $\mathbf{\lambda}$, which we denote by $B_{\lambda}$, is the set of all $(i,j)\in \bN^2\setminus S_{\lambda}$ such that $(i-1,j)\in S_{\lambda}$ or $(i,j-1)\in S_\lambda$.
\end{definition}

\begin{example}
Let $\lambda = \langle x^2, xy^2, y^3\rangle\in \Hilb^5(\bA^2_k)$. The monomials outside of $\lambda$ are $\{1,x,y,xy,y^2\}$. Therefore, $S_\lambda = \{(0,0), (1,0), (0,1), (1,1), (0,2)\}$ and $B_{\lambda} = \{(2,0), (2,1), (1,2), (0,3)\}$.
\end{example}

Often, we draw the Young diagram (using French notation) associated to the monomial ideal $\lambda$ instead of writing down the standard set. Notice that the Young diagram and the standard set give equivalent information. 

\begin{theorem}\label{p;justInsidejustOutside}
Let $\lambda$ be a colength-$n$ monomial ideal. Let $S_{\lambda}$ denote its standard set and let $B_{\lambda}$ denote its border. The coordinate ring of $U_{\lambda}$ is generated by those $c^{rs}_{hl}$ such that $(r,s)\in B_{\lambda}$ and $(h,l)\in S_{\lambda}$.
\end{theorem}

Before concluding this subsection, we note that for many colength-$n$ monomial ideals $\lambda$, $U_{\lambda}$ is isomorphic to $\bA^{2n}_k$. For example, when $n\leq 4$, $U_{\lambda}$ is always isomorphic to $\bA^{2n}_k$. 

\subsection{A torus action}

The torus $T^2 = \bG_m^2$ acts algebraically on $\bA^2_k$ by $(t_1,t_2)\cdot(x,y) = (t_1x,t_2y)$. This induces an action on $\Hilbn$. That is, if $I\in \Hilbn$, then \[(t_1,t_2)\cdot I = \langle f(t_1^{-1}x,t_2^{-1}y)~|~f(x,y)\in I \rangle.\] Notice that the $T^2$-fixed points of $\Hilbn$ are the colength-$n$ monomial ideals. 

\begin{lemma}
Each $U_{\lambda}$ is invariant under the $T^2$-action. As an element of the coordinate ring, $c^{r,s}_{h,l}$ has $T^2$-weight $(h-r,l-s)$. 
\end{lemma}

\begin{proof}
Suppose $I\in U_{\lambda}$ and $x^ry^s-\sum_{x^hy^l\notin\lambda}c^{r,s}_{h,l}x^hy^l$ is a generator of $I$. Then, \[\begin{array}{ccc}(t_1,t_2)\cdot (x^ry^s-\sum_{x^hy^l\notin\lambda}c^{r,s}_{h,l}x^hy^l)& = &t_1^{-r}t_2^{-s}x^ry^s - \sum_{x^hy^l\notin\lambda}c^{r,s}_{h,l}t_1^{-h}t_2^{-l}x^hy^l\\ & =& t_1^{-r}t_2^{-s}(x^ry^s - \sum_{x^hy^l\notin\lambda} t_1^{r-h}t_2^{s-l}c^{r,s}_{h,l}x^hy^l)\end{array}\] So, as an element of the coordinate ring of $U_{\lambda}$, $c^{r,s}_{h,l}$ has weight $(h-r,l-s)$.

$U_{\lambda} = \Spec (k[\{c^{r,s}_{h,l}\}]/J_{\lambda})$ is $T^2$-invariant because $J_{\lambda}$ is $T^2$-homogeneous by construction. 
\end{proof}

We can use the torus action to prove that $\Hilbn$ is non-singular. To begin, notice that every ideal $I\in\Hilbn$ has a $T^2$-fixed point in the closure of its orbit. (Indeed, with respect to any monomial order, $\textrm{init}(I)\in \overline{T^2\cdot I}$. As $\textrm{init}(I)$ is a colength-$n$ monomial ideal, it is a $T^2$-fixed point.) Therefore, if the Hilbert scheme is singular anywhere, it must also be singular at one of the $T^2$-fixed points. We will now show that $\Hilbn$ is non-singular at each $T^2$-fixed point.

\begin{definition}
Let $\lambda$ be a colength-$n$ monomial ideal and let $S_{\lambda}$ denote its standard set. For each $x\in S_{\lambda}$, we define the \textbf{arm}, $a(x)$, and \textbf{leg}, $l(x)$, as follows: For $x = (i,j)\in S_{\lambda}$, \[a(x):=\{\#(i,r)\in S_{\lambda}~|~r>j\},~~~~~l(x):=\{\#(r,j)\in S_{\lambda}~|~r>i\}.\]  
\end{definition}

\begin{proposition}
Let $\lambda$ be a $T^2$-fixed point of $\Hilbn$. The tangent space $T_{\lambda}\Hilbn$ is a $2n$-dimensional $T^2$-representation with $T^2$-weights \[\{(-l(x),a(x)+1)~|~x\in S_{\lambda}\}\cup \{(l(x)+1,-a(x))~|~x\in S_{\lambda}\}.\]
\end{proposition}

We sketch Haiman's proof (see \cite{H2}) of this result.

\begin{proof}[Proof (sketch)]
We show that the Zariski cotangent space $\frak{m}_{\lambda}/\frak{m}_{\lambda}^2$ has dimension at most $2n$ using the following procedure.
\begin{enumerate}
\item Work in the open patch $U_{\lambda}$ and note that $\frak{m}_{\lambda}$ is generated by all of the $c^{r,s}_{h,l}$ in the coordinate ring of $U_{\lambda}$.
\item Let $S_{\lambda}$ denote the standard set of $\lambda$. Identify each $c^{r,s}_{h,l}$ with the vector from $(r,s)$ to $(h,l)$. Notice that the position of the tail of this vector is an element of $\bN^2\setminus S_{\lambda}$ and that the position of the head of this vector is an element of $S_{\lambda}$. 
\item Say that one arrow can be \emph{translated} to another if there exists a sequence of horizontal and vertical shifts moving the first arrow to the second such that the position of the tail always remains an element of $\bN^2\setminus S_{\lambda}$ and the position of the head is always of the form $(a-c,b-d)$ where $(a,b)\in S_{\lambda}$ and $(c,d)\in \bZ_{\geq 0}^2$. (That is, it is permissible for the head of the arrow to lie below the $x$-axis or to the left of the $y$-axis.) Show that $c^{r_1,s_1}_{h_1,l_1}\equiv c^{r_2,s_2}_{h_2,l_2}~(\textrm{mod }\frak{m}_{\lambda}^2)$ if the arrow associated to $c^{r_1,s_1}_{h_1,l_1}$ can be translated to the arrow associated to $c^{r_2,s_2}_{h_2,l_2}$. Show that $c^{r,s}_{h,k}\equiv 0~(\textrm{mod }\frak{m}_{\lambda}^2)$ if $c^{r,s}_{h,k}$ can be translated such that the head of the arrow lies below the $x$-axis or to the left of the $y$-axis.
\item Notice that every arrow can be translated such that one of the following occurs: (i) the head crosses an axis, or (ii) there is an $x = (i,j)\in S_{\lambda}$ such that the head of the arrow is at $(i+l(x),j)$ and the tail is at $(i,j+a(x)+1)$, or (iii) there is an $x = (i,j)\in S_{\lambda}$ such that the head of the arrow is at $(i,j+a(x))$ and the tail is at $(i+l(x)+1,j)$. 
\end{enumerate}
This procedure yields $2n$ (i.e. $2$ for each $x\in S_{\lambda}$) elements which span $\frak{m}_{\lambda}/\frak{m}_{\lambda}^2$. Thus, the cotangent space is at most $2n$-dimensional. Because each $T^2$-fixed point lies in the closure of the smooth, $2n$-dimensional locus where all points are distinct, the tangent space at each $T^2$-fixed point is also at least $2n$-dimensional. Furthermore, we can see from the generators constructed above that the Zariski cotangent space has $T^2$-weights  \[\{-(-l(x),a(x)+1)~|~x\in S_{\lambda}\}\cup \{-(l(x)+1,-a(x))~|~x\in S_{\lambda}\}.\]
\end{proof}

Therefore, $\Hilbn$ is non-singular at each of the $T^2$-fixed points. It follows that $\Hilbn$ is non-singular everywhere.


\section{Frobenius splitting}

In this section, our main source is the textbook \cite{BK}. As in the text, we let $k$ be an algebraically closed field of characteristic $p>0$ and let $(X,\mathcal{O}_X)$ be a separated scheme of finite type over $k$. 

\subsection{Basic notions}

Recall that the \textbf{absolute Frobenius morphism}, $F:X\rightarrow X$, is defined to be the identity map on points and the $p^{th}$ power map on functions. Then, $F^{\sharp}:\mathcal{O}_X\rightarrow F_*\mathcal{O}_X$ is an $\mathcal{O}_X$-linear map where the module structure on $F_* \mathcal{O}_X$ is given by $a\cdot b =a^pb$ for any local sections a,b of $\mathcal{O}_X$.

\begin{definition}
A scheme $(X,\mathcal{O}_X)$ is \textbf{Frobenius split} by the $\mathcal{O}_X$-linear map $\phi: F_* \mathcal{O}_X\rightarrow \mathcal{O}_X$ if $\phi\circ F^{\sharp} = 1_{\mathcal{O}_X}$. The map $\phi$ is called a \textbf{splitting}.
\end{definition}

This definition of ``Frobenius split'' is a little bit different than \cite[Definition 1.1.3]{BK}. In the book, a scheme $X$ is called Frobenius split if \emph{there exists} a map $\phi: F_*\mathcal{O}_X\rightarrow \mathcal{O}_X$ such that $\phi\circ F^{\sharp} = 1_{\mathcal{O}_X}$. We call such schemes \textbf{Frobenius splittable}. We reserve the phrase ``$X$ is Frobenius split'' for when $X$ comes with a fixed splitting.

The next proposition follows immediately.

\begin{proposition}\cite[Proposition 1.2.1]{BK}
If $X$ is Frobenius split then $X$ is reduced.
\end{proposition}

\begin{proof}
We proceed by contradiction. Suppose $X$ is Frobenius split by $\phi: \mathcal{O}_X\rightarrow F_*\mathcal{O}_X$. Let $U\subseteq X$ be an affine open subscheme and suppose that $\mathcal{O}_X(U)$ has a nilpotent element $a$. Let $n>1$ be the integer such that $a^n=0$ but $a^{n-1}\neq 0$. Then $a^{n-1} = \phi((a^{n-1})^p) = \phi(0) = 0$, a contradiction.
\end{proof}

\begin{definition}
Let $X$ by Frobenius split by $\phi: F_*\mathcal{O}_X\rightarrow \mathcal{O}_X$. A closed subscheme $Y\subseteq X$ is \textbf{compatibly split} if $\phi(F_*\mathcal{I}_Y)\subseteq \mathcal{I}_Y$.
\end{definition}

Again, this definition differs slightly from \cite[Definition 1.1.3]{BK} because we specify that $X$ comes with a fixed splitting.

\begin{proposition}\cite[Proposition 1.2.1]{BK}
Let $X$ be Frobenius split by $\phi: F_*\mathcal{O}_X\rightarrow \mathcal{O}_X$. Then intersections, unions, and components of compatibly split subschemes are compatibly split.
\end{proposition}

\begin{proof}
Let $Y_1$ and $Y_2$ be compatibly split subschemes of $X$ with ideal sheaves $\mathcal{I}_{Y_1}$ and $\mathcal{I}_{Y_2}$. Then, \[\phi(F_*\mathcal{I}_{Y_1\cap Y_2}) = \phi(F_*\mathcal{I}_{Y_1})+\phi(F_*\mathcal{I}_{Y_2}) \subseteq \mathcal{I}_{Y_1}+\mathcal{I}_{Y_2},~\textrm{and}\] \[\phi(F_*\mathcal{I}_{Y_1\cup Y_2}) = \phi(F_*\mathcal{I}_{Y_1})\cap\phi(F_*\mathcal{I}_{Y_2})\subseteq \mathcal{I}_{Y_1}\cap\mathcal{I}_{Y_2}. \] Thus, intersections and unions of compatibly split subschemes are compatibly split.

Next, suppose that $Y_1 = D\cup E$ where $D$ is an irreducible component of $Y_1$. To show that $D$ is compatibly split, we show that $\phi(\mathcal{I}_{Y_1}(U):\mathcal{I}_{E}(U))\subseteq \mathcal{I}_{Y_1}(U):\mathcal{I}_{E}(U)$ for any affine open subscheme $U$. Let $a\in \mathcal{I}_{Y_1}(U):\mathcal{I}_{E}(U)$ and let $b\in \mathcal{I}_{E}(U)$. Then \[b\phi(a) = \phi(b^pa) \in \phi(\mathcal{I}_{Y_1}(U))\subseteq \mathcal{I}_{Y_1}(U)\] since $Y_1$ is compatibly split.
\end{proof}

Thus, given a Frobenius split scheme $X$ and a compatibly split subscheme $D$, we can intersect the components of $D$, decompose the intersection, intersect the new components, and so on to create a list of many compatibly split subvarieties of $X$. This process is the core of the Knutson-Lam-Speyer algorithm for finding compatibly split subvarieties of a Frobenius split scheme $(X,\phi)$ (see Subsection \ref{s;kls}). The algorithm also makes use of the following proposition.

\begin{proposition}\label{p;normalizationSplit}\cite[similar to Exercise 1.2.E (4)]{BK}
Let $X$ be irreducible and Frobenius split by $\phi: F_*\mathcal{O}_X\rightarrow \mathcal{O}_X$. 
\begin{enumerate}
\item Let $X_{\textrm{non-R1}}$ denote the codimension-$1$ component of the singular locus of $X$, with its reduced subscheme structure. Then, $X_{\textrm{non-R1}}$ is compatibly split.
\item Let $\nu:\widetilde{X}\rightarrow X$ be the normalization of $X$. There is an induced Frobenius splitting of $\widetilde{X}$ that compatibly splits $\nu^{-1}(X_{\textrm{non-R1}})$.
\end{enumerate}
\end{proposition}

\begin{proof}
We follow the hint provided in the textbook.

It suffices to show that the proposition holds for any affine open subscheme $U\subseteq X$. So, suppose that $U = \Spec A$ and that $\nu: \Spec B\rightarrow \Spec A$ is the normalization of $\Spec A$. Let $E\subseteq \Spec A$ denote the set over which $\nu$ fails to be an isomorphism and let $I\subseteq A$ be the conductor ideal, \[I:=\{a\in A~|~aB\subseteq A\}.\] Then $E = \Spec (A/I)$. Furthermore, $I$ is compatibly split. Indeed, let $a\in I$ and $b\in B$. Then \[\phi(a)b = \phi(ab^p) \in \phi(A)\subseteq A.\] 

Since $X_{\textrm{non-R1}}\cap U$ is a union of some of the components of $E$, $X_{\textrm{non-R1}}\cap U$ is compatibly split. This proves 1.

Now $\phi$ extends to a splitting $\tilde{\phi}$ of $\textrm{Frac}(A)$ by setting $\tilde{\phi}(a_1/a_2) := \frac{\phi(a_2^{p-1}a_1)}{a_2}$. Thus, $\phi$ extends to a splitting of $B$ so long as each $\tilde{\phi}(b)$, $b\in B$, is integral over $A$. 

To do this, we first show that $\tilde{\phi}(b)I\subseteq I$. So, suppose $a\in I$ and $b, b'\in B$. Then, \[(\tilde{\phi}(b)a)b' = \tilde{\phi}(b(ab')^p) = \tilde{\phi}(a^p(b(b')^p))\in \tilde{\phi}(A) = \phi(A)\subseteq A\]
since $a^p\in I$ and $b(b')^p\in B$. Thus, $\tilde{\phi}(b)I\subseteq I$. Then for any $n\geq 0$, $\tilde{\phi}(b)^nI\subseteq I$ (by induction). It follows that $\tilde{\phi}(b)$ is integral over $A$. Indeed, if $a\in I$ then $\tilde{\phi}(b)a, \tilde{\phi}(b)^2a,\dots$ are all in $I$. Because $A$ is Noetherian, there is some $N$ such that $\tilde{\phi}(b)^{N+1}a = c_1\tilde{\phi}(b)a+\cdots+c_N\tilde{\phi}(b)^Na$, $c_i\in A$. Since $A$ is a domain, $\tilde{\phi}(b)^{N+1} - (c_1\tilde{\phi}(b)+\cdots+c_N\tilde{\phi}(b)^N)=0$. Thus, $\tilde{\phi}(b)\in B$ and $\phi$ extends to a splitting of $B$.

Finally, $\nu^{-1}(E)$ is compatibly split since $I$ is also the ideal of $\nu^{-1}(E)$. As $\nu^{-1}(X_{\textrm{non-R1}}\cap U)$ is a union of some of the components of $\nu^{-1}(E)$, $\nu^{-1}(X_{\textrm{non-R1}}\cap U)$ is compatibly split in $\Spec B$.
\end{proof}

We end this subsection with a proposition regarding open subschemes of a Frobenius split scheme.

\begin{proposition}\label{p;openSplit} \cite[Lemma 1.1.7]{BK}
Suppose $X$ is Frobenius split by $\phi: F_*\mathcal{O}_X\rightarrow \mathcal{O}_X$, and $V\subseteq X$ is an open subscheme. Then $V$ has an induced Frobenius splitting. Furthermore, if $Y$ is a compatibly split subvariety of $X$, then $Z=Y\cap V$ is a compatibly split subvariety of $V$ with the induced splitting. 

If $Y$ is a reduced closed subscheme of $X$ such that $Y\cap V$ is dense in $Y$ and compatibly split by $\phi|_{V}$, then $Y$ is compatibly split by $\phi$.
\end{proposition} 

\begin{proof}
That $\phi|_{V}:F_*\mathcal{O}_V\rightarrow \mathcal{O}_V$ is a splitting is clear since $\phi: F_*\mathcal{O}_X\rightarrow \mathcal{O}_X$ is a morphism of sheaves. Furthermore, if $Y\subseteq X$ is compatibly split then so is $Z = Y\cap V$ since $\phi|_{V}(F_*\mathcal{I}_Y|_{V})\subseteq \mathcal{I}_Y|_{V}$.

Next suppose that $Z \subseteq V$ is compatibly split by $\phi|_{V}$. Then $Z = Y\cap V$ for some subscheme $Y\subseteq X$. Assume that $Y$ is reduced and that $Z$ is dense in $Y$. We now show that $\phi(F_*\mathcal{I}_Y)\subseteq\mathcal{I}_Y$. 

To begin, notice that $\mathcal{I}_Y\subseteq \phi(F_*\mathcal{I}_Y)$ and that $\mathcal{I}_{Y'} = \phi(F_*\mathcal{I}_Y)$ for some subscheme $Y'\subseteq X$. Then $Y\cap V = Y'\cap V$ since $\phi|_{V}$ compatibly splits $Y\cap V$. Thus, $\mathcal{I}_{Y'}$ vanishes on $Y\cap V$. As $Y$ is reduced and $Y\cap V$ is dense it $Y$, it follows that $\mathcal{I}_{Y'}$ vanishes on $Y$. So, $\mathcal{I}_Y = \mathcal{I}_{Y'}$ and $Y$ is compatibly split. 
\end{proof}

\subsection{Splittings of affine space}

In this subsection, we compile some useful results concerning splittings of affine space. 

\begin{remark}\cite[Remarks 1.1.4]{BK}
When $X = \Spec R$ is an affine scheme, a Frobenius splitting of $X$ is a map $\phi:R\rightarrow R$ satisfying the following three properties
\begin{center}
(1)~ $\phi(a+b) = \phi(a)+\phi(b)$,~~~  (2)~ $\phi(a^pb)=a\phi(b)$,~~~  (3)~ $\phi(1) = 1$,
\end{center}
for any $a,b\in R$. Indeed, (1) and (2) taken together is equivalent to the linearity of $\phi$ (noting that the module structure on the first copy of $R$ is given by $a\cdot b =a^pb$). If (2) holds, then $\phi(a^p) = a\phi(1)$, for any $a\in R$. Therefore $\phi$ splits the Frobenius endomorphism if and only if $\phi(1) = 1$.
\end{remark}

\begin{definition}\cite[Example 1.3.1]{BK}
Let $\textrm{Tr}:k[x_1,\dots,x_n]\rightarrow k[x_1,\dots,x_n]$ be the additive function defined on monomials as follows: 
 \begin{displaymath}
   \textrm{Tr}(m) = \left\{
     \begin{array}{ll}
       \frac{(x_1\cdots x_n m)^{1/p}}{x_1\cdots x_n}, & \textrm{if } (x_1\cdots x_n m) \textrm{ is a }p^{th}\textrm{ power}\\
       0, &  \textrm{else}
     \end{array}
   \right.
\end{displaymath} 
\end{definition}

Notice that $\textrm{Tr}(g\cdot)$, $g\in k[x_1,\dots,x_n]$ is a splitting of $\bA^n_k$ if $\textrm{Tr}(g\cdot 1) = 1$. In fact, all splittings of affine space have the form $\textrm{Tr}(g\cdot)$ for some $g\in k[x_1,\dots,x_n]$. 

Let $f\in k[x_1,\dots,x_n]$. The following theorem provides one sufficient condition for $\textrm{Tr}(f^{p-1}\cdot)$ to be a splitting of $\bA^n$. We make use of this result later.

\begin{theorem}\label{t;LMPinit} \cite{LMP}
Let $f\in k[x_1,\dots,x_n]$ be a degree $n$ polynomial such that, under some weighting of the variables, $\textrm{init}(f) = \prod_i{x_i}$. Then $\textrm{Tr}(f^{p-1}\cdot)$ is a Frobenius splitting of $k[x_1,\dots,x_n]$ that compatibly splits the hypersurface $\{f=0\}$. 
\end{theorem}

\begin{remark}
If $f\in k[x_1,\dots,x_n]$ is a degree $n$ polynomial such that, under some weighting of the variables, $\textrm{init}(f) = \prod_i{x_i}$, then the divisor $\{f=0\}$ is said to have \textbf{residual normal crossings} (see \cite{LMP}).
\end{remark}

We conclude this subsection by stating two theorems of Knutson concerning residual normal crossings divisors. Note that we make use of both theorems when discussing the compatibly split subvarieties of the affine open patch $U_{\langle x,y^n\rangle}\subseteq \Hilbn$.

\begin{theorem}\label{t;knutson}\cite[Theorem 2]{K}
Let $f\in k[x_1,\dots,x_n]$ be a degree $n$ polynomial and suppose that there is a weighting of the variables such that $\prod_{i=1}^n{x_i}$ is a term of the initial form $\textrm{init}(f)$. Then,
\begin{enumerate}
\item $\textrm{Tr}(f^{p-1}) = \textrm{Tr}(\textrm{init}(f)^{p-1})$, so $\textrm{Tr}(f^{p-1}\cdot)$ defines a Frobenius splitting iff $\textrm{Tr}(\textrm{init}(f)^{p-1}\cdot)$ does.
\item Assume hereafter that $\textrm{Tr}(f^{p-1}\cdot)$ and $\textrm{Tr}(\textrm{init}(f)^{p-1}\cdot)$ define Frobenius splittings. If $I$ is a compatibly split ideal with respect to the first splitting, then $\textrm{init}(I)$ is compatibly split with respect to the second splitting.
\item Let $\mathcal{Y}_f$ and $\mathcal{Y}_{\textrm{init}(f)}$ denote the poset of (irreducible) varieties compatibly split by $\textrm{Tr}(f^{p-1}\cdot)$ and $\textrm{Tr}((\textrm{init}(f))^{p-1}\cdot)$, partially ordered by inclusion. Then the map \[\pi_{f,\textrm{init}}:\mathcal{Y}_{\textrm{init}(f)}\rightarrow \mathcal{Y}_f, ~~~~~~~~Y'\mapsto \textrm{unique min. }Y\textrm{ such that }\textrm{init}(Y)\supseteq Y'\] is well-defined, order-preserving, and surjective.
\end{enumerate}
\end{theorem}

\begin{theorem}\label{t;yay}\cite{K}
Let $f\in k[x_1,\dots,x_n]$ and suppose there is weighting of the variables such that $\textrm{init}(f) = \prod_{i=1}^n{x_i}$. By the previous theorem, $\textrm{Tr}(f^{p-1}\cdot)$ Frobenius splits $\bA^n_k$. Let $Y_1$ and $Y_2$ be compatibly split subvarieties (with respect to this splitting). Then,
\begin{enumerate}
\item $\textrm{init}(Y_1\cap Y_2) = (\textrm{init}~Y_1)\cap (\textrm{init}~Y_2)$, and
\item if $Y_1\nsupseteq W$ and $\textrm{init}~Y_1\supseteq \textrm{init}~W$, then $W$ is not compatibly split.
\end{enumerate}
\end{theorem}

\begin{proof}
1. is \cite[Corollary 2]{K}.

To prove 2., we proceed by contradiction and suppose that $W$ is compatibly split. Then so is $W\cap Y_1$, and 
\[\begin{array}{cccc}\textrm{init}(W\cap Y_1) &= & (\textrm{init}~W)\cap (\textrm{init}~Y_1) & ~~\textrm{(by 1.)}\\
& = & \textrm{init}~W & ~~\textrm{(by assumption)} \end{array}\] 
Thus, $W\cap Y_1 = W$ (by Lemma \ref{l;init}) and so $W\subseteq Y_1$, a contradiction. 
\end{proof}

\subsection{Frobenius splittings and anticanonical sections}

Let $X$ be a regular variety. In this subsection, we discuss the relationship between Frobenius splittings of $X$ and sections of the $(p-1)^{\textrm{st}}$ power of $X$'s anticanonical sheaf. 

\begin{theorem}\label{t;divisors1}\cite[Section 1.3]{BK}
Let $X$ be a regular variety.
\begin{enumerate}
\item The absolute Frobenius $F:X\rightarrow X$ is finite and flat.
\item There is an isomorphism $\textrm{Hom}_{\mathcal{O}_X}(F_{*}\mathcal{O}_X,\mathcal{O}_X)\cong H^0(X,F_{*}(\omega_X^{1-p}))$, where $\omega_X$ is $X$'s canonical sheaf.
\end{enumerate}
\end{theorem}

Thus, certain sections of $F_*(\omega_X^{1-p})$ determine Frobenius splittings of $X$. We call such sections \textbf{splitting sections}. (Note: As only some elements of $\textrm{Hom}_{\mathcal{O}_X}(F_{*}\mathcal{O}_X,\mathcal{O}_X)$ are splittings, only some sections of $F_*(\omega_X^{1-p})$ are splitting sections.) 

\begin{theorem}\label{localSectionSplit}\cite[Theorem 1.3.8]{BK}
Let $X$ be a nonsingular variety. Let $\epsilon$ be the evaluation map \[\epsilon: \mathcal{H}om_{\mathcal{O}_X}(F_*\mathcal{O}_X,\mathcal{O}_X)\rightarrow \mathcal{O}_X,~~\phi\mapsto \phi(1).\] Then, $\epsilon$ can be identified with the map \[\hat{\tau}:F_*(\omega_X^{1-p})\rightarrow \mathcal{O}_X,\] given at any closed point $x\in X$ by \[\hat{\tau}(f(dt_1\wedge\cdots\wedge dt_n)^{-1}) = \textrm{Tr}(f), \textrm{  for all }f\in\mathcal{O}_{X,x}.\] Here $t_1,\dots,t_n$ is a system of local coordinates at $x$. Thus, an element $\sigma\in H^0(X,\omega_X^{1-p})$ determines a splitting of $X$ if and only if $\hat{\tau}(\sigma) =1$. 
\end{theorem}

\begin{definition}\cite[Exercise 1.3.E (2)]{BK}
If $\phi\in \textrm{Hom}_{\mathcal{O}_X}(F_{*}\mathcal{O}_X,\mathcal{O}_X)$ is splitting of a regular variety $X$, and $\phi$ corresponds (via the isomorphism in Theorem \ref{t;divisors1}) to $\sigma^{p-1}$, $\sigma\in H^0(X,F_*(\omega_X^{-1}))$, then we call the splitting $\phi$ a \textbf{$\mathbf{(p-1)^{st}}$ power}.
\end{definition}

\begin{remarks}
\begin{enumerate}
\item The particular splitting of $\Hilbn$ that we are concerned with is a $(p-1)^{st}$ power.   
\item Warning: Just because a splitting of $X$ is a $(p-1)^{st}$ power, it \emph{does not} necessarily imply that the induced splitting of a compatibly split subvariety is also a $(p-1)^{st}$ power! See the remark at the end of Subsection \ref{s;algorithm2pts} for more detail in the case of $\Hilb^2(\bA^2_k)$. 
\end{enumerate}
\end{remarks}

\begin{proposition}\label{p;zerosSplit}\cite[Proposition 1.3.11]{BK}
Let $X$ be a regular variety. If $\sigma\in H^0(X,\omega_X^{-1})$ is such that $\sigma^{p-1}$ determines a splitting of $X$, then the subscheme of zeros of $\sigma$ is compatibly split. In particular, this scheme is reduced.
\end{proposition}

We make use of the next lemma in Chapter 2.

\begin{lemma}\label{l;nonconstnonvan}
Let $X$ be a regular variety defined over a field $k$ of characteristic $p>0$ and let $\sigma_1, \sigma_2\in H^0(X,\omega_X^{-1})$. If (i) both of $\sigma_1^{p-1}$ and $\sigma_2^{p-1}$ are splitting sections of $X$, (ii) $\sigma_1$ and $\sigma_2$ vanish in the same locations, and (iii) there are no non-constant, non-vanishing functions on $X$, then $\sigma_1^{p-1} = \sigma_2^{p-1}$. 
\end{lemma}

\begin{proof}
By Proposition \ref{p;zerosSplit}, the subschemes of zeros of $\sigma_1$ and of $\sigma_2$ are reduced. Therefore, since $\sigma_1$ and $\sigma_2$ vanish in the same location and there are no non-constant, non-vanishing functions on $X$, we get that $\sigma_1 = c\sigma_2$ for some nonzero $c\in k$. Thus, $\sigma_1^{p-1} = c^{p-1}\sigma_2^{p-1}$, and we see that both $\sigma_2^{p-1}$ and $c^{p-1}\sigma_2^{p-1}$ are splitting sections. This can only occur if $c^{p-1}=1$. Therefore, $\sigma_1^{p-1} = \sigma_2^{p-1}$ as desired.
\end{proof}

Many of the above ideas can be extended to normal varieties.

\begin{definition}
Let $X$ be a normal variety with regular locus $X_{\textrm{reg}}$. 
\begin{enumerate}
\item \cite[Remark 1.3.12]{BK} Let $i:X_{\textrm{reg}}\rightarrow X$ be the inclusion of the regular locus. Define $\omega_{X}:=i_*\omega_{X_{\textrm{reg}}}$. This is the \textbf{canonical sheaf} of $X$.
\item A divisor $D\subseteq X$ is called \textbf{anticanonical} if $D\cap X_{\textrm{reg}}$ is anticanonical in $X_{\textrm{reg}}$.
\end{enumerate}
\end{definition}

\begin{theorem}\label{t;divisors}
\cite{KM} (Kumar-Mehta Lemma) Let $X$ be an irreducible normal variety which is Frobenius split by $\sigma\in H^{0}(X,F_*(\omega_X^{1-p}))$. If $Y$ is compatibly split then $Y$ is contained in the singular locus of $X$, or $Y\subseteq V(\sigma)$, where $V(\sigma)$ denotes the subscheme of zeros of $\sigma$. 
\end{theorem}

Using the ideas presented so far, we can find all compatibly split subvarieties of $\bA^n_k$ with the \textbf{standard splitting}, $\textrm{Tr}((x_1\cdots x_n)^{p-1}\cdot)$. 

\begin{example}\label{e;stdsplitting}
The anticanonical section $(x_1x_2\cdots x_n)\frac{d}{dx_1}\wedge\frac{d}{dx_2}\wedge\cdots\wedge \frac{d}{dx_n}$ determines the standard Frobenius splitting of $\bA^n_k$. By intersecting the components of the divisor $\{x_1x_2\cdots x_n=0\}$, decomposing the intersections, intersecting the new components, etc., we obtain the collection of coordinate subspaces. This is precisely the set of compatibly split subvarieties of $\bA^n_k$. (Proof: We proceed by induction on $n$. By Theorem \ref{t;divisors}, all compatibly split subvarieties of $\bA^n_k$ are contained inside of some $\{x_i=0\}$. Therefore, it suffices to show that the compatibly split subvarieties of each $\{x_i=0\}$ (with the induced splitting) are the coordinate subspaces. As the induced splitting of each $\{x_i=0\} = \Spec(k[x_1,\dots,x_{i-1},x_{i+1},\dots,x_n])$ is the standard splitting, we may apply induction to get the desired result.) 
\end{example}

Using this example, we get the following corollary of \cite[Theorem 2]{K}.

\begin{corollary}\cite{K}
Let $f\in k[x_1,\dots,x_n]$ be a degree $n$ polynomial and suppose that there is a weighting of the variables such that $\textrm{init}(f) = \prod_{i=1}^n{x_i}$. Then $\textrm{Tr}(f^{p-1}\cdot)$ defines a Frobenius splitting, and if $I$ is compatibly split with respect to this splitting, then $\textrm{init}(I)$ is a squarefree monomial ideal.
\end{corollary}

\subsection{Split morphisms and a result of David E Speyer}

\begin{definition}
Let $X$ and $Y$ be Frobenius split by $\phi_X$ and $\phi_Y$ respectively. Let $f:X\rightarrow Y$ be a morphism of schemes. Then $f$ is a \textbf{split morphism} if $f^{\sharp}\circ\phi_Y = \phi_X \circ f^{\sharp}$. In this case, we say that the splittings on $X$ and $Y$ are \textbf{compatible}.
\end{definition}

\begin{proposition}\label{p;splitmaps}
Let $f:X\rightarrow Y$ be a surjective morphism of integral schemes and suppose that $X$ is Frobenius split by $\phi_X$. 
\begin{enumerate}
\item There is at most one splitting of $Y$ making $f$ a split morphism.
\item Suppose $\phi_Y$ is a Frobenius splitting of $Y$ that makes $f$ a split morphism. If $Z$ is a compatibly split subscheme of $X$ then the scheme-theoretic image $f(Z)$ is a compatibly split subscheme of $Y$, and $f|_Z:Z\rightarrow f(Z)$ is a split morphism.
\item Suppose $Y$ and $W$ are Frobenius split by $\phi_Y$ and $\phi_W$. If $f:X\rightarrow Y$ and $g:Y\rightarrow W$ are split morphisms, then so is their composition $g\circ f$. 
\end{enumerate}
\end{proposition}

\begin{proof}
We begin by remarking that $f^{\#}:\mathcal{O}_Y\rightarrow f_*\mathcal{O}_X$ is injective because each of $X$ and $Y$ are integral, and $f$ is surjective. 
\begin{enumerate}
\item Suppose that $\phi_1$ and $\phi_2$ are two splittings of $Y$ which make $f$ a split morphism. Then, $\phi_X\circ f^{\#} = f^{\#}\circ \phi_1$ and $\phi_X\circ f^{\#} = f^{\#}\circ \phi_2$. Therefore, $f^{\#}\circ(\phi_1-\phi_2) = 0$ and since $f^{\#}$ is injective, $\phi_1 = \phi_2$. 
\item We consider the case $X = \Spec S$ and $Y = \Spec R$. Let $Z = \Spec S/I$ be a compatibly split subvariety of $X$ and let $\pi: S\rightarrow S/I$ be the quotient map. To show that $f(Z)$ is a compatibly split subvariety of $Y$, we show that $\textrm{ker}(\pi\circ f^{\#})$ is a compatibly split ideal of $R$. Equivalently, we show that $(\pi\circ f^{\#})(\phi_Y(\textrm{ker}(\pi\circ f^{\#}))) = 0\in S/I$. Now, \[(\pi\circ f^{\#}\circ \phi_Y)(\textrm{ker}(\pi\circ f^{\#})) = (\pi\circ\phi_X\circ f^{\#})(\textrm{ker}(\pi\circ f^{\#})) \subseteq \pi(\phi_X(\textrm{ker}(\pi)))\subseteq \pi(\textrm{ker}(\pi)) = 0\] where the equality holds because $f$ is a split morphism, and the second inclusion holds because $\ker(\pi) = I$ is a compatibly split ideal of $S$. Thus, $\textrm{ker}(\pi\circ f^{\#})$ is a compatibly split ideal of $R$. 

That $f|_Z:Z\rightarrow f(Z)$ is a split morphism then follows immediately. 
\item $(g\circ f)^{\#}\circ \phi_W = f^{\#}\circ g^{\#}\circ \phi_W = f^{\#}\circ \phi_Y\circ g^{\#} = \phi_X\circ f^{\#}\circ g^{\#}$.
\end{enumerate}
\end{proof}

Sometimes it's easier to prove that a morphism is split by working on an open subscheme. 

\begin{lemma}\label{l;openSplitMorph}
Let $f: X\rightarrow Y$ be a surjective morphism of (irreducible) Frobenius split varieties. Let $U\subseteq X$ and $V\subseteq Y$ be open subschemes such that $f$ restricts to a surjective morphism $f|_U:U\rightarrow V$. Suppose further that $U\subseteq X$ and $V\subseteq Y$ are given the induced Frobenius splittings. Then $f$ is a split morphism if and only if $f|_U$ is a split morphism.
\end{lemma}

\begin{proof}
The forward direction is clear. So, assume that $f|_U:U\rightarrow V$ is a surjective, split morphism. Suppose $W\subseteq Y$ is an open set. If $s\in \mathcal{O}_Y(W)$, we must show that $(f^{\#}\circ \phi_Y)(s) = (\phi_X\circ f^{\#})(s)$. We have \[((f|_U)^{\#}\circ\phi_V)(s|_{V\cap W}) = (f^{\#}\circ \phi_Y)(s|_{V\cap W}) = [(f^{\#}\circ\phi_Y)(s)]|_{U\cap f^{-1}(W)}\] and \[(\phi_U\circ (f|_U)^{\#})(s|_{V\cap W}) = (\phi_X\circ f^{\#})(s|_{V\cap W}) = [(\phi_X\circ f^{\#})(s)]|_{U\cap f^{-1}(W)}.\] Because $f|_U:U\rightarrow V$ is split, $(\phi_U\circ (f|_U)^{\#})(s|_{V\cap W}) = ((f|_U)^{\#}\circ \phi_V)(s|_{V\cap W})$. Therefore, \[[(f^{\#}\circ\phi_Y)(s)]|_{f^{-1}(W)\cap U} = [(\phi_X\circ f^{\#})(s)]|_{f^{-1}(W)\cap U}.\] As $f^{-1}(W)\cap U$ is open and dense in $f^{-1}(W)$, it follows that $(f^{\#}\circ\phi_Y)(s) = (\phi_X\circ f^{\#})(s)$. Therefore $f:X\rightarrow Y$ is a split morphism.
\end{proof}

We now provide some examples of split morphisms.

\begin{example}\label{p;finitemaps}\cite[Examples 1.1.10]{BK}
\begin{enumerate}
\item Let $X$ be a non-singular affine variety and let $G$ be a finite group which acts on $X$. Let $\pi:X\rightarrow X/G$ denote the quotient map. If $X$ is Frobenius split and $p$ does not divide $|G|$, then there is a unique induced splitting of $X/G$ making $\pi$ a split map.
\item Suppose that $f:X\rightarrow Y$ is a finite surjective map of varieties, that $Y$ is normal, and that $X$ is split by $\phi$. If $p$ does not divide $\textrm{deg}(f)$, then there is a unique splitting $\phi_Y$ of $Y$ making $f$ a split map.
\end{enumerate}
\end{example}

\begin{lemma}\label{l;etaleSplit}
Let $X$ and $Y$ be non-singular varieties. Suppose that $f:X\rightarrow Y$ is surjective and \'{e}tale (so we can identify anticanonical sheaves), and that all maps on residue fields $\kappa(y)\rightarrow \kappa(x)$ (where $y\in Y$ and $x\in X$ is one of $y$'s preimages) are isomorphisms. Suppose that $\sigma_Y\in H^0(Y,F_*(\omega_Y^{1-p}))$ is a splitting section of $Y$. Then,
\begin{enumerate}
\item $f^{*}(\sigma_Y)$ is a splitting section of $X$.
\item If $X$ is Frobenius split by the splitting section $f^{*}(\sigma_Y)$, then $f:X\rightarrow Y$ is a split morphism. 
\end{enumerate} 
\end{lemma}

\begin{proof}
\begin{enumerate}
\item Let $y\in Y$ and let $x\in X$ be any of $y$'s preimages. Then the induced map on completed local rings $\hat{\mathcal{O}}_{Y,y}\rightarrow \hat{\mathcal{O}}_{X,x}$ is an isomorphism. Identify $\hat{\mathcal{O}}_{Y,y}$ with $k[[s_1,\dots,s_n]]$ and $\hat{\mathcal{O}}_{X,x}$ with $k[[t_1,\dots,t_n]]$ such that the induced map on completed local rings sends $s_i$ to $t_i$. Suppose that $g(s_1,\dots,s_n)\in k[[s_1,\dots,s_n]]$ is the local expansion of $\sigma_Y$ at $y$. Then, $g(t_1,\dots,t_n)\in k[[t_1,\dots,t_n]]$ is the local expansion of $f^*(\sigma_Y)$ at $x$. Because $\sigma_X$ is a splitting section of $Y$, we can apply \cite[Theorem 1.8]{BK} (see Theorem \ref{localSectionSplit}) to see that $\textrm{Tr}(g(s_1,\dots,s_n)) = 1$. Then, $\textrm{Tr}(g(t_1,\dots,t_n))$ is also $1$ and $f^*(\sigma_Y)$ is a splitting section of $X$.
\item This follows by identifying $\hat{\mathcal{O}}_{Y,y}$ with $k[[s_1,\dots,s_n]]$ and $\hat{\mathcal{O}}_{X,x}$ with $k[[t_1,\dots,t_n]]$ as done in the proof of item 1.
\end{enumerate}
\end{proof}

\begin{example}\label{l;splitprojection}\cite[see Exercise 1.3 (8)]{BK}
Let $X$ and $Y$ be Frobenius split with splittings $\phi_X$ and $\phi_Y$ respectively. Then, the tensor product \[\phi:F_*\mathcal{O}_{X\times Y}\rightarrow \mathcal{O}_{X\times Y},~~~f_1\otimes f_2\mapsto \phi_X(f_1)\otimes\phi_Y(f_2)\] is a splitting of $X\times Y$. If $X$ and $Y$ are non-singular and $\phi_X$ (respectively $\phi_Y$) corresponds to the section $\sigma_X\in H^0(X,\omega_X^{1-p})$ (respectively $\sigma_Y\in H^0(Y,\omega_Y^{1-p})$), then $\phi$ corresponds to $\sigma_X\boxtimes\sigma_Y\in H^0(X\times Y, \omega_{X\times Y}^{1-p})$. 

The projection morphisms $\pi_1:X\times Y\rightarrow X$ and $\pi_2:X\times Y\rightarrow Y$ are split morphisms.
\end{example}

\begin{remark}\label{r;prodSplit}\cite[see Exercise 1.3 (8)]{BK}
Let all schemes and Frobenius splittings be as in the above example. It is useful to note that if $Z_1$ is a compatibly split subscheme of $X$ and $Z_2$ is a compatibly split subscheme of $Y$, then $Z_1\times Z_2$ is compatibly split in $X\times Y$. 
\end{remark}

\begin{example}\label{e;aff}
Let $S$ be a graded ring and let $X = \textrm{Proj}(S)$ be Frobenius split by $\phi:F_*\mathcal{O}_X\rightarrow \mathcal{O}_X$. Define $X_{\textrm{aff}} := \Spec \mathcal{O}_X(X)$. Because $\mathcal{O}_X$ is a split sheaf, $\mathcal{O}_X(X)$ is a split ring. Furthermore, the structure map $\pi:X\rightarrow X_{\textrm{aff}}$ is a split map. To see this, consider any $\frac{a}{b}\in\mathcal{O}_{X_{\textrm{aff}}}(U)$ for $a,b\in \mathcal{O}_X(X)$. Then, if $\phi_{X_{\textrm{aff}}}$ denotes the induced splitting of $X_{\textrm{aff}}$, we have:
\[(\pi^{\sharp}\circ\phi_{X_{\textrm{aff}}})\left(\frac{a}{b}\right) = \pi^{\sharp}\left(\frac{\phi_{X_{\textrm{aff}}}(ab^{p-1})}{b}\right) = (\phi\circ\pi^{\sharp})\left(\frac{a}{b}\right)\]
The last equality follows because $\phi$ agrees with $\phi_{X_{\textrm{aff}}}$ on $\mathcal{O}_X(X)$.
\end{example}

The remainder of this subsection concerns an unpublished theorem of David E Speyer. He graciously allowed his theorem and proof to be included in this thesis. Both the theorem and proof were communicated to Allen Knutson via email, and most of what appears below is taken directly from the emailed document. (However, this document contains slightly different proofs of Speyer's first two lemmas (Lemma \ref{l;lemma1} and Lemma \ref{Key}), Lemma \ref{l;normModp} did not appear as a separate lemma in the emailed document, and there is a little bit of additional explanation at the end of Proposition \ref{MainCase}.)

\begin{theorem} \label{t;des} (Speyer)
Let $f:Y \to X$ be a finite map of varieties over $\Spec \bZ$.
Then there is an integer $N$ such that, for any $p >N$, any choice of compatible Frobenius splittings on $X/p$ and $Y/p$, and any compatibly split subvariety $V$ of $X/p$, the reduction of $f^{-1}(V)$ is compatibly split.
\end{theorem}

This theorem applies, in particular, when $Y$ is the integral closure of $X$. 

We begin with some lemmas about field extensions in characteristic $p$.

\begin{lemma} \label{l;lemma1}
Let $L/K$ be a finite dimensional extension of characteristic $p>0$ fields. Then $\Tr_{L/K}(x^p) = \Tr_{L/K}(x)^p$.
\end{lemma}

\begin{proof}
If $L/K$ is not separable, then both sides are zero, so assume that $L/K$ is separable.
Let $x\in L/K$ and consider the intermediate extension $K(x)/K$. Then, the minimal polynomial of $x$ in $K(x)/K$ agrees with the characteristic polynomial of multiplication by $x$. (Indeed, suppose that $1,x,\dots,x^d$ is a basis of $K(x)/K$. Then, the minimal polynomial of $x$ is $\lambda^{d+1}-c_1\lambda^d-c_2\lambda^{d-1}-\cdots-c_d\lambda-c_{d+1}$ for some $c_1,\dots,c_{d+1}\in K$, and multiplication by $x$ is given by the $(d+1)\times(d+1)$-matrix that has $1$s along the subdiagonal, $[c_1,\dots,c_{d+1}]^t$ as the last column, and $0$s elsewhere. The characteristic polynomial of this matrix is $\lambda^{d+1}-c_1\lambda^d-c_2\lambda^{d-1}-\cdots-c_d\lambda-c_{d+1}$.) This polynomial has all distinct roots by the separability assumption and so multiplication by $x$ has all distinct eigenvalues. Writing the matrix of multiplication by $x$ in an eigenbasis, we see that $\Tr_{K(x)/K}(x^p) = \Tr_{K(x)/K}(x)^p$. Then,
\[\Tr_{L/K}(x^p) = \Tr_{K(x)/K}\circ\Tr_{L/K(x)}(x^p)= \Tr_{K(x)/K}([L:K(x)]x^p)= [L:K(x)]\Tr_{K(x)/K}(x^p)\] and
\[\Tr_{L/K}(x)^p = \Tr_{K(x)/K}([L:K(x)]x)^p = [L:K(x)]^p\Tr_{K(x)/K}(x)^p  = [L:K(x)]\Tr_{K(x)/K}(x^p).\]
\end{proof}

The next lemma is crucial.

\begin{lemma} \label{Key}
Let $L/K$ be a finite dimensional extension of characteristic $p$ fields and let $\phi$ be a splitting of $L$ which restricts to a splitting of $K$.
Then $\Tr_{L/K} \circ \phi = \phi \circ \Tr_{L/K}$.
\end{lemma}

\begin{proof}
If $L/K$ is not separable, then both sides are zero, so assume that $L/K$ is separable.
Suppose that $\ell_1,\dots,\ell_d$ is a basis for $L$ over $K$. Then, $\ell_1^p,\dots,\ell_d^p$ is too.   
Any $x\in L/K$ can be written as $\sum_{i=1}^dc_i\ell_i^p$, for some $c_i\in K$. Then,\[(\Tr_{L/K}\circ \phi)(x) = \Tr_{L/K}(\sum_{i=1}^d \phi(c_i)\ell_i) = \sum_{i=1}^d\phi(c_i)\Tr_{L/K}(\ell_i)\] where the second equality holds because $\phi$ preserves $K$ by assumption. On the other hand, \[(\phi\circ \Tr_{L/K})(x) = \phi(\sum_{i=1}^d c_i\Tr_{L/K}(\ell_i^p)) = \phi(\sum_{i=1}^d c_i\Tr_{L/K}(\ell_i)^p) = \sum_{i=1}^d \phi(c_i)\Tr_{L/K}(\ell_i)\] where the second equality follows by Lemma \ref{l;lemma1}. 
\end{proof}


We now prove the central case of Theorem~\ref{t;des}.

\begin{proposition} \label{MainCase}
Let $f: Y \to X$ be a finite surjective map of varieties of characteristic $p$, with $X$ and $Y$ both normal. Suppose that $p$ is greater than $\deg f$.
Then, for any compatible splittings on $X/p$ and $Y/p$, and $V$ any compatibly split subvariety of $X/p$, the reduction of $f^{-1}(V)$ is also split.
\end{proposition}

We note that the scheme $f^{-1}(V)$ may not be reduced and, therefore, it is not necessarily true that $f^{-1}(V)$ is split when equipped with the inverse image scheme structure.
An example is to take $p$ an odd prime, $\Spec k[t] \to \Spec k[t^2]$ with $\phi$ induced by $t^{p-1} /dt^{p-1}$, and $V$ to be the origin.

\begin{proof}
We may pass to the neighborhood of a generic point of $V$.
Let $A$ be the completion of the local ring of $X$ at $V$, with $I$ the maximal ideal of $A$.
Let $B$ be the completion of the local ring of $Y$ at one of the primes above the generic point of $V$, and $J$ the maximal ideal of $B$.
So we know that $I$ is compatibly split, and we must show that $J$ is.

Since $X$ and $Y$ are normal, we know that $A$ and $B$ are integral domains.
Let their fraction fields be $K$ and $L$.
The degree of $L/K$ is bounded by the degree of the map $f$.

Suppose for the sake of contradiction that $\phi(J) \not \subseteq J$. 
Since $B$ is local, $1 \in \phi(J)$; say $\phi(x) = 1$.
Then, by Lemma~\ref{Key}, $\phi(\Tr_{L/K}(x))= \Tr_{L/K}(1) = [L:K] $. By the hypothesis on $p$, this is a nonzero scalar.

We now claim that $\Tr_{L/K}(x) \in I$. To prove this, let $E$ be the Galois closure of $L/K$. Then, \[\Tr_{E/K}(x) = \Tr_{L/K}\circ\Tr_{E/L}(x) = [E:L]\Tr_{L/K}(x).\] $\Gal(E/K)$ acts on the primes lying over $I$ in the integral closure of $A$ in $E$. So, $\Tr_{E/K}(x)\in I$. By the above computation, $\Tr_{L/K}(x)\in I$ as well.

This violates the hypothesis that $I$ is compatibly split.
\end{proof}

We consider one more lemma (Lemma \ref{l;normModp}) before including the remainder of Speyer's proof of Theorem \ref{t;des}. We thank user QiL on math.stackexchange.com for providing a proof of this lemma. (See http://math.stackexchange.com/questions/163929/.) 

\begin{lemma}\label{l;normModp}
Let $X$ be a variety over $\Spec\bZ$ and let $\nu:\tilde{X}\rightarrow X$ be its normalization. For $p$ sufficiently large, $\tilde{X}/p$ is the normalization of $X/p$.
\end{lemma}

\begin{proof}
Let $s\in \Spec \bZ$ and let $X_s$ (respectively $\tilde{X}_s$) denote the fiber of $X\rightarrow \Spec \bZ$ (respectively $\tilde{X}\rightarrow \Spec \bZ$) over the point $s$. We will show that there is some open set $S\subseteq \Spec \bZ$ such that for every $s\in S$, (i) $\tilde{X}_s\rightarrow X_s$ is birational, and (ii) $\tilde{X}_s$ is normal.

To show that (i) holds, we apply \cite[IV.13.1.1]{ega}. Let $Z\subseteq X$ be the closed subscheme over which $\nu:\tilde{X}\rightarrow X$ fails to be an isomorphism. Then, $\textrm{dim}(Z_s) = \textrm{dim}(Z_{\bQ})<\textrm{dim}(X_{\bQ})$ for all but at most finitely many $s\in \Spec \bZ$.  Each irreducible component of $X_s$ has dimension at least $\textrm{dim}(X_{\bQ})$. Thus, $Z_s$ is nowhere dense in $X_s$ and so $\tilde{X}_s\rightarrow X_s$ is birational.

To show that (ii) holds, let $f:\tilde{X}\rightarrow \Spec \bZ$ denote the structure morphism and apply \cite[IV.12.1.6 (iv)]{ega}. This says that the set \[U := \{x\in \tilde{X}~|~\tilde{X}_{f(x)}\textrm{ is geometrically normal}\}\] is open. Then, $f(U)$ is a constructible set of $\Spec \bZ$. Because $\tilde{X}_{\bQ}$ is geometrically normal, and $U\subseteq \tilde{X}$ is open, $f(U)$ must be an open subset of $\Spec \bZ$. 
\end{proof}

We now prove Theorem \ref{t;des}. We proceed by induction on $\dim X$.

\begin{proof}[Proof of Theorem \ref{t;des}]
We can immediately pass to components, and thus reduce to the case that $X$ and $Y$ are integral.
Also, if $f$ is not surjective, then we can factor $f$ as $Y \to f(Y) \to X$. 
Any compatible splittings of $X$ and $Y$ will induce a splitting of $f(Y)$ and split map $Y\to f(Y)$, whose image has smaller dimension. We therefore obtain the result by induction.

Thus, we are reduced to the case where $X$ and $Y$ are integral, and $f$ is surjective.
Let $\tilde{X}$ and $\tilde{Y}$ denote the normalizations of $X$ and $Y$.
For $p$ sufficiently large, Lemma \ref{l;normModp} implies that $\tilde{X}/p$ is the normalization of $X/p$ and that $\tilde{Y}/p$ is the normalization of $Y/p$.
Any compatible splitting on $X/p$ and $Y/p$ will give splittings of the normalizations of $X/p$ and $Y/p$, and all of these splittings will be compatible.
Write $\tilde{f}$ for the map $\tilde{Y} \to \tilde{X}$, and write $m$ and $n$ for the maps $\tilde{X} \to X$ and $\tilde{Y} \to Y$.

Let $\Spec A$ be an affine chart on $X$, and let $\tilde{A}$ be the normalization of $A$.
Consider the conductor ideal $D = \{ d \in A~|~d\tilde{A}\subseteq A \}$ and recall that for any Frobenius splitting of $A/p$, $D/p$ is compatibly split (see the proof of Proposition \ref{p;normalizationSplit}).

The construction of $D$ sheafifies; let $\Delta$  and $\widetilde{\Delta}$ be the corresponding subvarieties of $X$ and $\tilde{X}$.
Then $\tilde{X} \setminus \Delta \to X \setminus \Delta$ is an isomorphism, and $\widetilde{\Delta} \to \Delta$ is a finite map.
By induction, there is an $N'$ such that for $p>N'$, split subvarieties of $\Delta/p$ lift to split subvarieties of $\widetilde{\Delta}/p$.
Let $N = \max(\deg f, N')$.

Take $p>N$, a compatible splitting $\phi$, and a split subvariety $V$ of $X/p$.
If the generic point of $V$ is not in $\Delta$, then $m^{-1}(V)$ is split because $m$ is an isomorphism away from $\Delta$.
If the generic point is in $\Delta$, then $m^{-1}(V)$ is split because $p > N'$.
So, either way, $m^{-1}(V)$ is split.

By Proposition~\ref{MainCase}, $\tilde{f}^{-1}(m^{-1}(V))$ is split. 
The image of a split variety is split, so \[n (\tilde{f}^{-1}(m^{-1}(V))) = f^{-1}(V)\] is split. 
This is the desired result.
\end{proof}

\subsection{The Knutson-Lam-Speyer algorithm for finding compatibly split subvarieties}\label{s;kls}

In this subsection we discuss an unpublished algorithm due to Allen Knutson, Thomas Lam, and David Speyer which, in certain cases, finds all compatibly split subvarieties of $X$. Note that many of the ideas appearing below are very similar to those found in \cite[Section 5]{KLS}. 

\begin{algorithm}\label{a} (Knutson-Lam-Speyer)

Let $X$ be a normal variety that is Frobenius split by $\sigma^{p-1}$, $\sigma\in H^0(X,F_*(\omega_X^{-1}))$. Let $D = V(\sigma)$ be the vanishing set of the anticanonical section $\sigma$. 

The input of the algorithm is the pair $(X,D)$. The output of the algorithm is a list $L$ of compatibly split subvarieties of $X$. The steps of the algorithm are as follows. We start with an empty list $L'$.

\noindent\textbf{Step 1: } Add $X$ to the list $L'$. 

\noindent\textbf{Step 2: } Let $S$ be the closure (in $X$) of the singular locus of $X\setminus D$. Add all compatibly split subvarieties of $S$ to $L'$.

\noindent\textbf{Step 3: } If $D$ is empty, terminate the algorithm for the pair $(X,D)$. If $D$ is non-empty, decompose $D$ into its irreducible components, $D = D_1\cup\cdots\cup D_r$. For each $i\in 1,\dots,r$, let $E_i = D_1\cup\cdots\cup \hat{D_i}\cup\cdots\cup D_r$. Replace $(X,D)$ by the collection of new pairs $\{(D_i,D_i\cap E_i)~|~1\leq i\leq r\}$.

\noindent\textbf{Step 4: } Replace each $(D_i, D_i\cap E_i)$, $1\leq i\leq r$, with the pair $P_i$ using the following procedure: If $D_i$ is normal, then $P_i := (D_i,D_i\cap E_i)$. If $D_i$ is not normal, then $P_i := (\tilde{D_i},\nu^{-1}(D_i\cap E_i)\cup \nu^{-1}((D_i)_{\textrm{non-R1}}))$ where $\nu:\tilde{D_i}\rightarrow D_i$ is the normalization of $D_i$ and $(D_i)_{\textrm{non-R1}}$ is the (possibly empty) union of codimension-$1$ components of the singular locus of $D_i$. For each $P_i$, $1\leq i\leq r$, return to step 1 and rename $P_i$ by $(X,D)$ to match the notation used above.

At the very end, map all varieties in the list $L'$ forward to $X$ by the composition of the relevant finite morphisms (i.e. the composition of restricted normalization morphisms). The elements of the list $L$ are the images, under these maps, of the elements of the list $L'$.
\end{algorithm}

\begin{lemma}
Let $X$ be a variety defined over $\Spec \bZ$. Suppose $X/p$ is normal and Frobenius split by $\sigma^{p-1}$, $\sigma\in H^0(X/p,F_*(\omega_{X/p}^{-1}))$. Let $D = V(\sigma)$ be the vanishing set of the anticanonical section $\sigma$. 

Run the algorithm starting with the pair $(X/p,D)$ to obtain a list $L$ of subvarieties of $X/p$. There exists an integer $N>0$ such that, for all $p>N$, every subvariety appearing in the list $L$ is compatibly split.
\end{lemma}

\begin{proof}
To begin, for every pair $(Y,D_Y)$ appearing after some number of iterations of steps 1, 3, and 4, we show that (i) $D_Y$ is compatibly split in $Y$, and (ii) there is a finite split morphism $f:Y\rightarrow f(Y)$ such that $f(Y)$ is a compatibly split subvariety of $X/p$.

We proceed by induction on the number of iterations of steps 1, 3, and 4. When $n=0$, the result is automatic. So suppose that $(Y,D_Y)$ is a pair that shows up after $n$ iterations of steps 1, 3, and 4. By induction, $D_Y$ is a compatibly split subvariety of $Y$ and there is a finite split morphism $f:Y\rightarrow f(Y)$ such that $f(Y)$ is a compatibly split subvariety of $X/p$. Let $D_Y = D_1\cup\cdots\cup D_r$ where each $D_i$ is irreducible, and let $E_i = D_1\cup\cdots \cup \hat{D_i}\cup\cdots D_r$, $1\leq i\leq r$. If $D_i$ is normal, then $(D_i,D_i\cap E_i)$ appears after $n+1$ iterations of steps 1, 3, and 4, $D_i\cap E_i$ is compatibly split in $D_i$, and $f|_{D_i}$ is a finite split morphism such that $f|_{D_i}(D_i)$ is a compatibly split subvariety of $X/p$ (by Proposition \ref{p;splitmaps}). If $D_i$ is not normal, then $(\tilde{D_i},\nu^{-1}(D_i\cap E_i)\cup \nu^{-1}((D_i)_{\textrm{non-R1}}))$ (where $\nu:\tilde{D_i}\rightarrow D_i$ is the normalization map) appears after $n+1$ iterations of steps 1, 3, and 4. By Theorem \ref{t;des}, there is an $N'$ such that, for all $p>N'$, $\nu^{-1}(D_i\cap E_i)\cup \nu^{-1}((D_i)_{\textrm{non-R1}})$ is compatibly split in $\tilde{D_i}$. Furthermore, $f|_{D_i}\circ \nu$ is a finite split morphism such that $(f|_{D_i}\circ\nu)(\tilde{D_i})$ is a compatibly split subvariety of $X$ (by Propositions \ref{p;normalizationSplit} and \ref{p;splitmaps}).

Any subvariety in the list $L$ which is not found using only steps 1, 3, and 4 of the algorithm must be of the form $f(Z)$ where $f:Y\rightarrow f(Y)$ is a finite split morphism, $f(Y)$ is a compatibly split subvariety of $X/p$, and $Z$ is a compatibly split subvariety contained in the singular locus of $Y$. It follows by Proposition \ref{p;splitmaps} that $f(Z)$ is compatibly split. 
\end{proof}

\begin{remark}
Step 2 of the algorithm is difficult to perform in general. Luckily, in the setting that we care about in this thesis, steps 1, 3, and 4 are enough to find all compatibly split subvarieties (at least in small examples). However, as the next example illustrates, step 2 is necessary.
\end{remark}

\begin{example} Let $p\equiv 1 ~(\textrm{mod}~3)$ and consider the splitting of $\bA^3$ given by $\textrm{Tr}(f^{p-1}\cdot)$, $f=x^3+y^3+z^3$. In this case, $(X,D) = (\mathbb{A}_k^3,\{x^3+y^3+z^3=0\})$. Without checking for compatibly split subvarieties of the singular locus of $D$, we would miss finding the origin which is compatibly split.
\end{example} 

We now show that, in certain cases, the algorithm is guaranteed to find \emph{all} compatibly split subvarieties of a normal Frobenius split variety $X$. In order to prove this, we make use of the following lemmas.


\begin{lemma}\label{l;norm}\cite[Lemma 1.1.7 (iii)]{BK}
If $X$ is a normal variety and $U$ is an open subset with complement of codimension at least $2$, then $X$ has a Frobenius splitting if and only if $U$ does. In fact, any splitting of $U$ is the restriction of a unique splitting of $X$. In particular, $X$ has a splitting if and only if its regular locus does.
\end{lemma}

\begin{lemma}\label{l;adj2}(similar to \cite[Exercise 1.3.E (4*)]{BK})
Let $X$ be a non-singular Frobenius split variety and let $D_1$ be a non-singular, compatibly split prime divisor. Suppose that the splitting of $X$ is given by $\sigma_1^{p-1}\tilde{\tau}^{p-1}$ where $V(\sigma_1) = D_1$ and $\sigma_1\tilde{\tau}\in H^0(X,\omega_X^{-1})$. Then the induced splitting of $D_1$ is given by $\tau^{p-1}$ where $\tau$ is the residue of $\tilde{\tau}$ in $\omega_{D_1}^{-1}\cong (\omega_X(D_1)\otimes_{\mathcal{O}_X}\mathcal{O}_{D_1})^{-1}$.
\end{lemma}

\begin{proof}
Let $x$ be any closed point of $D_1$. We show that the statement holds in the completion of the local ring $\mathcal{O}_{X,x}$. 

Let $t_1,...,t_n$ be local parameters so that the completion of $\mathcal{O}_{X,x}$ is identified with the power series ring $k[[t_1,...,t_n]]$. We can choose these parameters such that $t_1$ vanishes along $D_1$. Then, $\sigma\tilde{\tau}$ can be expanded locally as $t_1\tilde{h}(dt_1\wedge\cdots\wedge dt_n)^{-1}$ where the power series $\tilde{h}\in k[[t_1,\dots,t_n]]$ is not divisible by $t_1$. Therefore, the splitting of $k[[t_1,\dots,t_n]]$ is given by $\textrm{Tr}(t_1^{p-1}\tilde{h}^{p-1}\cdot)$.

The induced splitting of the quotient $k[[t_1,\dots,t_n]]/\langle t_1\rangle$ is $\textrm{Tr}(h^{p-1}\cdot)$, where $h := \tilde{h}(0,t_2,\dots,t_n)$. As $h(dt_2\wedge \cdots\wedge dt_n)^{-1}$ is the local expansion $\tau$ in the power series ring $k[[t_2,\dots,t_n]]\cong \hat{\mathcal{O}}_{D_1,x}$, we obtain the desired result.
\end{proof}

\begin{lemma}\label{l;adj}
Let $X$ be a normal variety that is Frobenius split by $\sigma^{p-1}$ where $\sigma\in H^0(X,\omega_X^{-1})$. Let $D = V(\sigma)$ and suppose that $D = D_1+\cdots + D_r = D_1+E$ where the support of each $D_i$ is irreducible and non-empty. Let $\textrm{sing}(X)$ denote the singular locus of $X$ and suppose that $\textrm{sing}(X)\cap D_1$ has codimension $\geq 2$ in $D_1$. 
\begin{enumerate}
\item If $D_1$ is normal, then the induced splitting of $D_1$ is given by $\tau^{p-1}$, where $\tau\in H^{0}(D_1,\omega_{D_1}^{-1})$ and $V(\tau) = D_1\cap E$. Furthermore, all compatibly split subvarieties of $D_1$ are contained in $D_1\cap E$ or in $\textrm{sing}(D_1)$.
\item Suppose $D_1$ is not normal. Let $(D_1)_{\textrm{non-R1}}$ denote the codimension-$1$ component of $\textrm{sing}(D_1)$. If $\nu:\tilde{D_1}\rightarrow D_1$ is the normalization of $D_1$, then, for large $p$, the splitting of $\tilde{D_1}$ is given by $\tau^{p-1}$ where $V(\tau) = \nu^{-1}(D_1\cap E)\cup\nu^{-1}((D_1)_{\textrm{non-R1}})$. Furthermore, all compatibly split subvarieties of $\tilde{D_1}$ are contained in $\nu^{-1}(D_1\cap E)\cup\nu^{-1}((D_1)_{\textrm{non-R1}})$ or $\textrm{sing}(\tilde{D_1})$.
\end{enumerate}
\end{lemma}

\begin{proof}
\begin{enumerate}

\item Let $X' = X_{\textrm{reg}}\setminus \textrm{sing}(D_1)$, $D_1' = (D_1)_{\textrm{reg}}\cap X'$, and $E' = E\cap X'$. Then, $X'$ and $D_1'$ are non-singular. By Lemma \ref{l;adj2}, the induced splitting of $D_1'$ is given by an anticanonical section $\tau'\in H^0(D_1',\omega_{D_1'}^{-1})$ such that $V(\tau') = D_1'\cap E'$. Because $(D_1)_{\textrm{reg}}\cap\textrm{sing}(X)$ has codimension $\geq 2$ in $(D_1)_{\textrm{reg}}$ and $D_1$ is normal, $\tau'$ extends uniquely to an anticanonical section $\tau\in H^0(D_1,\omega_{D_1}^{-1})$ such that $\tau^{p-1}$ determines a splitting of $D_1$ and $V(\tau) = D_1\cap E$. By the Kumar-Mehta Lemma (see Theorem \ref{t;divisors}), all compatibly split subvarieties of $D_1$ are contained inside of $D_1\cap E$ or $\textrm{sing}(D_1)$.

\item 
Now suppose that $D_1$ is not normal and let $\nu:\tilde{D_1}\rightarrow D_1$ be the normalization of $D_1$. By Proposition \ref{p;normalizationSplit}, $\tilde{D_1}$ has a Frobenius splitting induced from $D_1$ that compatibly splits $\nu^{-1}((D_1)_{\textrm{non-R1}})$. By Speyer's theorem, $\nu^{-1}(D_1\cap E)$ is also compatibly split for large primes $p$. 

Let $\tau\in H^0((\tilde{D_1})_{\textrm{reg}}, F_{*}(\omega_{(\tilde{D_1})_{\textrm{reg}}}^{1-p}))$ determine the splitting of $(\tilde{D_1})_{\textrm{reg}}$. Let $D_\tau$ denote the associated divisor. Then, \[D_{\tau} = (p-1)\nu^{-1}(D_1\cap E) + (p-1)\nu^{-1}((D_1)_{\textrm{non-R1}})+F\] for some effective divisor $F$. 

By the same argument as the one given to prove 1., the splitting of $(D_1)_{\textrm{reg}}$ is a $(p-1)^{\textrm{st}}$ power and $(D_1)_{\textrm{reg}}\cap E$ is the corresponding compatibly split anticanonical divisor in $(D_1)_{\textrm{reg}}$. As $\nu$ is an isomorphism over the open set $(D_1)_{\textrm{reg}}$, we see that $\tau|_{\nu^{-1}((D_1)_{\textrm{reg}})}$ is a $(p-1)^{st}$ power, and that \[D_{\tau}|_{\nu^{-1}((D_1)_{\textrm{reg}})} =  (p-1)[\nu^{-1}(D_1\cap E)|_{\nu^{-1}((D_1)_{\textrm{reg}})}].\] As $F$ is effective and $\nu^{-1}((D_1)_{\textrm{non-R1}})$ is compatibly split, $F$ must be empty. Thus, \[D_{\tau} = (p-1)\nu^{-1}(D_1\cap E) + (p-1)\nu^{-1}((D_1)_{\textrm{non-R1}}).\]

By the Kumar-Mehta Lemma, all compatibly split subvarieties of $\tilde{D_1}$ are contained in $\nu^{-1}(D_1\cap E)\cup\nu^{-1}((D_1)_{\textrm{non-R1}})$ or $\textrm{sing}(\tilde{D_1})$.
\end{enumerate}
\end{proof}

\begin{proposition}
Suppose $X$ is normal and Frobenius split by $\sigma^{p-1}$ where $\sigma\in H^0(X,\omega_X^{-1})$ and $D = V(\sigma)$. Suppose that $D = D_1+\cdots + D_r$ where the support of each $D_i$ is irreducible and non-empty. Let \[P = \{(X',D')~|~(X',D')\textrm{ appears at some stage of the algorithm and }X'\textrm{ is normal}\}.\]  If, for all $(X',D')\in P$, $\textrm{sing}(X')\cap D'$ has codimension $\geq 2$ in each component of $D'$, then, for large $p$, the Knutson-Lam-Speyer algorithm with input $(X,D)$ finds all compatibly split subvarieties of $X$.
\end{proposition}

\begin{proof}
We follow the ideas found in \cite[Theorem 5.3]{KLS}.

Suppose all of the necessary conditions hold, yet there exists some compatibly split $Y\subseteq X$ that is not found by the algorithm. Then there exists $Z$ found by the algorithm with $Y\subseteq Z$ and $\textrm{codim}_YZ$ minimized. $Z$ could have been found by the algorithm in one of two ways. 

The first possiblity is that there exists some $\tilde{Z}$ and some finite surjective morphism $f:\tilde{Z}\rightarrow Z$ such that $(\tilde{Z},D_{\tilde{Z}})$ is a pair arising from step 4. of the algorithm. Note that we may assume that $\tilde{Z}$ is normal. By repeatedly applying Lemma \ref{l;adj}, we see that $D_{\tilde{Z}}$ is anticanonical in $\tilde{Z}$. By Speyer's theorem, $f^{-1}(Y)$ is compatibly split in $\tilde{Z}$ for large primes $p$. 

If $f^{-1}(Y)\subseteq \overline{\textrm{sing}(\tilde{Z})\setminus (D_{\tilde{Z}}\cap \textrm{sing}(\tilde{Z}))}$ then the algorithm finds $Y$, a contradiction. So suppose otherwise. By Lemma \ref{l;adj}, $f^{-1}(Y)\subseteq D_{\tilde{Z}}$. Thus, $Y\subseteq f(D_1)$ for some component $D_1$ of $D_{\tilde{Z}}$. This contradicts the minimality assumption.  

Next suppose that there is no such $\tilde{Z}$ such that $f:\tilde{Z}\rightarrow Z$ is a finite surjective map and $(\tilde{Z},D_Z)$ is a pair appearing at some stage of the algorithm. Then, there must exist some $Z'$ found by the algorithm with $Z\subseteq Z'$ such that (i) there exists a finite map $f:\tilde{Z'}\rightarrow Z'$, (ii) $(\tilde{Z'},D')$ is a pair in appearing at some stage of the algorithm, and (iii) $f^{-1}(Z)$ is a compatibly split subvariety contained inside of $\overline{\textrm{sing}(\tilde{Z'})\setminus D'}$. In this case, $f^{-1}(Y)$ is also a subvariety contained in $\overline{\textrm{sing}(\tilde{Z'})\setminus D'}$. This contradicts the fact that $Y$ was not found by the algorithm.
\end{proof}

\begin{remark}
Katzman and Schwede also have an algorithm for finding all compatibly split subvarieties of a Frobenius split scheme $X$ (see \cite{KS}). 
\end{remark}

\subsection{Multigraded/torus-invariant splittings}

Our reference for this subsection is \cite{K4}.

\begin{definition}
Let $R$ be multigraded by a lattice $\Lambda$ of rank $r$. (I.e. $R = \bigoplus_{\Lambda}R_{\lambda}$.) Identify $\Lambda$ with $\bZ^r$ and $\lambda$ with $(\lambda_1,\dots,\lambda_r)$, $\lambda_i\in \bZ$. Suppose $R$ is Frobenius split by $\phi$. We call $\phi$ a \textbf{multigraded splitting} if $\phi(R_{\lambda}) \subseteq R_{\lambda/p}$, which is $0$ unless $p|\lambda_i$ for each $1\leq i\leq r$. 
\end{definition}

\begin{example}
Suppose that $k[x_1,\dots,x_n]$ comes with a multigrading (weighting) and that, with respect to that multigrading, $f\in k[x_1,\dots,x_n]$ is homogeneous. If $\textrm{Tr}(f^{p-1}\cdot)$ is a splitting, then $\textrm{Tr}(f^{p-1}\cdot)$ is a multigraded splitting. (Proof: Because $\textrm{Tr}(f^{p-1}\cdot)$ is a splitting, there is a term of $f^{p-1}$ of the form $x_1^{p-1}\cdots x_n^{p-1}$ (else $\textrm{Tr}(f^{p-1}1)\neq 1$). Because $f$ is homogeneous, so is $f^{p-1}$. Thus, every term of $f^{p-1}$ has the same weight as $x_1^{p-1}\cdots x_n^{p-1}$. Letting $\lambda = (\lambda_1,\dots,\lambda_r)$ denote the weight of $x_1\cdots x_n$, we see that every term of $f^{p-1}$ has weight $(p-1)\lambda$.

Now suppose that $m$ is a monomial such that $\textrm{Tr}(f^{p-1}m)\neq 0$. Then for at least one monomial $m'$ in $f^{p-1}$, $x_1\cdots x_nm'm$ is a $p^{th}$ power. Let $\mu$ denote the weight of $m$. Notice that $x_1\cdots x_nm'$ has weight $p\lambda$. Thus, $p|\mu$.) 
\end{example}

\begin{remark}
Since a torus action on $R$ induces a multigrading and vice versa, the above definition could have alternatively been formulated in terms of a torus action. Thus, a multigraded splitting is also called a \textbf{torus-invariant splitting}. In the above example, we see that a $(p-1)^{st}$ power splitting of affine space is torus-invariant whenever the compatibly split anticanonical divisor is torus-invariant.
\end{remark}

\begin{proposition}\cite[See proof of Corollary 2]{K4}
Suppose a torus $T$ acts on $X$ and suppose that $X$ is Frobenius split by the $T$-invariant splitting $\phi$. Then, all compatibly split subvarieties of $X$ are stable under the action of $T$.
\end{proposition}

\begin{remark}
This result also follows from the Knutson-Lam-Speyer algorithm in the setting where steps 1, 3, and 4, are enough to find all compatibly split subvarieties.
\end{remark}

\subsection{A $T^2$-invariant Frobenius splitting of $\Hilbn$}\label{splitHilb}
The material in this subsection comes from \cite{KT} and \cite[Section 1.3 and Chapter 7]{BK}.

\begin{definition}
\begin{enumerate}
\item A normal variety $Y$ is \textbf{Gorenstein} if its canonical sheaf $\omega_Y$ is invertible. 
\item Given a Gorenstein variety $Y$ and a normal variety $X$, a proper, birational morphism $f:X\rightarrow Y$ is called \textbf{crepant} if $f^*\omega_Y = \omega_X$.
\end{enumerate}
\end{definition}

\begin{lemma}\label{l;crep1}\cite[Lemma 1.3.13]{BK}
Let $f:X\rightarrow Y$ be a crepant morphism. If $Y$ is Frobenius split then so is $X$.
\end{lemma}

\begin{theorem}\label{t;crep2}\cite[Special case of Theorem 1]{KT}
Let $k$ be an algebraically closed field of characteristic $p>2$. Let $S^n(\bA^2_k)_*$ denote the open locus of the Chow variety with at least $n-1$ distinct points. Let $\Hilbn_*$ denote the preimage of $S^n(\bA^2_k)_*$ under the Hilbert-Chow morphism $\Psi$. Then, $\Psi: \Hilbn_*\rightarrow S^n(\bA^2_k)_*$ is a crepant resolution.
\end{theorem}

Kumar and Thomsen also prove the following:

\begin{theorem}\label{t;crep3}\cite[Corollary 1]{KT}
Let $p>n$. Then $S^n(\bA^2_k)$ is Gorenstein and $\Psi: \Hilbn\rightarrow S^n(\bA^2_k)$ is a crepant resolution.
\end{theorem}

Theorem \ref{t;crep2} can be used to prove that $\Hilbn$ is Frobenius split. 

\begin{theorem}\cite[Special case of Theorem 2]{KT}
Let $k$ be an algebraically closed field of characteristic $p>2$. The standard splitting of $\bA^2_k$ induces a Frobenius splitting of $\Hilbn$. 
\end{theorem}

We sketch the proof found in \cite[Section 5]{KT}.

\begin{proof}
Consider $\bA^2_k$ with the standard splitting $\phi = \textrm{Tr}((xy)^{p-1}\cdot)$. There is then a natural choice of splitting on $(\bA^2)^n$ given by $\phi^{\boxtimes n} = \textrm{Tr}((x_1y_1\cdots x_ny_n)^{p-1}\cdot)$. Notice that $\phi^{\boxtimes n}$ takes $S_n$-invariant functions to $S_n$-invariant functions. Thus, $\phi^{\boxtimes n}$ induces a Frobenius splitting of $S^n(\bA^2_k)$.

Let $V\subseteq S^n(\bA^2_k)$ denote the open set where none of the $n$ points collide. Because $S^n(\bA^2_k)$ is Frobenius split, and $V$ is open, $V$ has an induced splitting. Let $\sigma'$ be the associated splitting section of $\omega^{1-p}_V$. Since $V$ is codimension-$2$ in $S^n(\bA^2)_*$ and $S^n(\bA^2_k)_*$ is normal, $\sigma'$ can be extended to a section $\sigma$ of $\omega^{1-p}_{S^n(\bA^2_k)_*}$. Let $\Psi$ denote the Hilbert-Chow morphism. Because $\Psi$ is crepant, $\Psi^*(\sigma) = \tilde{\sigma}$ is a section of $\omega^{1-p}_{\Hilbn_*}$. Since the Hilbert scheme is smooth and $\Hilbn_{*}$ has codimension-$2$ in $\Hilbn$, $\tilde{\sigma}$ can be extended to a section of $\omega^{1-p}_{\Hilbn}$. This is a splitting section because its restriction to $\Psi^{-1}(V)\cong V$ is $\sigma'$ and $\sigma'$ is a splitting section.
\end{proof}


\section{Moment polyhedra}

\subsection{Definitions and a few results}

Our main references for this subsection are \cite{BrionFrench}, \cite{Brion}, \cite{K3}, and \cite{Sja}. 

We use the following notation throughout this subsection: Let $k$ be a field, let $R$ be a $k$-algebra, and let $\bG_m$ be the multiplicative group $\Spec k[x,x^{-1}]$. Suppose that $R$ is an $\bN$-graded Noetherian domain and that $T := \bG_m^r$ acts algebraically on $R$ with weights $\lambda_1,\dots,\lambda_r\in \Lambda :=\textrm{Hom}(T,\bG_m)$. Recall that $\Lambda\cong \bZ^r$. Since $R$ is bigraded by weight and degree, we may write \[R = \bigoplus_{(\lambda,s)\in\Lambda\times \bN}R_{\lambda,s}.\]

\begin{definition}
The \textbf{moment polyhedron} $P(R)$ of $R$ is defined to be the following subset of $\Lambda\otimes_{\bZ}\bQ$: \[P(R) := \left\{\frac{\lambda}{s}~\bigg|~\lambda\in\Lambda,~s\in \bN, \textrm{ and } R_{(\lambda, s)}\neq 0\right\}.\]
\end{definition}

\begin{theorem}
The moment polyhedron $P(R)$ is a convex polyhedron. If $\textrm{Proj}(R)$ (with the $\bN$-grading) is projective over $k$, then $P$ is a convex polytope.
\end{theorem}

\begin{example}\label{e;momentpoly}
Let $R = k[x,w_0,w_1]$ where $x$ has degree $0$ and both $w_0$ and $w_1$ have degree $1$. Then $\textrm{Proj}(R) = \bA^1_k\times \bP^1_k$. Let $T^2$ act on $R$ with weights $(1,0)$, $(0,0)$, and $(0,1)$. Suppose that $x^{\alpha}w_0^{\beta}w_1^{\gamma}\in R_{(\lambda,s)}$. This element has $T^2$-weight $\lambda = (\alpha,\gamma)$ and $\bN$-weight $s = \beta+\gamma$. Thus, $\frac{\lambda}{s} = (\frac{\alpha}{\beta+\gamma},\frac{\gamma}{\beta+\gamma})$. Notice that $\frac{\alpha}{\beta+\gamma}$ can be any non-negative rational number and that $\frac{\gamma}{\beta+\gamma}\in [0,1]\cap\bQ$. Thus, $P(R)$ is the following polyhedron:

\begin{figure}[!h]
\begin{center}
\includegraphics[scale = 0.34]{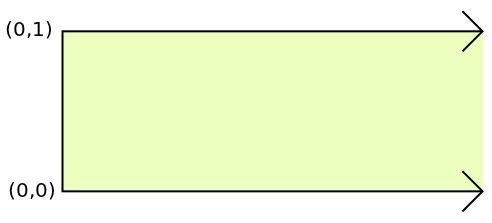}
\end{center}
\end{figure}
\end{example}

\begin{remarks}
\begin{enumerate}
\item The sources mentioned at the beginning of this subsection work over $\bC$. However, this assumption is not necessary to define $P(R)$; we need only require that $\Lambda$, the weight lattice of $T$, is isomorphic to $\bZ^r$ (where $r$ is the rank of the torus). Taking $T = \bG_m^r$ satisfies this requirement, even when working over a field of characteristic $p>0$.
\item We don't really discuss \textbf{moment maps} in this thesis. However, it is worthwhile to note that when it makes sense to define a moment map $\psi:\textrm{Proj}(R)\rightarrow \frak{t}^*$, i.e. over $\bC$, the polyhedron $P(R)$ consists of the rational points in the image of the moment map. 
\end{enumerate}
\end{remarks}

If $J$ is a $T$-homogeneous and $\bN$-homogeneous prime ideal of $R$, then $R/J$ also has a moment polyhedron, $P(R/J)$. Notice that $P(R/J)$ is supported inside of $P(R)$. Drawing the moment polyhedra of all $T$-invariant subvarieties of $\textrm{Proj}(R)$ in the same picture yields the \textbf{x-ray polyhedron} of $R$. (See \cite[Section 2]{Tol} for the original definition.) From here on, we usually draw the whole x-ray polyhedron rather than just the moment polyhedron. In particular, we draw the x-rays of $\Hilbn$, $n\leq 4$, in Subsection \ref{s;xraysHilb}.

\begin{definition}
Let the torus $T$ act algebraically on $R$ and let $P(R)$ denote the corresponding x-ray. Let $Q$ be a subpolytope of $P(R)$ (i.e. $Q$ is the moment polytope of some $T$-invariant subvariety of $\textrm{Proj}(R)$). Define $Q^{\perp}\subseteq T$ as follows: \[Q^{\perp}:=\bigcap_{q_1, q_2\in Q, m\in \bZ,\atop  m(q_1-q_2)\in \Lambda}\textrm{ker}~m(q_1-q_2).\]
\end{definition}

\begin{example}
Let $Q$ be the vertical edge in the moment polyhedron in Example \ref{e;momentpoly}. Suppose $q_1, q_2\in Q$, $m\in \bZ$, and $m(q_1-q_2)\in \Lambda$. Then, $m(q_1-q_2)$ is a group homomorphism $T^2\rightarrow \bG_m$ of the form $(t_1,t_2)\mapsto t_1^0t_2^n$ for some $n\in\bZ$. (I.e. $m(q_1-q_2)\in \Lambda$ is identified with $(0,n)\in \bZ^2$.) Then, \[Q^{\perp} = \{(t,1)~|~t\in\bG_m\}.\]
\end{example}

\begin{proposition}
Any subvariety $Y\subseteq \textrm{Proj}(R)$ whose moment polytope is a subpolytope of $Q$ is pointwise fixed by $Q^{\perp}$.
\end{proposition}

\begin{definition}\label{d;bone}
Let $X = \textrm{Proj}(R)$. Define the \textbf{bones over }$\mathbf{Q}$ to be the maximal components of $X^{Q^{\perp}}$ whose moment polytopes lie inside $Q$. 
\end{definition}

Knutson suggested that the term ``bone'' be used in place of having to refer to ``preimages under the moment map''. This substitution of language is necessary when we have an x-ray (hence the term ``bone'') but no moment map (eg. when working over a field of characteristic $p>0$). Note that we use the term ``bone'' throughout Subsection \ref{s;nleq5}, which concerns the compatibly split subvarieties of $\Hilbn$, $n\leq 5$. 

Before stating the final proposition of this subsection, we recall some definitions/facts about polytopes.

\begin{definition}(See, for example, the notes \cite{Goe}.)
A \textbf{halfspace} is a set of the form \[\{\mu\in \Lambda\otimes_{\bZ}\bQ~|~\langle \mu,\alpha\rangle\leq \beta, \textrm{ for some } \alpha\in \Lambda^* \textrm{ and some }\beta\in \bQ\}.\] A \textbf{polyhedron} is the intersection of finitely many half spaces. The inequality $\langle \alpha, \mu\rangle \leq \beta$, $\alpha\in \Lambda^*$ and $\beta\in \bQ$, is called a \textbf{valid inequality} for the polyhedron $P$ if $\langle \alpha, \mu\rangle \leq \beta$ for every $\mu\in P$. A \textbf{face} of a polyhedron $P$ is $\{\mu\in P~|~\langle \alpha, \mu\rangle = \beta\}$ where $\langle \alpha, \mu\rangle \leq \beta$ is a valid inequality of $P$.
\end{definition}

\begin{proposition}\label{p;extedge}(Knutson)
Let $F$ be a face of the moment polytope $P(R)$ and suppose that $F = \{\mu\in P(R)~|~\langle \alpha, \mu\rangle = \beta\}$ where $\langle \alpha, \mu\rangle \leq \beta$ is a valid inequality of $P(R)$.
\begin{enumerate}
\item The direct sum \[I = \bigoplus_{n\in \bN}\bigoplus_{\mu\in P(R)\atop \langle \alpha, \mu\rangle <\beta}R_{n\mu,n}\] is an $\bN$- and $T$-homogeneous prime ideal of $R$. $\textrm{Proj}(R/I)$ is the bone of $F$.
\item Suppose that $R$ is a $k$-algebra, where $k$ has characteristic $p>0$. Suppose further that $R$ has an $\bN$- and $T$-invariant Frobenius splitting. Then, the bone of $F$ is compatibly split. 
\end{enumerate}
\end{proposition}

\begin{proof}
\begin{enumerate}
\item That $I$ is an ideal follows because $\langle \alpha, \mu\rangle \leq \beta$ is a valid inequality of $P(R)$. Indeed, suppose $r\in R$ and $i\in I$. Assume that $r\in R_{\lambda_1,n_1}$ and $i\in R_{\lambda_2,n_2}$. Then, \[\left\langle \alpha, \frac{\lambda_1+\lambda_2}{n_1+n_2}\right\rangle = \frac{n_1}{n_1+n_2}\left\langle \alpha, \frac{\lambda_1}{n_1}\right\rangle + \frac{n_2}{n_1+n_2}\left\langle \alpha, \frac{\lambda_2}{n_2}\right\rangle.\] The first summand on the right side of the equation is $\leq \frac{n_1}{n_1+n_2}\beta$ and the second summand is $<\frac{n_2}{n_1+n_2}\beta$. Thus, the sum is $<\beta$ and we see that $ri\in I$. 

A nearly identical argument proves that $I$ is prime: Suppose that $ab\in I$. Assume that $a\in R_{\lambda_1,n_1}$ and $b\in R_{\lambda_2,n_2}$. Then $ab\in R_{\lambda_1+\lambda_2,n_1+n_2}$. Thus, $\langle \alpha, \frac{\lambda_1+\lambda_2}{n_1+n_2}\rangle<\beta$, and so at least one of $\langle \alpha, \frac{\lambda_1}{n_1}\rangle$ or $\langle \alpha, \frac{\lambda_2}{n_2}\rangle$ must also be $<\beta$ (see the above equation). It follows that at least one of $a$ or $b$ is in $I$.

That $I$ is $\bN$- and $T$-homogeneous, and that $\textrm{Proj}(R/I)$ is the bone of $F$ follows by the construction of $I$. 

\item Let $\phi:R\rightarrow R$ be an $\bN$- and $T$-invariant Frobenius splitting of $R$. Then, \[\phi(I) = \bigoplus_{n\in \bN}\bigoplus_{\mu\in P(R)\atop \langle \alpha, \mu\rangle <\beta}\phi(R_{n\mu,n})\subseteq \bigoplus_{\frac{n}{p}\in \bN}\bigoplus_{\mu\in P(R)\atop \langle \alpha, \mu\rangle <\beta}R_{\frac{n\mu}{p},\frac{n}{p}}\subseteq I.\]
\end{enumerate}
\end{proof}

\subsection{Constructing the moment polyhedron of $\Hilbn$}\label{s;xraysHilb}

In this subsection, we discuss the moment polyhedron (really the x-ray) of $\Hilbn$. 

\begin{proposition}
Let $\lambda$ be a $T^2$-fixed point of $\Hilbn$. Let $S_{\lambda}$ denote its standard set. Then, $\lambda$'s moment polytope is the point \[\sum_{(i,j)\in S_\lambda}(i,j).\]
\end{proposition}

Using this proposition, along with the $T^2$-weights of the tangent space $T_{\lambda}\Hilbn$ (see the Hilbert scheme background section), we can construct the x-ray of $\Hilbn$. 

\begin{figure}[!h]
\begin{center}
\includegraphics[scale = 0.43]{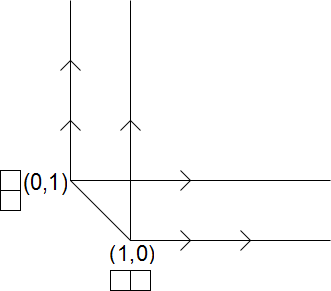}
\caption{The x-ray polyhedron of $\textrm{Hilb}^2(\bA^2_k)$. The arrows are meant to denote $T^2$-weights.}
\end{center}
\end{figure} 


\begin{figure}[!h]
\begin{center}
\includegraphics[scale = 0.70]{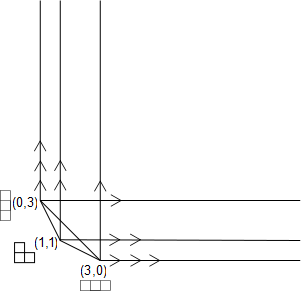}
\caption{The x-ray polyhedron of $\textrm{Hilb}^3(\bA^2_k)$}
\label{fig;hilb3polytope}
\end{center}
\end{figure} 

\begin{figure}[!h]
\begin{center}
\includegraphics[scale = 0.60]{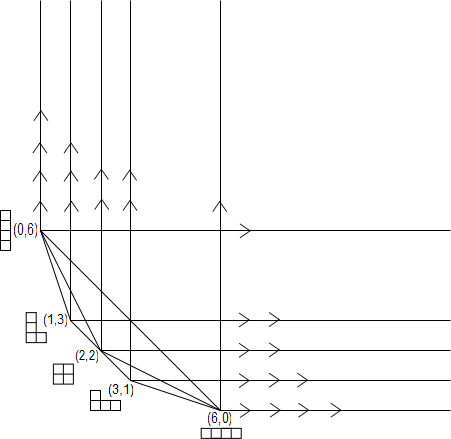}
\caption{The x-ray polyhedron of $\textrm{Hilb}^4(\bA^2_k)$}
\label{fig;hilb4polytope}
\end{center}
\end{figure}

\newpage
Though we never use the moment map, we provide it here anyway: Let $\frak{t} = \textrm{Lie}((S^1)^2)$. There is a moment map $\Phi: \textrm{Hilb}^n(\mathbb{C}^2)\rightarrow \frak{t}^{*}$. It is given by \[I\rightarrow \sum_{i,j\in\mathbb{N}\cup\{0\}} d_{ij}(i,j)\] where (under the standard Hermitian inner product), the rank $n$ orthogonal projection of $\mathbb{C}[x,y]$ to $I^{\bot}$ is $x^iy^j\mapsto d_{ij}x^iy^j+\cdots$. This is derived from the inclusion \[\Hilbn\hookrightarrow \textrm{Gr}^n(\bC[x,y])\hookrightarrow \bP(\wedge^n\bC[x,y]^*).\]


\section{Gr\"{o}bner bases, degenerations, and some combinatorics}
\subsection{Gr\"{o}bner bases and Gr\"{o}bner degenerations}\label{s;gb}

Our main source for this section is \cite[Chapter 15]{Eis}.

\begin{definition}
\begin{enumerate}
\item A \textbf{monomial order} (or \textbf{term order}) $>$ of the polynomial ring $k[x_1,\dots,x_n]$ is a total ordering of monomials such that, if $n$ is a (non-constant) monomial and $m_1>m_2$, then $nm_1>nm_2>m_2$. 
\item Let $f\in k[x_1,\dots,x_n]$. The \textbf{initial term} of $f$, denoted $\textrm{init}_{>}(f)$, is the largest term appearing in $f$ with respect to the given monomial order. 
\item Let $I\subseteq k[x_1,\dots,x_n]$ be an ideal. The \textbf{initial ideal} is the ideal $\textrm{init}_>(I) := \langle \textrm{init}_>(f)~|~f\in I\rangle$.
\end{enumerate}
\end{definition}

\begin{theorem}
Let $>$ be a monomial order on $k[x_1,\dots,x_n]$ and let $I$ be an ideal of $k[x_1,\dots,x_n]$. The monomials which are not in $\textrm{init}(I)$ form a $k$-vector space basis for the quotient $k[x_1,\dots,x_n]/I$.
\end{theorem}

Notice that we have already made use of this theorem when discussing an open cover of the Hilbert scheme.

\begin{example}
One common monomial order is the \textbf{Lex order}: Consider the polynomial ring $k[x_1,\dots,x_n]$ with variables ordered by $x_1>x_2>\cdots>x_n$. Then, \[x_1^{a_1}\cdots x_n^{a_n}>_{\textrm{Lex}} x_1^{b_1}\cdots x_n^{b_n}\] if $a_i>b_i$ for the \emph{first} index $i$ where $a_i\neq b_i$.
\end{example}

Often, we'll be concerned with total orderings of the monomials of $k[x_1,\dots,x_n]$ which only satisfy the condition, ``if $n$ is a (non-constant) monomial and $m_1>m_2$, then $nm_1>nm_2$''. In particular, unlike in the definition of monomial order given above, we don't require $1$ to be smaller than any non-constant monomial. For example, in Section 2.3 we will be concerned with the \textbf{Revlex} ordering (defined below).

\begin{example}
Consider the polynomial ring $k[x_1,\dots,x_n]$ with variables ordered by $x_1>x_2>\cdots>x_n$. Then, \[x_1^{a_1}\cdots x_n^{a_n}>_{\textrm{Revlex}} x_1^{b_1}\cdots x_n^{b_n}\] if $a_i<b_i$ for the \emph{last} index $i$ where $a_i\neq b_i$.
\end{example}

Often, weights are used to order monomials. 

\begin{definition}
Assign integers (i.e. \textbf{weights}) $\lambda_1,\dots,\lambda_n$ to the variables $x_1,\dots,x_n$. Then the weight of $x_1^{a_1}\cdots x_n^{a_n}$ is $\lambda_1a_1+\cdots \lambda_na_n$, and \[x_1^{a_1}\cdots x_n^{a_n}>x_1^{b_1}\cdots x_n^{b_n} ~~\textrm{if}~~ \lambda_1a_1+\cdots \lambda_na_n>\lambda_1b_1+\cdots +\lambda_nb_n.\]
\end{definition}

\begin{remark}
Two different monomials in $k[x_1,\dots,x_n]$ may have the same weight, and so weights alone are not enough to distinguish between all monomials. Therefore, if $>$ denotes a weight order on $k[x_1,\dots,x_n]$, $\textrm{init}_{>}f$ may consist of multiple terms for some polynomials $f\in k[x_1,\dots x_n]$. In this setting, we call $\textrm{init}(f)$ the \textbf{initial form} of $f$.
\end{remark}

\begin{example}
For a fixed ideal $I$, we can define weight orders $>_1$ and $>_2$ so that \[\textrm{init}_{>_1}(I) = \textrm{init}_{\textrm{Lex}}(I)~~ \textrm{and}~~ \textrm{init}_{>_2}(I) = \textrm{init}_{\textrm{Revlex}}(I).\] The weight order $>_1$ is obtained by weighting $x_1,x_2, \dots,x_n$ by $N_1,N_2,\dots,N_n$ where $N_1\gg N_2\gg\cdots \gg N_n>0$. The weight order $>_2$ is obtained by weighting $x_1,x_2,\dots, x_n$ by $N_1,N_2,\dots,N_n$ where $-N_n\gg \cdots \gg -N_2\gg-N_1>0$.
\end{example}

\begin{definition}
Fix a monomial order $>$ on $k[x_1,\dots,x_n]$. Let $I\subseteq k[x_1,\dots,x_n]$ be an ideal. 
\begin{enumerate}
\item A set of polynomials $G = \{g_1,\dots,g_r\}\subseteq I$ is called a \textbf{Gr\"{o}bner basis} of $I$ with respect to the monomial order $>$ if $\textrm{init}_>I = \langle \textrm{init}_>(g)~|~g\in G\rangle$.
\item A Gr\"{o}bner basis $G$ is \textbf{reduced} if (i) for each $g_i\in G$, $\textrm{init}(g_i)$ does not divide any term of $g_j$ for $i\neq j$, and (ii) the coefficient of each $\textrm{init}(g_i)$ is $1$.
\item Let $G$ be a set of polynomials $\{g_1,\dots,g_r\}$ (not necessarily a Gr\"{o}bner basis). Let \[m_{i,j} = \frac{\textrm{init}(g_i)}{\textrm{GCD}(\textrm{init}(g_i),\textrm{init}(g_j))}.\] The \textbf{S-polynomial} (or \textbf{S-pair}) of $g_i$ and $g_j$ is $S(g_i,g_j):=m_{j,i}g_i-m_{i,j}g_j$.
\end{enumerate}
\end{definition}

\begin{proposition}(Buchberger's criterion)
Consider a set of polynomials $G =\{g_1,\dots,g_r\}$. Apply the division algorithm with respect to $G$ to each $S$-pair $S(g_i,g_j)$ to obtain the remainder $h_{i,j}$. $G$ is a Gr\"{o}bner basis if and only if each $h_{i,j}$ is $0$.
\end{proposition}

Buchberger's criterion can be turned into an algorithm (\textbf{Buchberger's algorithm}) for producing a Gr\"{o}bner basis from a given generating set $G = \{g_1,\dots,g_r\}$: Apply the division algorithm with respect to $G$ to each $S$-pair $S(g_i,g_j)$ to obtain the remainder $h_{i,j}$. Suppose that some of these $h_{i,j}$ are non-zero. Consider a new set \[G' = \{g_1,\dots,g_r\}\cup\{h_{i,j}~|~h_{i,j}\neq 0\}.\] Apply Buchberger's criterion to $G'$. If $G'$ is not a Gr\"{o}bner basis, add in the new non-zero remainders. Repeat. This process must terminate after finitely many steps.

We make use of the next lemma when discussing degenerations of the compatibly split subvarieties of $U_{\langle x,y^n\rangle}$.

\begin{lemma}\label{l;init}

Consider two ideals $I, J\subset k[x_1,\dots,x_n]$. Suppose that $k[x_1,\dots,x_n]$ comes with a given monomial order. If $I\subset J$ and $\textrm{init}(I) = \textrm{init}(J)$ then $I = J$.
\end{lemma}

\begin{proof}
Let $\{g_1,\dots,g_r\}$ be a Gr\"{o}bner basis for $I$. Let $f\in J$ be given. Because both $I\subset J$ and $\textrm{init}(I) = \textrm{init}(J)$ we may reduce $f$ to $0$ using the given Gr\"{o}bner basis of $I$. Thus, $f\in I$.  
\end{proof}

We end this subsection by recalling the relationship between Gr\"{o}bner bases and flat families.

\begin{proposition}
Let $\{\lambda_1,\dots,\lambda_n\}$ be weights of $\{x_1,\dots,x_n\}$ and let $\lambda(m)$ denote the weight of the monomial $m$. Let $I\subseteq k[x_1,\dots,x_n]$ be an ideal and let $G = \{g_1,\dots,g_r\}$ be a Gr\"{o}bner basis with respect to the given weight order. 

For each $g_i\in G$, write $g_i = \sum c_im_i$ where $m_i$ is a monomial and $c_i\in k^*$. Let $b = \textrm{max}~\lambda(m_i)$ Replace each $g_i\in G$ by \[\tilde{g_i} = t^b g(t^{-\lambda_1}x_1,\dots, t^{-\lambda_r}x_r).\] Let $\tilde{I} = \langle \tilde{g_i}~|~g_i\in G\rangle$. 
\begin{enumerate}
\item $k[x_1,\dots,x_n][t]/\tilde{I}$ is free and thus flat as a $k[t]$-module.
\item $k[x_1,\dots,x_n][t]/((t)+\tilde{I})\cong S/\textrm{init}(I)$. 
\end{enumerate}
Thus, $k[x_1,\dots,x_n][t]/\tilde{I}$ is a flat family over $k[t]$ of quotients of $k[x_1,\dots,x_n]$. The fiber over $0$ is $k[x_1,\dots,x_n]/\textrm{init}(I)$.
\end{proposition}

This type of flat family is called a \textbf{Gr\"{o}bner degeneration}.

\subsection{Lex and Revlex degenerations}\label{s;lexRev}

In this subsection we review the geometry of Gr\"{o}bner degenerations with respect to the Lex and Revlex weightings. We begin with Lex. For this, our references are \cite{KMY} and \cite{Kfpsac}.

\begin{definition}
Let $f\in k[x_1,\dots,x_n]$. Weight variable $x_i$ by $N> 0$ and weight each of the rest of the variables by $0$. Define $\textrm{Lex}_{x_i}(f)$ to be the leading form of $f$ with respect to the given weighting.

If $X$ is a subscheme of $\bA^n_k$ with defining ideal $I$, define \[\textrm{Lex}_{x_i}X := \Spec\left(\frac{k[x_1,\dots,x_n]}{\textrm{Lex}_{x_i}I}\right).\]
\end{definition}

Treating $\bA^n_k = \Spec(k[x_1,\dots,x_n])$ as a vector space $V$, we can decompose it as $V = H\oplus L$, where $H$ is the hyperplane $\{x_i=0\}$ and $L = \bA^1_{x_i}$ is the $x_i$-axis. There is a $k^*$ action on $V$ given by $t\cdot(\vec{h},l) = (\vec{h},tl), t\in k^*$. We say that this action ``scales L and fixes H''.

Let $X$ be a closed subscheme of $V$. Let $X'$ denote the flat limit $\textrm{lim}_{t\rightarrow 0}(t\cdot X)$. Notice that this limit is just an alternate definition of the Gr\"{o}bner degeneration $\textrm{Lex}_{x_i}X$ defined above. The following theorem explains the geometry of $\textrm{Lex}_{x_i}X$.

\begin{theorem}\cite[Theorem 2.2]{KMY}
Let $X$ be a closed subscheme of $V$ and let $L$ be a 1-dimensional subspace of V. Denote by $\Pi$ the scheme-theoretic closure of the image of $X$ in $V/L$, and by $\overline{X}$ the closure of $X$ in $\textrm{Bl}_{\bP L}\bP(V\oplus \bA^1_k)$ (the blow-up of $\bP(V\oplus \bA^1_k)$ at the point $\bP{L}$). Set $\Gamma = \overline{X} \cap (V/L)$, where the intersection of schemes takes place in $\textrm{Bl}_{\bP L}\bP(V\oplus \bA^1_k)$.

If $H$ is a hyperplane complementary to $L$ in $V$, and we identify $H$ with $V/L$, then the flat limit $X' := \textrm{lim}_{t\rightarrow 0} t \cdot X$ under scaling $L$ and fixing $H$ satisfies \[X' \supseteq (\Pi \times \{0\}) \cup (\Gamma\times L),\] with equality as sets. 
\end{theorem}

\begin{remarks}
\begin{enumerate}
\item The authors of \cite{KMY} call the decomposition $(\Pi \times \{0\}) \cup (\Gamma\times L)$ a ``geometric vertex decomposition'' as it is reminiscent of a vertex decomposition of a simplicial complex. (See the next subsection for relevant definitions.)
\item In the case that we deal with in this thesis, we get scheme-theoretic equality in the above theorem because all subschemes in question are Frobenius split and thus reduced.
\end{enumerate}
\end{remarks}

The Revlex case is simpler. 

\begin{definition}
Let $f\in k[x_1,\dots,x_n]$. Weight variable $x_i$ by $-N$ where $N> 0$. Weight each of the rest of the variables by $0$. Define $\textrm{Revlex}_{x_i}(f)$ to be the leading form of $f$ with respect to the given weighting.

If $X$ is a subscheme of $\bA^n_k$ with scheme theoretic defining ideal $I$, define \[\textrm{Revlex}_{x_i}X := \Spec\left(\frac{k[x_1,\dots,x_n]}{\textrm{Revlex}_{x_i}I}\right).\]
\end{definition}

\begin{lemma}\label{l;revlex}
Let $Y$ be an integral subscheme of $\Spec k[x_1,...,x_n]$. Let $H$ be the hyperplane $\{x_1 = 0\}$ and let $L = \bA^1_{x_1}$ be the $x_1$-axis. 
\begin{itemize}
\item If $Y\subseteq H$ then $\textrm{Revlex}_{x_1}(Y) = Y\times O_{x_1}~~\subseteq~~ H\times L$.
\item If $Y\nsubseteq H$ then $\textrm{Revlex}_{x_1}(Y) = (Y\cap H)\times L~~\subseteq~~ H\times L$.
\end{itemize}
\end{lemma}

\begin{proof}
If $Y\subseteq H$ then $x_1\in I(Y)$ and $\textrm{Revlex}_{x_1}I(Y) = I(Y)$.

Now suppose that $Y\nsubseteq H$. Then $x_1\notin I(Y)$. As $I(Y)$ is prime, we have that if $x_1$ divides any $g\in I(Y)$ then $\frac{g}{x_1}\in I(Y)$. So, $\textrm{Revlex}_{x_1}I(Y) = \langle h(0,x_2,x_3,\dots,x_n) | h\in I(Y)\rangle$. Thus, \[\frac{k[x_1,\dots,x_r]}{\textrm{Revlex}_{x_1}I(Y)}\cong \frac{k[x_1,\dots,x_r]}{I(Y)+\langle x_1\rangle}\otimes k[x_1]\] as desired.
\end{proof}

\subsection{Stanley-Reisner schemes and simplicial complexes}

Our references for this subsection are the textbooks \cite{MS}, \cite{St} unless otherwise indicated.

\begin{definition}
A \textbf{simplicial complex} $\Delta$ on the vertex set $\{v_1,\dots,v_n\}$ is a collection of subsets, called \textbf{faces} (or \textbf{simplices}), with the following property: if $\sigma\in \Delta$, and $\tau \subseteq \sigma$, then $\tau\in \Delta$. We call a maximal face a \textbf{facet}. If $\sigma$ is an $d+1$ element subset of $\{v_1,\dots,v_n\}$ then we say that $\sigma$ has dimension $d$. If all facets have the same dimension $d$, the we say that $\Delta$ is \textbf{pure} of dimension $d$.
\end{definition}

We take particular interest in \textbf{vertex-decomposable} simplicial complexes.

\begin{definition}\cite{BP}
Let $\Delta$ be a simplicial complex on the vertex set $V$. Let $v\in V$. Define the following subcomplexes:
\begin{enumerate}
\item $\textrm{del}(v):=\{F\in \Delta~|~F\cup\{v\}\notin \Delta\}$.
\item $\textrm{link}(v):=\{F\in \Delta~|~v\notin F,~F\cup\{v\}\in \Delta\}$.
\item $\textrm{star}(v):=\{F\in \Delta~|~F\cup\{v\}\in \Delta\}$.
\end{enumerate}
\end{definition}

Notice that $\textrm{star}(v)$ is a cone on $\textrm{link}(v)$ and that \[\Delta = \textrm{del}(v)\cup_{\textrm{link}(v)}\textrm{star}(v).\]

\begin{definition}\cite{BP}
Let $\Delta$ be a pure $d$-dimensional simplicial complex on the vertex set $V$. We say $\Delta$ is \textbf{vertex-decomposable} if either $\Delta$ is empty or there exists a vertex $v\in V$ such that
\begin{enumerate}
\item $\textrm{del}(v)$ is pure, $d$-dimensional, and vertex-decomposable, and
\item $\textrm{link}(v)$ is pure, $(d-1)$-dimensional, and vertex-decomposable.
\end{enumerate}
\end{definition}

As mentioned in \cite{Kfpsac}, a particularly nice case is when (i) $\textrm{del}(v)$ is homeomorphic to a $d$-dimensional ball and (ii) $\textrm{link}(v)$ is homeomorphic to a $(d-1)$-dimensional ball (or sphere) on the spherical surface of $\textrm{del}(v)$. Then, $\Delta$ too is homeomorphic to a $d$-dimensional ball (or sphere). 

\begin{definition}
Let $\Delta$ be a pure $d$-dimensional simplicial complex. We say that $\Delta$ is \textbf{shellable} if there exists an ordering of the facets $\sigma_1,\dots,\sigma_s$ such that for $i=2,\dots,s$, \[\overline{\sigma_i}\cap (\bigcup_{j=1}^{i-1}\overline{\sigma_j})\] is a pure $(d-1)$-dimensional complex. Note that $\overline{\sigma_i}$ denotes the smallest simplicial complex with facet $\sigma_i$.
\end{definition}

\begin{theorem}\cite{BP}
Vertex-decomposable simplicial complexes are shellable.
\end{theorem}

To every simplicial complex on the vertex set $\{v_1,\dots,v_n\}$, we can associate a squarefree monomial ideal $I\in k[x_1,\dots,x_n]$. 

\begin{definition}
Let $\Delta$ be a simplicial complex on vertex set $V = \{v_1,\dots,v_n\}$. Let $\sigma\subseteq V$ and define $\mathbf{x}^{\sigma} = \prod_{v_i\in\sigma}x_i$. The \textbf{Stanley-Reisner ideal} $I_{\Delta}\subseteq k[x_1,\dots,x_n]$ is the ideal \[I_{\Delta} := \langle \mathbf{x}^{\sigma}~|~\sigma\notin \Delta\rangle.\] The \textbf{Stanley-Reisner ring} associated to $\Delta$ (also called the \textbf{face ring}) is the ring $R_{\Delta}:=k[x_1,\dots,x_n]/I_{\Delta}$. The \textbf{Stanley-Reisner scheme} associated to $\Delta$ is $\Spec~R_{\Delta}$.
\end{definition}

Notice that $I_{\Delta}$ is actually generated by all $\mathbf{x}^{\sigma}$ such that $\sigma$ is a \emph{minimal} non-face of $\Delta$.

\begin{theorem}
The correspondence $\Delta\mapsto I_{\Delta}$ is a bijection from simplicial complexes on the vertices $\{v_1,\dots,v_n\}$ to squarefree monomial ideals in $k[x_1,\dots,x_n]$. Furthermore, if $\overline{\sigma} = \{v_1,\dots,v_n\}\setminus \sigma$, and $\frak{m}^{\overline{\sigma}} = \langle x_i~|~v_i\in \overline{\sigma}\rangle$, then, \[I_{\Delta} = \bigcap_{\sigma\in\Delta}\frak{m}^{\overline{\sigma}}.\]
\end{theorem}

\begin{example}\label{e;SRexample}
Let $\Delta\subseteq \{v_1,\dots,v_6\}$ be the simplicial complex with facets $\{v_3,v_4\}$, $\{v_1,v_2\}$, $\{v_1,v_3\}$, $\{v_2,v_3,v_5\}$. The geometric realization of $\Delta$ is shown in Figure \ref{fig;geomDelta}.

\begin{figure}[!h]
\begin{center}
\includegraphics[scale = 0.45]{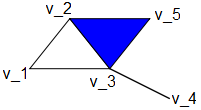}
\caption{The geometric realization of $\Delta$ in Example \ref{e;SRexample}.}
\label{fig;geomDelta}
\end{center}
\end{figure}
The Stanley-Reisner ideal of $\Delta$ is generated by the minimal non-faces of $\Delta$. Thus, \[I_\Delta = \langle x_6, x_1x_4, x_1x_5, x_2x_4, x_4x_5, x_1x_2x_3 \rangle.\]
On the other hand, given $I_\Delta$, we can decompose to obtain \[I_{\Delta} = \langle x_1,x_2,x_5,x_6\rangle\cap \langle x_3,x_4,x_5,x_6\rangle \cap \langle x_2,x_4,x_5,x_6\rangle\cap\langle x_1,x_4,x_6\rangle.\] Thus, the facets of $\Delta$ are $\{v_3,v_4\}$, $\{v_1,v_2\}$, $\{v_1,v_3\}$, and $\{v_2,v_3,v_5\}$.
\end{example}

Sometimes properties of a simplicial complex translate to properties of the associated Stanley-Reisner scheme. 

\begin{definition}
A simplicial complex $\Delta$ is \textbf{Cohen-Macaulay} if its Stanley-Reisner ring is Cohen-Macaulay.
\end{definition}

\begin{theorem}(Hochster)
A shellable simplicial complex is Cohen-Macaulay.
\end{theorem}

In particular, a vertex-decomposable simplicial complex is Cohen-Macaulay. We make use of this later.

\begin{remark}
One can view the geometric realization of the simplicial complex $\Delta$ as the ``moment polytope'' of $\textrm{Proj}(R_{\Delta})$ (which needn't actually be a polytope since $R_{\Delta}$ isn't a domain).
\end{remark}


\chapter{Compatibly split subvarieties of the Hilbert scheme of points in the plane}

\section{Preliminary remarks}\label{s;premarks}

Let $k$ be an algebraically closed field of characteristic $p>2$. Fix the $T^2$-invariant Frobenius splitting of $\Hilbn$ induced from the standard splitting of $\bA^2_k$ (see Subsection \ref{splitHilb}). In this section, we begin investigating the question, ``What are all compatibly split subvarieties of $\Hilbn$?'' We note that the $n=1$ case is trivial as $\Hilb^1(\bA^2_k)\cong \bA^2_k$ with the standard splitting. Indeed, as we saw in Example \ref{e;stdsplitting}, there are four compatibly split subvarieties of $\bA^2_k$ with the standard splitting: $\bA^2_k$, the $y$-axis, the $x$-axis, and the origin. Identifying $(x,y)\in \bA^2_k$ with the location of the one point in $\Hilb^1(\bA^2_k)$, we see that the compatibly split subvarieties of $\Hilb^1(\bA^2_k)$ can be described by ``the point is in $\bA^2_k$'', ``the point is on the $y$-axis'', ``the point is on the $x$-axis'', and ``the point is at the origin''.

\begin{remark} 
Throughout the rest of the thesis, when referring to closed points of $\Hilbn$, we usually mean colength-$n$ ideals of $k[x,y]$, rather than the corresponding length-$n$ subschemes of $\bA^2_k$.
\end{remark}

\subsection{A partial description of all compatibly split subvarieties}\label{s;partDescription}

In this subsection, we gain some intuition into our problem through Proposition \ref{p;intuition}. It asserts that finding all compatibly split subvarieties of $\Hilbn$ amounts to finding all compatibly split subvarieties $Z\subseteq \Hilb^{m}(\bA^2_k)$, $m\leq n$, with the property that $Z\subseteq \Hilb^m_0(\bA^2_k)$. (Recall that $\Hilb^m_0(\bA^2_k)$, the punctual Hilbert scheme, parametrizes length-$m$ subschemes supported at the origin.)

\begin{proposition}\label{p;intuition}
For $p>n!$, $Y\subseteq \Hilbn$ is a compatibly split subvariety if and only if $Y$ is the closure, in $\Hilbn$, of the (set-theoretic) image of a map 
\[i_{r,s,t,Z}:\Hilb^r(\textrm{punctured x-axis})\times \Hilb^s(\textrm{punctured y-axis})\times \Hilb^t(\bA^2_k\setminus \{xy=0\})\times Z\rightarrow \Hilbn\]
\[(I_1,I_2,I_3,I_4)\mapsto I_1\cap I_2\cap I_3\cap I_4\] 
where $Y$ is given the reduced induced subscheme structure and where 
\begin{enumerate}
\item each of the punctured axes have been punctured at the origin, 
\item $r,s,t$ are non-negative integers with $r+s+t\leq n$, and
\item $Z\subseteq \Hilb^{n-r-s-t}(\bA^2_k)$ is a compatibly split (irreducible) subvariety contained in $\Hilb_0^{n-r-s-t}(\bA_k^2)$.
\end{enumerate} 
\end{proposition}

\begin{lemma}\label{l;morph}
Let $U := \{(I,J)\in\Hilb^a(\bA^2_k)\times\Hilb^b(\bA^2_k)~|~I+J = k[x,y]\}$. Then $U$ is an open subscheme of $\Hilb^a(\bA^2_k)\times \Hilb^b(\bA^2_k)$, and \[j_{a,b}:U\rightarrow \Hilb^{a+b}(\bA^2_k),~~(I,J)\mapsto I\cap J\] is a well-defined morphism. Also, $j_{a,b}(U)$ is open in $\Hilb^{a+b}(\bA^2_k)$.
\end{lemma}

\begin{proof}
Let $R =k[x_1,y_1,\dots,x_a,y_a,u_1,v_1,\dots,u_b,v_b]$ and consider $\Spec R = (\bA^2_k)^a\times(\bA^2_k)^b$. Let 
\begin{displaymath}
I = \bigcap_{1\leq i\leq a \atop 1\leq j\leq b}\langle x_i-u_j, y_i-v_j\rangle.
\end{displaymath}
The product of symmetric groups $S_a\times S_b$ acts on $(\bA^2_k)^a \times(\bA^2_k)^b$ by 
\[(\sigma,\tau)\cdot(x_1,y_1,\dots,x_a,y_a,u_1,v_1,\dots,u_b,v_b) = (x_{\sigma(1)},y_{\sigma(1)},\dots,x_{\sigma(a)},y_{\sigma(a)},u_{\tau(1)},v_{\tau(1)},\dots,u_{\tau(b)},v_{\tau(b)}).\]
The quotient $Y = \Spec((R/I)^{S_a\times S_b})$ is a closed subvariety of $S^a(\bA^2_k)\times S^b(\bA^2_k)$. Let $Y'$ denote the preimage of $Y$ (under the product of the appropriate Hilbert-Chow morphisms) in $\Hilb^a(\bA^2_k)\times \Hilb^b(\bA^2_k)$. $U$ is the complement of $Y'$ and so is an open subscheme of $\Hilb^a(\bA^2_k)\times\Hilb^b(\bA^2_k)$. 

Next notice that, since $I$ and $J$ are coprime, $I\cap J$ is an element of $\Hilb^{a+b}(\bA^2_k)$. Therefore $j_{a,b}$ is well-defined as a map of sets. We now show that $j_{a,b}(U)$ is an open subscheme of $\Hilb^{a+b}(\bA^2_k)$. To begin, let $Y\subseteq S^a(\bA^2_k)\times S^b(\bA^2_k)$ be as above. Because $\pi: S^a(\bA^2_k)\times S^b(\bA^2_k)\rightarrow S^{a+b}(\bA^2_k)$ is a finite morphism, $\pi(Y)$ is a closed subscheme of $S^{a+b}(\bA^2_k)$. Its complement $S^{a+b}(\bA^2_k)\setminus \pi(Y)$ is open and thus the preimage of $S^{a+b}(\bA^2_k)\setminus \pi(Y)$ (under the appropriate Hilbert-Chow morphism) is open in $\Hilb^{a+b}(\bA^2_k)$. This preimage is $j_{a,b}(U)$ and so $j_{a,b}(U)$ is an open subscheme of $\Hilb^{a+b}(\bA^2_k)$. 

Finally, we show that $j_{a,b}$ is a morphism by finding an open cover of $U$ such that the restriction of $j_{a,b}$ to each set in the open cover is a morphism. 

Suppose that $I\in U_{\lambda}\subseteq \Hilb^a(\bA^2_k)$ and $J\in U_{\lambda'}\subseteq \Hilb^b(\bA^2_k)$. Then, $I$ is generated by \[\{x^ry^s-\sum_{x^hy^l\notin\lambda}c^{r,s}_{h,l}x^hy^l~|~x^ry^s\in \lambda\},\] and $J$ is generated by \[\{x^ry^s-\sum_{x^hy^l\notin\lambda'}d^{r,s}_{h,l}x^hy^l~|~x^ry^s\in\lambda'\}.\] Suppose that $I\cap J\in U_{\lambda''}\subseteq \Hilb^{a+b}(\bA^2_k)$. Then $I\cap J$ is generated by \[\{x^ry^s-\sum_{x^hy^l\notin\lambda''}e^{r,s}_{h,l}x^hy^l~|~x^ry^s\in\lambda''\}.\] Since $I$ and $J$ are coprime, $I\cap J = IJ$ and we see that each $e^{r,s}_{h,l}$ is a rational function in the various $c^{r,s}_{h,l}$ and the various $d^{r,s}_{h,l}$. Suppose $e^{r_0,s_0}_{h_0,l_0} = f/g$ for some polynomials $f$ and $g$ in the various $c^{r,s}_{h,l}$ and $d^{r,s}_{h,l}$. Then, \[gx^{r_0}y^{s_0}-fx^{h_0}y^{l_0}-\sum_{\begin{subarray}{c} x^hy^l\notin\lambda''\\(h,l)\neq(h_0,l_0)\end{subarray}}ge^{r,s}_{h,l}x^hy^l\] is an element of $IJ$. If $g$ ever vanishes, then $IJ\notin U_{\lambda''}$. Thus, $f/g$ is a regular function and $j_{a,b}$ restricted to the open set $U\cap (U_{\lambda}\times U_{\lambda'})\cap j_{a,b}^{-1}(j_{a,b}(U)\cap U_{\lambda''})$ is a morphism. Sets of this form (i.e. of the form $U\cap (U_{\lambda}\times U_{\lambda'})\cap j_{a,b}^{-1}(j_{a,b}(U)\cap U_{\lambda''})$) cover $U$ and so we see that $j_{a,b}$ is a morphism.
\end{proof}

\begin{remarks}
\begin{enumerate}
\item Lemma \ref{l;morph} implies that the set-theoretically defined map $i_{r,s,t,Z}$ in Proposition \ref{p;intuition} is a morphism of schemes. 
\item Since each of the factors in the domain of $i_{r,s,t,Z}$ are irreducible, so is the closure of the image. 
\end{enumerate}
\end{remarks}

\begin{lemma}\label{l;splitmapHilb}
Let $U = \{(I,J)\in\Hilb^n(\bA^2_k)\times\Hilb^1(\bA^2_k)~|~I+J = k[x,y]\}$ and suppose that $p>n+1$. Let $j_{n,1}:U\rightarrow \Hilb^{n+1}(\bA^2_k)$ be the morphism defined in Lemma \ref{l;morph}. Because $j_{n,1}(U)$ is an open subscheme of $\Hilb^{n+1}(\bA^2_k)$, it has an induced splitting. With respect to this induced splitting, $j_{n,1}:U\rightarrow j_{n,1}(U)$ is a split morphism. 
\end{lemma}

\begin{proof}
Let $V\subseteq U$ be the open subscheme where none of the points collide. Then, $j_{n,1}(V)$ is the open subscheme in $j_{n,1}(U)$ where none of the points collide. Give each of $V$ and $j_{n,1}(V)$ the splittings induced from $U$ and $j_{n,1}(U)$ respectively. Let $\sigma_1\in H^0(V,\omega_V^{-1})$ and $\sigma_2\in H^0(j_{n,1}(V), \omega_{j_{n,1}(V)}^{-1})$ be such that $\sigma_1^{p-1}$ and $\sigma_2^{p-1}$ determine these induced splittings. 

To see that $j_{n,1}:U\rightarrow j_{n,1}(U)$ is a split morphism, it suffices to show that $j_{n,1}|_V$ is a split morphism (see Lemma \ref{l;openSplitMorph}). To begin, note the following:
\begin{enumerate}
\item The morphism $j_{n,1}|_V:V\rightarrow j_{n,1}(V)$ is \'{e}tale. (Indeed, $j_{n,1}|_V$ is flat because it is a finite surjective morphism between smooth varieties. That $j_{n,1}|_V$ is unramified follows from the assumption that $p>n+1=\textrm{deg}(j_{n,1}|_V)$.) 
\item There are no non-constant, non-vanishing functions on $V$. (Proof: There are no non-constant, non-vanishing functions on $S^n(\bA^{2}_k)$, $S^n(\bA^2_k)$ is normal, and the complement of $V$ in $S^n(\bA^2_k)$ is codimension-$2$.)
\end{enumerate}
By 1. and Lemma \ref{l;etaleSplit}, $(j_{n,1}|_V)^*(\sigma_2^{p-1})$ is a splitting section of $V$. Furthermore, if $V$ is Frobenius split by $(j_{n,1}|_V)^*(\sigma_2^{p-1})$ and $j_{n,1}(V)$ is Frobenius split by $\sigma_2^{p-1}$, then $j_{n,1}|_V$ a split morphism. Therefore, to obtain the desired result, it suffices to show that $\sigma_1^{p-1} = (j_{n,1}|_V)^*(\sigma_2)^{p-1}$.  This equality holds because of 2. and the fact that $\sigma_1$ and $(j_{n,1}|_V)^*(\sigma_2)$ vanish in the same location (see Lemma \ref{l;nonconstnonvan}).
\end{proof}

\begin{lemma}
If $Y\subseteq \Hilbn$ is the closure of the image of $i_{r,s,t,Z}$, for some $r,s,t,Z$, then $Y$ is compatibly split. 
\end{lemma}

\begin{proof}
If $r=s=t=0$, then the result is automatic. So suppose otherwise and proceed by induction on $n$. When $n=1$ the result is clear. 

Let $\phi_n$ and $\phi_1$ denote the splittings of $\Hilbn$ and $\Hilb^1(\bA^2_k)$ and suppose that $Y_1\subseteq \Hilbn$ and $Y_2\subseteq\Hilb^1(\bA^2_k)$ are each irreducible compatibly split subschemes.  Then, with respect to the splitting $\phi_n\otimes\phi_1$, $Y_1\times Y_2\subseteq \Hilbn\times \Hilb^1(\bA^2_k)$ is compatibly split (see the remark following Example \ref{l;splitprojection}). 

Now, suppose that $\textrm{dim}(Y_2)>0$ (i.e. $Y_2\cong$ x-axis, or $Y_2\cong$ y-axis, or $Y_2\cong\bA^2$). Let $U$ and $j_{n,1}$ be as in Lemma \ref{l;splitmapHilb}. Note that $(Y_1\times Y_2)\cap U\neq \emptyset$. Because $j_{n,1}$ is a split morphism, the closure of $j_{n,1}((Y_1\times Y_2)\cap U)$ in $j_{n,1}(U)$ is compatibly split in $j_{n,1}(U)$. Because $j_{n,1}(U)$ is open in $\Hilb^{n+1}(\bA^2_k)$ and has the induced splitting, the closure of $j_{n,1}((Y_1\times Y_2)\cap U)$ in $\Hilb^{n+1}(\bA^2_k)$ (with the reduced induced scheme structure) is compatibly split.

By induction, we may assume that $Y_1$ is the closure of the image of some $i_{r',s',t',Z}$. Then, the closure of $j_{n,1}((Y_1\times Y_2)\cap U)$ in $\Hilb^{n+1}(\bA^2_k)$ agrees with the closure of the image of $i_{r,s,t,Z}$, where either (i) $r=r'+1$, $s=s'$, $t=t'$, or (ii) $r=r'$, $s=s'+1$, $t=t'$, or (iii) $r=r'$, $s=s'$, $t = t'+1$. It follows that the closure of the image of $i_{r,s,t,Z}$ is compatibly split in $\Hilb^{n+1}(\bA^2_k)$. 
\end{proof}

We now complete the proof of Proposition \ref{p;intuition}. It remains to show that if $Y\subseteq \Hilbn$ is compatibly split then $Y$ must be the closure of the image of a map $i_{r,s,t,Z}$ for some $r,s,t,Z$.

\begin{proof}[Proof of Proposition \ref{p;intuition}]
Let $\pi:(\bA^2_k)^n\rightarrow S^n(\bA^2_k)$ denote the quotient map. Let $p$ be larger than $\textrm{deg}(\pi) = n!$.

By the compatibility of the splittings of $\Hilbn$ and $S^n(\bA^2_k)$, the image of any compatibly split subvariety of $\Hilbn$ is compatibly split in $S^n(\bA^2_k)$. By the compatibility of the splittings of $(\bA^2_k)^n$ and $S^n(\bA^2_k)$, and by Speyer's theorem regarding finite morphisms (see Proposition \ref{MainCase}; this is where we use $p>n!$), $Y\in S^n(\bA^2_k)$ is compatibly split if and only if $\pi^{-1}(Y)$ is compatibly split in $(\bA^2_k)^n$. 

Recall that the set of compatibly split subvarieties of $(\bA^2_k)^n$ with the standard splitting is the set of all coordinate subspaces of $(\bA^2_k)^n$. Thus, the set of compatibly split subvarieties of $S^n(\bA^2_k)$ is \[\left\{\pi(V)~ \bigg|~ V = \bigcup_{\sigma\in S_n}\sigma(S), S\textrm{ a coordinate subspace of}~ (\bA^2_k)^n\right\}\] and so any compatibly split subvariety $Y\subseteq \Hilbn$ must map by the Hilbert-Chow morphism to some $\pi(V)$. Notice that each $\pi(V)$ is the closure of the image, in $S^n(\bA^2_k)$, of a product of the form \[S^r(\textrm{punctured x-axis})\times S^s(\textrm{punctured y-axis})\times S^t(\bA^2_k\setminus \{xy=0\})\times Z'\] where $Z'$ is the cycle $(n-r-s-t)[(0,0)]$.

Now suppose that $Y\subseteq\Hilbn$ is a compatibly split subvariety. Since $Y$ maps by the Hilbert-Chow morphism to some $\pi(V)$, it follows that $Y$ is the closure of the image of the map \[i:\Hilb^r(\textrm{punctured x-axis})\times \Hilb^s(\textrm{punctured y-axis})\times \Hilb^t(\bA^2_k\setminus \{xy=0\})\times Z\rightarrow \Hilbn\] \[(I_1,I_2,I_3,I_4)\mapsto I_1\cap I_2\cap I_3\cap I_4\] for some $r,s,t\geq 0$ with $r+s+t\leq n$ and for some subvariety $Z\subseteq \Hilb^{n-r-s-t}_0(\bA_k^2)$. 

To show that $Z$ must be a \emph{compatibly split} subvariety of $\Hilb^{n-r-s-t}(\bA_k^2)$, we proceed by induction on $n$. When $n = 1$, the result is automatic.

Now suppose that $Y$ is compatibly split in $\Hilb^{n+1}(\bA^2_k)$. If $Y$ is contained in $\Hilb^{n+1}_0(\bA^2_k)$ then we are done. So suppose otherwise. Consider the inclusion $j:\Hilb^{n+1}(\bA^2_k)\hookrightarrow \Hilb^{n+1}(\bP^2_k)$ where $j$ is the map induced by the inclusion $\bA^2_k\hookrightarrow \bP^2_k$, $(x,y)\mapsto (x:y:1)$. Let $\overline{Y}$ denote the closure of $j(Y)$. 

Consider the splitting of $\Hilb^{n+1}(\bP^2_k)$ induced by the standard splitting of $\bP^2_k$ (i.e. the splitting with compatibly split anticanonical divisor $\{xyz = 0\}$). Let $D = D_1\cup D_2\cup D_3$ denote the compatibly split anticanonical divisor in $\Hilb^{n+1}(\bP^2_k)$ and let $D_3$ denote the component which has empty intersection with $\Hilb^{n+1}(\bA^2_k)$ (i.e. a stratum representative of $D_3$ has one point supported on the line $\{z=0\}$). Then $\Hilb^{n+1}(\bA^2_k)$ is an open set in $\Hilb^{n+1}(\bP^2_k)$ (via the inclusion $j$) and has the induced splitting. By \cite[Lemma 1.1.7]{BK} (see Proposition \ref{p;openSplit}), $\overline{Y}$ is compatibly split in $\Hilb^{n+1}(\bP^2_k)$. Thus, $D_3\cap \overline{Y}$ is compatibly split. Because $Y\nsubseteq \Hilb^{n+1}_0(\bA^2_k)$ by assumption, $D_3\cap \overline{Y}$ is non-empty.

Let $U\subseteq \Hilb^{n}(\bP^2_k)\times \Hilb^1(\bP^2_k)$ be defined by \[U:=\{(I,J)\in \Hilb^n(\bP^2_k)\times \Hilb^1(\bP^2_k)~|~\textrm{the schemes associated to }I+J \textrm{ and } I+\langle z\rangle \textrm{ are empty}\}\] and consider the map \[\tilde{j}_{n,1}:U\rightarrow \Hilb^{n+1}(\bP^2_k),~~(I,J)\mapsto I\cap J.\] By a similar argument to the one in Lemma \ref{l;morph}, $U$ is open in $\Hilb^{n}(\bP^2_k)\times \Hilb^1(\bP^2_k)$ and $\tilde{j}_{n,1}(U)$ is open in $\Hilb^{n+1}(\bP^2_k)$. Thus, $U$ and $\tilde{j}_{n,1}(U)$ have induced splittings. By a similar argument to the one in Lemma \ref{l;splitmapHilb}, $\tilde{j}_{n,1}:U\rightarrow \tilde{j}_{n,1}(U)$ is a split morphism.

Notice that (the reduction of) $\tilde{j}_{n,1}^{-1}(\tilde{j}_{n,1}(U)\cap(D_3\cap \overline{Y}))$ is isomorphic to $S\times \Hilb^1(\{z=0\}\subseteq \bP^2_k)$ for some subvariety $S\subseteq \Hilb^{n}(\bA^2_k)$. Furthermore, because the morphism $\tilde{j}_{n,1}$ is finite and split, we may apply Speyer's theorem to see that $S\times \Hilb^1(\{z=0\})$ is compatibly split in $\Hilbn\times \Hilb^1(\{z=0\})$. Finally, because the projection map $\pi_1: \Hilbn\times \Hilb^1(D_3)\rightarrow \Hilbn$ is a split map (see Example \ref{l;splitprojection}), $S\subseteq \Hilbn$ is compatibly split and we are done by induction. 
\end{proof}

\subsection{Stratum representatives}\label{s;stratrep}
We will often describe a compatibly split subvariety $Y$ by a ``stratum representative'' that it contains. 

\begin{definition}
Let $Y$ be a compatibly split subvariety. We say that an element $I\in Y$ is a \textbf{stratum representative} of $Y$ if $I$ is not in any compatibly split subvariety that is properly contained in $Y$.
\end{definition}

By Proposition \ref{p;intuition}, if $Y\subseteq \Hilbn$ is compatibly split, we may always choose a stratum representative $I\in Y$ such that $\Spec(k[x,y]/I)$ has the following form:

\begin{figure}[!h]
\begin{center}
\includegraphics[scale = 0.35]{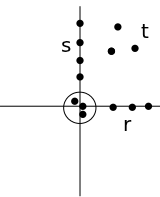}
\end{center}
\end{figure}
\noindent Note that the circled points correspond to a stratum representative of $Z\subseteq \Hilb^{n-r-s-t}_0(\bA^2_k)$. ($Z$ is as in Proposition \ref{p;intuition}.)

\section{Compatibly split subvarieties of $\Hilbn$ for $n\leq 5$}

Throughout this section, let $k$ be an algebraically closed field of (any) characteristic $p>n!$.

\subsection{Running the Knutson-Lam-Speyer algorithm in the $n=2$ case}\label{s;algorithm2pts}

The compatibly split subvarieties of $\Hilbn$ for $n=2,3,4$ can be found using Algorithm \ref{a}. In this subsection, we run the algorithm when $n=2$ and prove that the algorithm finds all compatibly split subvarieties in this case. We use the stratum representative pictures described above to denote the compatibly split subvarieties that arise. In addition we carry through the explicit computations on the open patch $U_{\langle x,y^2\rangle}$. 

Let $\sigma^{p-1}\in H^0(\Hilb^2(\bA^2_k),F_{*}(\omega_{\Hilb^2(\bA^2_k )}^{1-p}))$ determine the torus invariant splitting of $\Hilb^2(\bA^2_k)$ and let $D = V(\sigma)$. 

In Subsection \ref{s;inducedsplitting}, we show that the induced splitting of $U_{\langle x,y^2\rangle}\cong \Spec k[a_1,b_1,a_2,b_2]$ is given by $\textrm{Tr}(f_2^{p-1}\cdot)$ where $f_2 = (a_1b_1a_2-a_1^2+a_2^2b_2)(b_2)$. (So, the intersections of $U_{\langle x,y^2\rangle}$ with each of two components of $D$ appearing below are given by $\{a_1b_1a_2-a_1^2+a_2^2b_2=0\}$ and $\{b_2 = 0\}$ respectively.) To help understand the geometry, let $(x_1,y_1)$ and $(x_2,y_2)$ denote the two (unordered) locations of the points. On the open set where none of the points collide, $a_1 = \frac{x_1y_1-x_2y_1}{y_1-y_2}$, $a_2 = \frac{x_1-x_2}{y_1-y_2}$, $b_1 = y_1+y_2$, and $b_2 = y_1y_2$.

\begin{figure}[!h]
\begin{center}
\includegraphics[scale = 0.25]{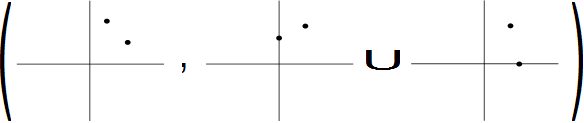}
\caption{$(\Hilb^2(\bA^2_k), D=D_1\cup D_2)$ is the input of the algorithm.}
\end{center}
\end{figure}

$\Hilb^2(\bA^2_k)$ is smooth, so we can skip step 2 of the algorithm. Apply step 3 by intersecting the two components of $D$. Notice that this intersection decomposes into two components. On $U_{\langle x,y^2\rangle}$, the two components of the intersection are given by $\langle a_1,b_2\rangle$ and $\langle a_2b_1-a_1,b_2\rangle$ respectively.

\begin{figure}[!h]
\begin{center}
\includegraphics[scale = 0.25]{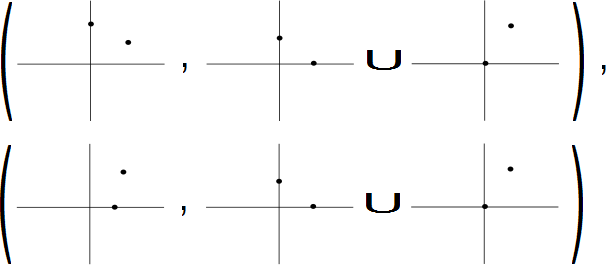}
\caption{Apply step 3 of the algorithm to obtain $(D_1,D_1\cap D_2)$ and $(D_2, D_2\cap D_1)$.}
\label{fig;step3FirstTime}
\end{center}
\end{figure}

By symmetry, we need only continue with the first of the two pairs, which we denote by $(D_1,D_1\cap D_2)$. Notice that $D_1$ is not normal. The singular locus of $D_1$ is the subvariety with a stratum representative where two points are on the y-axis. We can check this explicitly on $U_{\langle x,y^2\rangle}$; the singular locus of $\{a_1b_1a_2-a_1^2+a_2^2b_2=0\}$ is $\{a_1 = a_2 = 0\}$. 

Let $\widetilde{D_1}$ denote the normalization of $D_1$. A stratum representative for $\widetilde{D_1}$ consists of one \emph{labelled} point on the $y$-axis and one point in $\bA^2_k\setminus\{xy=0\}$. Indeed, the corresponding closed subvariety of the isospectral Hilbert scheme (the scheme of labelled points in the plane) is normal, maps finitely to $D_1$, and maps isomorphically away from the preimage of the singular locus of $D_1$. [We can see this explicitly as follows: As seen in \cite[Proposition 3.4.2]{H4}, the isospectral Hilbert scheme $X_2$ of $2$ labelled points in the plane is given by $\textrm{Proj}(k[x_1,x_2,y_1,y_2](tI_2))$ where $I_2 = \langle x_1-x_2,y_1-y_2\rangle$ and $(x_1,y_1)$ and $(x_2,y_2)$ are the (now ordered) locations of the two points. The subvariety $\{x_1=0\}$ has homogeneous coordinate ring $k[x_1,x_2,y_1,y_2,w_0,w_1]/\langle (y_1-y_2)w_0+x_2w_1\rangle$ where $x_1,x_2$, $y_1,y_2$ have degree $0$ and $w_0$, $w_1$ have degree $1$. This ring is integrally closed. So the subvariety $\{x_1=0\}$ is projectively normal. Futhermore, because $X_2/S_2 = \Hilb^2(\bA^2_k)$ (see \cite{H1}), $\{x_1=0\}$ maps finitely (via the quotient map) onto $D_1$. This map is an isomorphism away from the preimage of the (irreducible and codimension-$1$) singular locus of $D_1$. Thus, $\{x_1=0\}$ is isomorphic to $\widetilde{D_1}$, the normalization of $D_1$.] 

Working in the open patch $U_{\langle x,y^2\rangle}$, we see that the integral closure of $k[a_1,b_1,a_2,b_2]/\langle a_1b_1a_2-a_1^2+a_2^2b_2\rangle$ is $k[w,a_1,b_1,a_2,b_2]/I$ where $I = \langle wa_2-a_1, w^2-wb_1-b_2\rangle$. Computing the relevant preimages (proper transforms), we see that the subvarieties appearing below are determined by the ideals $\langle a_1, a_2, w^2-wb_1-b_2\rangle$, $\langle b_2,a_1,w\rangle$, and $\langle b_2, a_2b_1-a_1, w-b_1\rangle$ respectively.

\begin{figure}[!h]
\begin{center}
\includegraphics[scale = 0.25]{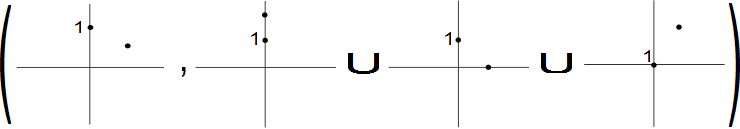}
\caption{Apply step 4 of the algorithm to $(D_1,D_1\cap D_2)$ appearing in Figure \ref{fig;step3FirstTime}.}
\label{fig;figStep3}
\end{center}
\end{figure}

Intersect each component of the divisor in Figure \ref{fig;figStep3} with each of the other two components to obtain the pairs in Figure \ref{fig;figStep3Final}. The remaining computations in the open patch $U_{\langle x,y^2\rangle}$ are straightforward so we omit them.

\begin{figure}[!h]
\begin{center}
\includegraphics[scale = 0.25]{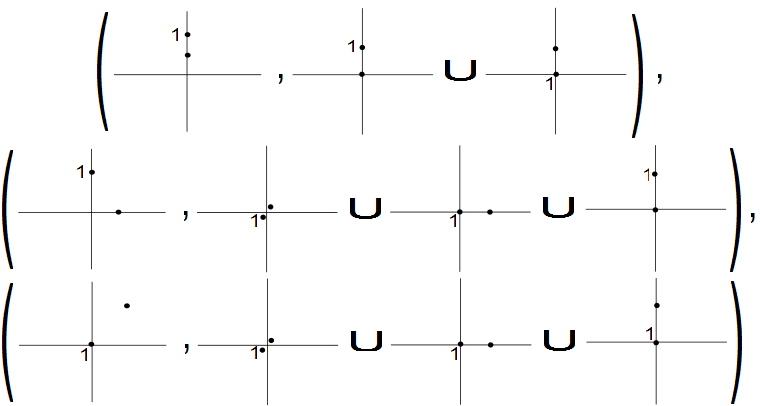}
\caption{Apply step 3 once more. Note that the picture where two points are at the origin denotes a stratum representative of $\textrm{Hilb}^2_0(\mathbb(\mathbb{A})^2_k)\cong \mathbb{P}^1_k$.}
\label{fig;figStep3Final}
\end{center}
\end{figure}

Iterating the steps of the algorithm once more obtains the preimage of the $T^2$-fixed points in the isospectral Hilbert scheme. Mapping all subvarieties forward to $\textrm{Hilb}_2(\mathbb{A}_k^2)$ yields the stratification shown in Figure \ref{fig;hilb2}. (Note that all defining ideals of the compatibly split subvarieties of $U_{\langle x,y^2\rangle}$ can be found in Figure \ref{fig;posetcomplex}.) 

 \begin{figure}[!h]
 \begin{center}
 \includegraphics[scale = 0.32]{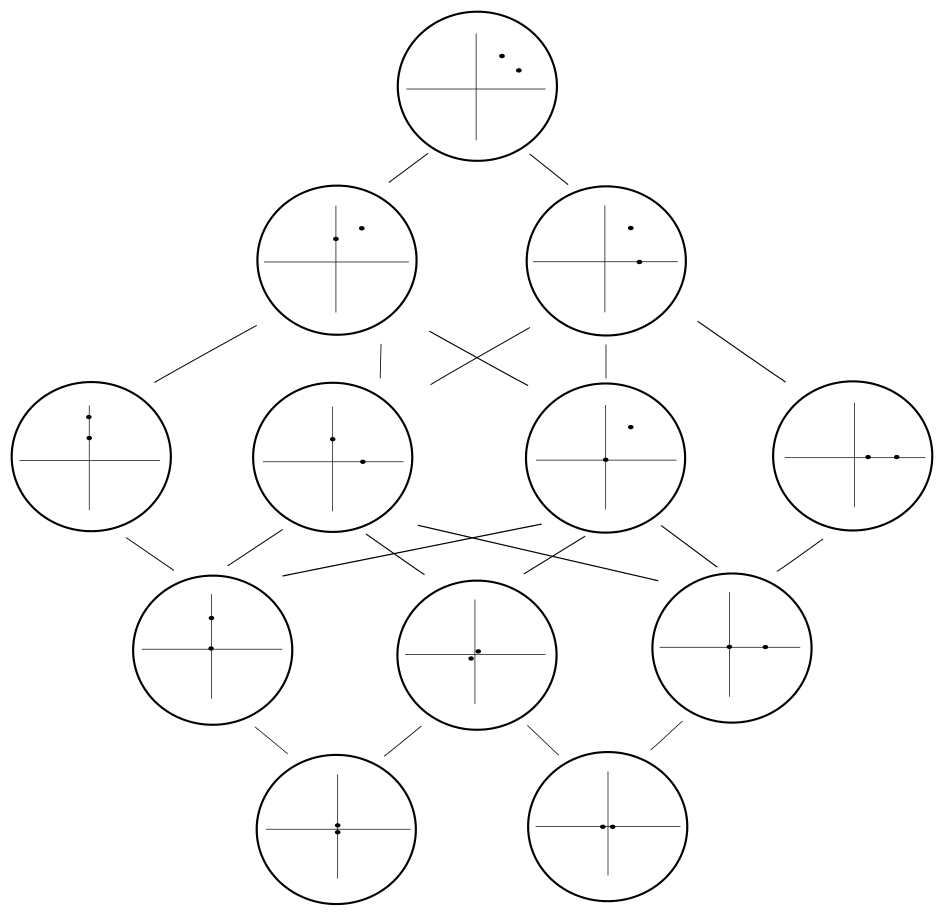}
 \caption{The compatibly split subvarieties of $\Hilb^2(\bA^2_k)$}
 \label{fig;hilb2}
 \end{center}
 \end{figure}

\begin{proposition}
The Knutson-Lam-Speyer algorithm finds all compatibly split subvarieties of the Hilbert scheme of $2$ points in the plane. 
\end{proposition}

\begin{proof}
By Proposition \ref{p;intuition} it suffices to check that the algorithm finds all compatibly split subvarieties contained inside of $\Hilb^2_0(\bA^2_k)$. To do so, notice that all $T^2$-invariant subvarieties of $\Hilb^2_0(\bA^2_k)$ (i.e. $\Hilb^2_0(\bA^2_k)$, $\{\langle x,y^2\rangle\}$ and $\{\langle x^2,y\rangle\}$) are the ``bones'' (see Definition \ref{d;bone}) of exterior faces of the moment polyhedron of $\Hilb^2(\bA^2_k)$. Thus, by Proposition \ref{p;extedge}, the set of torus invariant subvarieties contained in $\Hilb^2_0(\bA^2_k)$ is equal to the set of compatibly split subvarieties contained in $\Hilb^2_0(\bA^2_k)$. 
\end{proof}

\begin{remark}\label{halfSplit}
Some compatibly split subvarieties of $\Hilb^2(\bA^2_k)$ have splittings which are not $(p-1)^{st}$ powers. For example, let $D_1$ be as above (i.e. a stratum representative of $D_1$ has one point on the y-axis). Recall that $D_1$ is not normal. Let $\nu:\tilde{D_1}\rightarrow D_1$ denote the normalization, and let $Y$ denote the preimage, under $\nu$, of the (codimension-$1$ part of the) singular locus of $D_1$. Notice that $\nu|_Y$ is generically 2:1 but is ramified along the locus where the two points collide. Letting $s = y_1+y_2$ and $m = y_1y_2$ be the two coordinates on $\nu(Y)\cong\mathbb{A}^2_k/S_2$, we can check that the splitting of $\nu(Y)$ is given by the section $(s^2-4m)^{(p-1)/2}m^{p-1}$. This is not a $(p-1)^{st}$ power. Notice that $\{m=0\}$ is compatibly split. We would like to say that $\{s^2=4m\}$ (which agrees with the ramification locus of $\nu|_Y$) is ``half split'', but we don't have a general definition to give.
\end{remark}

\subsection{Compatibly split subvarieties of $\Hilbn$ for $n\leq 5$}\label{s;nleq5}

In this subsection, we describe all compatibly split subvarieties of $\Hilb^n(\bA^2_k)$ for $n\leq 4$, as well as provide a conjectural list of all compatibly split subvarieties of $\Hilb^5(\bA^2_k)$. We prove that this conjectural list is correct up to the possible inclusion of one particular one-dimensional subvariety of $\Hilb^5(\bA^2_k)$, and we show that this particular one-dimensional subvariety is not compatibly split for at least those primes $p$ satisfying $2 <p \leq 23$. 

By Proposition \ref{p;intuition}, we need only describe those compatibly split subvarieties which are contained inside the punctual Hilbert scheme, $\Hilb^n_0(\bA^2_k)$. We have already covered the cases of $n=1,2$, so we begin with $n=3$.

\begin{proposition}\label{p;3pts}
Let $\tableau{{ \ }\\{ \ }\\{ \ }}~,$ $\tableau{{ \ }\\{ \ }&{ \ }}~,$ and $\tableau{{ \ }&{ \ }&{ \ }}$ denote the moment polytopes of the $T^2$-fixed points $\langle x,y^3\rangle$, $\langle x^2,xy,y^2\rangle$, and $\langle x^3,y\rangle$. The compatibly split subvarieties of $\Hilb^3(\bA^2_k)$ which are contained inside the punctual Hilbert scheme $\Hilb^3_0(\bA^2_k)$ are precisely those $Y$ appearing in the following list.
\begin{enumerate}
\item $Y$ is the bone (see Definition \ref{d;bone}) of the edge in the moment polyhedron of $\Hilb^3(\bA^2_k)$ connecting $\tableau{{ \ }\\{ \ }\\{ \ }}$ and $\tableau{{ \ }\\{ \ }&{ \ }}~.$ That is, $Y=\{\langle y^3,xy,x^2,ay^2+bx\rangle~|~[a,b]\in \bP^1_k\}$.  
\item $Y$ is the bone of the edge in the moment polyhedron of $\Hilb^3(\bA^2_k)$ connecting $\tableau{{ \ }\\{ \ }&{ \ }}$ and $\tableau{{ \ }&{ \ }&{ \ }}~.$ That is, $Y=\{\langle y^2,xy,x^3,ax^2+by\rangle~|~[a,b]\in \bP^1_k\}$.
\item $Y$ is one of the $T^2$-fixed points (i.e. $Y = \langle x,y^3\rangle$, $Y = \langle x^2,xy,y^2\rangle$, or $Y = \langle x^3,y\rangle$).
\end{enumerate}
\end{proposition}

\begin{figure}[!h]
\begin{center}
\includegraphics[scale = 0.4]{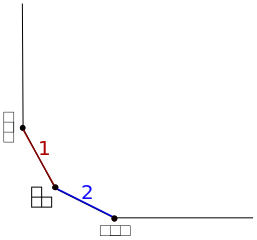}
\caption{The two diagonal edges are the moment polytopes mentioned in items 1. and 2. of Proposition \ref{p;3pts}. The three vertices are the moment polytopes of the three $T^2$-fixed points of $\Hilb^3(\bA^2_k)$.}
\end{center}
\end{figure} 

Before proving this proposition, we consider a few helpful lemmas. 

\begin{lemma}\label{l;fixedPointCover}
Suppose $S$ is a $T^2$-invariant closed subvariety of $\Hilb^n_0(\bA^2_k)$. Let $\{\lambda_1,\dots,\lambda_r\}$ be the (non-empty) set of fixed points that $S$ contains. Then $S\cap U_{\lambda_i}$, $1\leq i\leq r$, is non-empty and $S = (S\cap U_{\lambda_1})\cup\cdots\cup (S\cap U_{\lambda_k})$, an open affine cover.
\end{lemma}

\begin{proof}
Suppose that $\{\lambda_1,\dots,\lambda_r\}$ is the set of fixed points of $S$. Since $\lambda_i\in U_{\lambda_i}$, we see that $S\cap U_{\lambda_i}\neq\emptyset$. 

Next suppose that $S \supsetneq (S\cap U_{\lambda_1})\cup\cdots\cup (S\cap U_{\lambda_k})$. Then $S\cap U_\lambda\neq\emptyset$ for some $\lambda\notin \{\lambda_1,\dots,\lambda_r\}$. Let $I\in (S\cap U_\lambda)\setminus [(S\cap U_{\lambda_1})\cup\cdots\cup (S\cap U_{\lambda_k})]$. As $U_{\lambda}$ is $T^2$-invariant, the entire $T^2$-orbit of $I$ is contained in $U_{\lambda}$. As each $U_{\lambda_i}$, $1\leq i\leq r$, is $T^2$-invariant, the orbit has trivial intersection with each $U_{\lambda_i}$, $1\leq i\leq r$. Futhermore, since $S$ is $T^2$-invariant and projective, the orbit closure of $I$ contains a fixed point $\lambda_i\in S$. Thus, the orbit of $I$ has non-trivial intersection with any open set containing $\lambda_i$; in particular it has non-trivial intersection with $U_{\lambda_i}$, a contradiction. 
\end{proof}

\begin{lemma}
If $Y$ is the bone of the edge in the moment polyhedron of $\Hilb^3(\bA^2_k)$ connecting $\tableau{{ \ }\\{ \ }\\{ \ }}$ and $\tableau{{ \ }\\{ \ }&{ \ }}~,$ then $Y=\{\langle y^3,xy,x^2,ay^2+bx\rangle~|~[a,b]\in \bP^1_k\}$. If $Y$ is the bone of the edge in the moment polyhedron of $\Hilb^3(\bA^2_k)$ connecting $\tableau{{ \ }\\{ \ }&{ \ }}$ and $\tableau{{ \ }&{ \ }&{ \ }}~,$ then $Y=\{\langle y^2,xy,x^3,ax^2+by\rangle~|~[a,b]\in \bP^1_k\}$.
\end{lemma}

\begin{proof}
Let $Y$ denote the bone of the edge connecting $\tableau{{ \ }\\{ \ }\\{ \ }}$ and $\tableau{{ \ }\\{ \ }&{ \ }}~.$ Since the vector from $\tableau{{ \ }\\{ \ }\\{ \ }}$ to $\tableau{{ \ }\\{ \ }&{ \ }}$ is $(1,-2)$ (see Subsection \ref{s;xraysHilb}), $Y\subseteq \Hilb^3_0(\bA^2_k)$ is pointwise fixed by the subtorus $T^1 = \{(t^2,t)~|~t\in\bG_m\}$. Next, by Lemma \ref{l;fixedPointCover}, $Y = (Y\cap U_{\langle x,y^3\rangle})\cup(Y\cap U_{\langle x^2,xy,y^2\rangle})$. All elements of $Y\cap U_{\langle x,y^3\rangle}$ which are pointwise fixed by $T^1$ have the form $\langle y^3,xy, x-\lambda y^2\rangle$, $\lambda\in k$ and all elements of $Y\cap U_{\langle x^2,xy,y^2\rangle}$ which are pointwise fixed by $T^1$ have the form $\langle x^2,xy,y^2-\mu x\rangle$, $\mu\in k$. Gluing these two copies of $\bA^1_k$ yields $\bP^1_k = \{\langle y^3,xy,x^2,ay^2+bx\rangle~|~[a,b]\in \bP^1_k\}$. Similar reasoning shows that the bone of the edge connecting $\tableau{{ \ }\\{ \ }&{ \ }}$ and $\tableau{{ \ }&{ \ }&{ \ }}$ is as claimed.
\end{proof}

We now state Lemma \ref{l;puncSplit} which is proved in Section 2.3. It is useful for showing that there cannot be any compatibly split subvarieties contained in $\Hilb^3_0(\bA^2_k)$ which are not listed in Proposition \ref{p;3pts}.

\newtheorem*{thm:associativity}{Lemma \ref{l;puncSplit}}
\begin{thm:associativity}
Let $Y\subseteq \Hilbn$ be a compatibly split subvariety that has non-trivial intersection with $U_{\langle x,y^n\rangle}$. If $Y$ is contained inside the punctual Hilbert scheme $\Hilb^n_0(\bA^2_k)$, then $Y$ is either the $0$-dimensional subvariety $\{\langle x,y^n\rangle\}$ or the $1$-dimensional subvariety that is pointwise fixed by the subtorus $T^1 = \{(t^{n-1},t)~|~t\in \bG_m\}$.
\end{thm:associativity}

We are now ready to prove Proposition \ref{p;3pts}.

\begin{proof}[Proof of Proposition \ref{p;3pts}]
First notice that each of the subvarieties mentioned in the proposition is the bone of some (exterior) face of the moment polyhedron of $\Hilb^3(\bA^2_k)$ (see Figure \ref{fig;hilb3polytope}). Thus, all subvarieties listed in the statement of the proposition are compatibly split. 

Next, note that the moment polytope of each compatibly split subvariety of $\Hilb^3(\bA^2_k)$ which is contained inside of $\Hilb^3_0(\bA^2_k)$ must be a subpolytope of the moment polytope of $\Hilb^3_0(\bA^2_k)$. In Figure \ref{fig;hilb3polytope} we see that the moment polytope of $\Hilb^3_0(\bA^2_k)$ is the triangle with vertices $\tableau{{ \ }\\{ \ }\\{ \ }}~,$ $\tableau{{ \ }\\{ \ }&{ \ }}~,$ and $\tableau{{ \ }&{ \ }&{ \ }}~.$ Thus, to show that there are no additional compatibly split subvarieties contained in $\Hilb^3_0(\bA^2_k)$, it is enough to show that the bone of the edge connecting $\tableau{{ \ }\\{ \ }\\{ \ }}$ and $\tableau{{ \ }&{ \ }&{ \ }}$ is not compatibly split. Indeed, if the bone of this edge is not compatibly split, then neither is $\Hilb^3_0(\bA^2_k)$ since the bone of any (exterior) face of the moment polytope of a compatibly split subvariety is compatibly split. The bone of the edge connecting $\tableau{{ \ }\\{ \ }\\{ \ }}$ and $\tableau{{ \ }&{ \ }&{ \ }}$ is not compatibly split by Lemma \ref{l;puncSplit}.
\end{proof}

\begin{figure}[!h]
\begin{center}
\includegraphics[scale = 0.32]{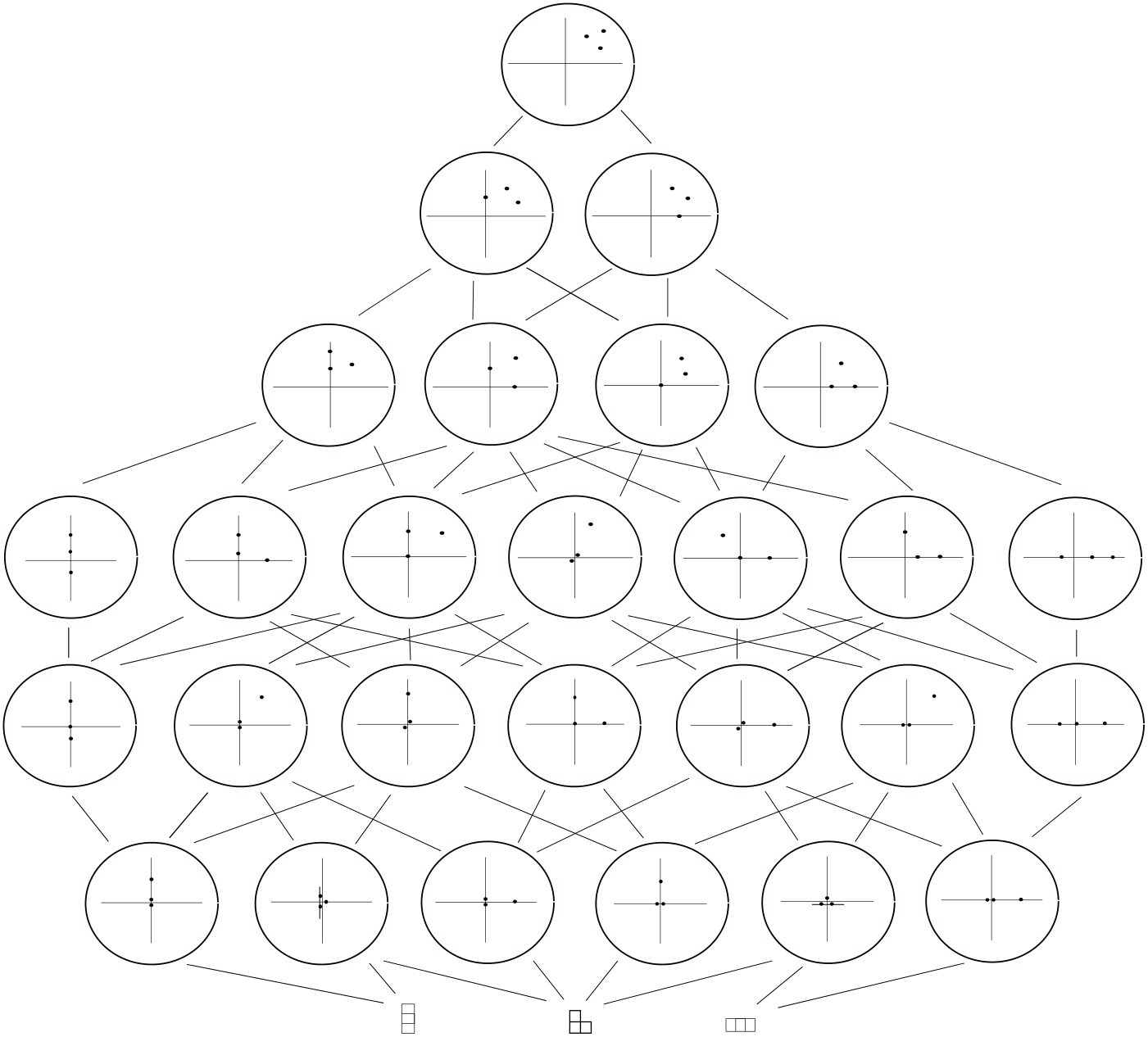}
\caption{The compatibly split subvarieties of $\Hilb^3(\bA^2_k)$ using stratum representative pictures. The lines drawn indicate containment.}
\label{fig;3pts}
\end{center}
\end{figure} 

The stratification of $\Hilb^3(\bA^2_k)$ by all compatibly split subvarieties is drawn in Figure \ref{fig;3pts}. Most of the stratum representative pictures are of the form ``stratum representative picture from the $n=2$ case with a point added on one of the two axes or in $\bA^2_k\setminus \{xy=0\}$''. So, it remains to explain the pictures representing those subvarieties $Y\subseteq \Hilb^3_0(\bA^2_k)$. The stratum representative second from the left (respectively second from the right) on the line containing dimension-$1$ compatibly split subvarieties represents the bone of the edge connecting $\tableau{{ \ }\\{ \ }\\{ \ }}$ and $\tableau{{ \ }\\{ \ }&{ \ }}$ (resp. $\tableau{{ \ }\\{ \ }&{ \ }}$ and $\tableau{{ \ }&{ \ }&{ \ }}$). The vertical line (resp. horizontal line) in the picture is to indicate that $\frac{\partial f}{\partial y} (0,0) = 0$ (resp. $\frac{\partial f}{\partial x} (0,0) = 0$) for all $f$ in any colength-$3$ ideal $I\in Y$. For each of the $0$-dimensional compatibly split subvarieties $\{\langle x,y^3\rangle\}$, $\{\langle x^2,xy,y^2\rangle\}$, and $\{\langle x^3,y\rangle\}$, we draw the standard set associated to the relevant monomial ideal. 

\begin{proposition}\label{p;4pts}
Let $\tableau{{ \ }\\{ \ }\\{ \ }\\{ \ }}~,$ $\tableau{{ \ }\\{ \ }\\{ \ }&{ \ }}~,$ $\tableau{{ \ }&{ \ }\\{ \ }&{ \ }}~,$ $\tableau{{ \ }\\{ \ }&{ \ }&{ \ }}~,$ and $\tableau{{ \ }&{ \ }&{ \ }&{ \ }}$ denote the moment polytopes of the $T^2$-fixed points $\langle x,y^4\rangle$, $\langle x^2,xy,y^3\rangle$, $\langle x^2,y^2\rangle$, $\langle x^3,xy,y^2\rangle$, $\langle x^4,y\rangle$. The compatibly split subvarieties of $\Hilb^4(\bA^2_k)$ which are contained inside the punctual Hilbert scheme $\Hilb^4_0(\bA^2_k)$ are precisely those $Y$ appearing in the following list.
\begin{enumerate}
\item $Y$ is the bone of the edge in the moment polyhedron of $\Hilb^4(\bA^2_k)$ connecting $\tableau{{ \ }\\{ \ }\\{ \ }\\{ \ }}$ and $\tableau{{ \ }\\{ \ }\\{ \ }&{ \ }}~.$ That is, $Y=\{\langle y^4,xy,x^2,ay^3+bx\rangle~|~[a,b]\in \bP^1_k\}$. 
\item $Y$ is the bone of the edge in the moment polyhedron of $\Hilb^4(\bA^2_k)$ connecting $\tableau{{ \ }\\{ \ }&{ \ }&{ \ }}$ and $\tableau{{ \ }&{ \ }&{ \ }&{ \ }}~.$ That is, $Y=\{\langle y^2,xy,x^4,ay+bx^3\rangle~|~[a,b]\in \bP^1_k\}$. 
\item $Y$ is the bone of the edge in the moment polyhedron connecting the vertices $\tableau{{ \ }\\{ \ }\\{ \ }&{ \ }}~,$ $\tableau{{ \ }&{ \ }\\{ \ }&{ \ }}~,$ and $\tableau{{ \ }\\{ \ }&{ \ }&{ \ }}~.$ That is, \[Y=\left\{\langle y^3,xy^2,x^2y,x^3,ax^2+bxy+cy^2, dx^2+exy+fy^2 \rangle~\bigg|~ \textrm{rank} \begin{pmatrix}a & b & c \\d & e & f \end{pmatrix} = 2 \right\}.\] In other words, $Y\cong \textrm{Gr}_2(\langle y^2,xy,x^2\rangle / \langle y^3,xy^2,x^2y,x^3\rangle)$, the Grassmannian of $2$-planes in the vector space $\langle y^2,xy,x^2\rangle / \langle y^3,xy^2,x^2y,x^3\rangle$. 
\item Let the subvariety in 3. be given by $\textrm{Proj}(k[w_0,w_1,w_2])$ where $w_0$, $w_1$, and $w_2$ are the Pl\"{u}cker coordinates $ae-bd$, $af-cd$, and $bf-ce$. Then the subvarieties $\{w_1 = 0\}$ and $\{w_1^2-w_0w_2 = 0\}$ are compatibly split subvarieties of $\Hilb^4(\bA^2_k)$.
\item $Y$ is one of the following four $T^2$-fixed points: $\langle x,y^4\rangle$, $\langle x^2,xy,y^3\rangle$, $\langle x^3,xy,y^2\rangle$, or $\langle x^4,y\rangle$. Note that $\langle x^2,y^2\rangle$ is not compatibly split.
\end{enumerate}
\end{proposition}

\begin{figure}[!h]
\begin{center}
\includegraphics[scale = 0.4]{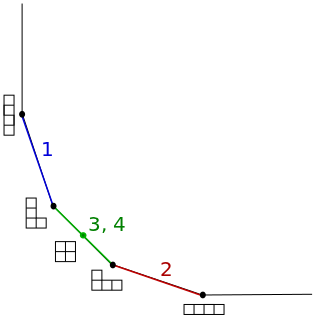}
\caption{The moment polytopes of the subvarieties in items 1. through 4. of Proposition \ref{p;4pts} are labelled in the figure. Notice that the (green) interior vertex is the moment polytope of the $T^2$-fixed point $\langle x^2, y^2\rangle$, which is not compatibly split. Note also that there should not be an interior vertex drawn in the moment polytope of the subvarieties described in item 4.}
\end{center}
\end{figure} 

\begin{lemma}
If $Y$ is the bone of the edge in the moment polyhedron of $\Hilb^4(\bA^2_k)$ connecting $\tableau{{ \ }\\{ \ }\\{ \ }\\{ \ }}$ and $\tableau{{ \ }\\{ \ }\\{ \ }&{ \ }}~,$ then $Y=\{\langle y^4,xy,x^2,ay^3+bx\rangle~|~[a,b]\in \bP^1_k\}$. If $Y$ is the bone of the edge in the moment polyhedron of $\Hilb^4(\bA^2_k)$ connecting $\tableau{{ \ }\\{ \ }&{ \ }&{ \ }}$ and $\tableau{{ \ }&{ \ }&{ \ }&{ \ }}~,$ then $Y=\{\langle y^2,xy,x^4,ay+bx^3\rangle~|~[a,b]\in \bP^1_k\}$. If $Y$ is the bone of the edge in the moment polyhedron connecting the vertices $\tableau{{ \ }\\{ \ }\\{ \ }&{ \ }}~,$ $\tableau{{ \ }&{ \ }\\{ \ }&{ \ }}~,$ and $\tableau{{ \ }\\{ \ }&{ \ }&{ \ }}~,$ then \[Y:=\left\{\langle y^3,xy^2,x^2y,x^3,ax^2+bxy+cy^2, dx^2+exy+fy^2 \rangle~\bigg|~\textrm{rank} \begin{pmatrix} a & b & c \\ d & e & f \end{pmatrix} = 2\right\}.\]
\end{lemma}

\begin{proof}
The argument to verify the first two claims is nearly identical to the one given in the 3-point case. So, we just consider the case where $Y$ is the bone of the edge in the moment polyhedron connecting the vertices $\tableau{{ \ }\\{ \ }\\{ \ }&{ \ }}~,$ $\tableau{{ \ }&{ \ }\\{ \ }&{ \ }}~,$ and $\tableau{{ \ }\\{ \ }&{ \ }&{ \ }}~.$ The coordinates of $\tableau{{ \ }\\{ \ }\\{ \ }&{ \ }}~,$ $\tableau{{ \ }&{ \ }\\{ \ }&{ \ }}~,$ and $\tableau{{ \ }\\{ \ }&{ \ }&{ \ }}$ are $(1,3)$, $(2,2)$, and $(3,1)$ respectively. So, the subtorus $T^1 = \{(t,t)~|~t\in\bG_m\}$ pointwise fixes $Y$. All elements of $Y\cap U_{\langle y^3,xy,x^2\rangle}$ which are pointwise fixed by $T^1$ have the form $\langle y^3, xy^2, xy-\lambda_1y^2, x^2-\lambda_2y^2 \rangle$, $\lambda_1,\lambda_2\in k$, all elements of $Y\cap U_{\langle x^2,y^2\rangle}$ which are pointwise fixed by $T^1$ have the form $\langle y^2x, x^2y, y^2-\mu_1 xy, x^2-\mu_2 xy \rangle$, $\mu\in k$, and all elements of $Y\cap U_{\langle x^3, xy, y^2\rangle}$ which are pointwise fixed by $T^1$ have the form $\langle x^3, x^2y, xy-\nu_1 x^2, y^2-\nu_2 x^2 \rangle$, $\nu_1, \nu_2\in k$. These copies of $\bA^2_k$ glue to give \[Y:=\left\{\langle y^3,xy^2,x^2y,x^3,ax^2+bxy+cy^2, dx^2+exy+fy^2 \rangle~\bigg|~\textrm{rank} \begin{pmatrix} a & b & c \\ d & e & f \end{pmatrix} = 2\right\}.\] That is, $Y=\textrm{Proj}(k[w_0,w_1,w_2])$ where $w_0$, $w_1$, and $w_2$ are given by the Pl\"{u}cker coordinates $ae-bd$, $af-cd$, and $bf-ce$. Notice that $Y\cap U_{\langle y^3,xy,x^2\rangle}$, $Y\cap U_{\langle x^2,y^2\rangle}$, and $Y\cap U_{\langle x^2, xy, y^3\rangle}$ are the open sets $\{w_0\neq 0\}$, $\{w_1\neq 0\}$, and $\{w_2\neq 0\}$ respectively. 
\end{proof}

We now prove Proposition \ref{p;4pts}.

\begin{proof}[Proof of Proposition \ref{p;4pts}]
Each of the subvarieties mentioned in items 1., 2., 3., and 5. of the proposition is the bone of some exterior face of the moment polyhedron of $\Hilb^4(\bA^2_k)$ (see Figure \ref{fig;hilb4polytope}). Thus, they are all compatibly split. We now show that the subvarieties mentioned in item 4. are compatibly split by realizing them as intersections of known compatibly split subvarieties. 

Let $Y_1 = \{w_1^2-w_0w_2 = 0\}$ and $Y_2 = \{w_1=0\}$ be the subvarieties of $\textrm{Proj}(k[w_0,w_1,w_2])$ described in item 4. Notice that the $0$-dimensional subvariety $\{\langle x^2, xy, y^3\rangle\}$ is given by $\{w_1 = w_2 = 0\}\subseteq \textrm{Proj}(k[w_0,w_1,w_2])$ and so $\{\langle x^2, xy, y^3\rangle\}$ is a subvariety of both $Y_1$ and $Y_2$. Thus, $Y_1$ and $Y_2$ are compatibly split if and only if $Y_1\cap U_{\langle x^2, xy,y^3\rangle}$ and $Y_2\cap U_{\langle x^2, xy, y^3\rangle}$ are compatibly split in $U_{\langle x^2, xy, y^3\rangle}$ with the induced splitting. 

We can check that every element of $U_{\langle x^2,xy,y^3\rangle}$ is an ideal generated by polynomials of the form \[\begin{array}{ccccccccc}y^3&-&b_1y^2&-&b_2y&-&b_3x&-&c_1\\ xy^2&-&a_1y^2&-&c_2y&-&c_3x&-&c_4\\xy&-&a_2y^2&-&c_5y&-&b_4x&-&c_6\\x^2&-&a_3y^2&-&c_7y&-&a_4x&-&c_8\end{array}\]where each $a_i$ and $b_i$ are elements of $k$ and each $c_i$ is a polynomial (obtained via Buchberger's S-pair criterion) in $a_1,\dots,a_4$, $b_1\dots,b_4$. (See the proof of Lemma \ref{l;cij} for further explanation regarding computations of this sort.) Thus, $U_{\langle x^2, xy, y^3\rangle}\cong k[a_1,\dots,a_4,b_1,\dots,b_4]$.

Let $Z_1 = \{\langle x^2,xy,y^2\rangle\}\subseteq \Hilb^3(\bA^2_k)$ and let $W_1$ be the closure of the image of \[i_{0,0,1,Z_1}:\Hilb^1(\bA^2_k\setminus\{xy=0\})\times Z_1\rightarrow \Hilb^4(\bA^2_k)\] where $i_{0,0,1,Z_1}$ is as in Proposition \ref{p;intuition}. Then $W_1$ is compatibly split. Now, let $W_1^{o}$ denote the image of $i_{0,0,1,Z_1}$. Notice that $W_1^{o}$ is an open set of $W_1$ and that any $I\in W_1^{o}$ has the form $I = \langle x^2,xy,y^2\rangle\cap \langle x-\alpha, y-\beta\rangle$, $\alpha, \beta\in k^{*}$. Thus, $I = \langle y^3-\beta y^2, xy^2-\alpha y^2, xy-\frac{\alpha}{\beta}y^2, x^2-\frac{\alpha^2}{\beta^2}y^2\rangle$, and we see that $W_1\cap U_{\langle x^2,xy,y^3\rangle}$ is given by the ideal $\langle a_2b_1-a_1, a_2^2-a_3, a_1a_2-a_3b_1, a_3b_1^2-a_1^2, b_2, b_3, b_4, a_4\rangle$. 

Let $Z_2$ be the bone of the exterior edge in the moment polyhedron of $\Hilb^3(\bA^2_k)$ connecting $\tableau{{ \ }\\{ \ }\\{ \ }}$ and $\tableau{{ \ }\\{ \ }&{ \ }}~.$ Let $W_2$ be the closure of the image of \[i_{1,0,0,Z_2}: \Hilb^1(\textrm{punctured x-axis})\times Z_2\rightarrow \Hilb^4(\bA^2_k)\] where $i_{1,0,0,Z_2}$ is as in Proposition \ref{p;intuition}. Then $W_2$ is compatibly split. Let $W_2^{o}$ denote the image of $i_{1,0,0,Z_2}$. Let $W_3$ denote the closure of the image of $i_{1,0,0,Z_1}$ where $Z_1 = \{\langle x^2,xy,y^2\rangle\}\subseteq \Hilb^3(\bA^2_k)$ as in the previous paragraph. Then $W_3$ is a closed subvariety of $W_2$. Let $W_2' = W_2^{o}\cap (W_2\setminus W_3)$. Notice that $W_2'$ is an open set in $W_2$ and that any $I\in W_2'$ has the form $I = \langle y^3, x-\alpha y^2\rangle\cap\langle x-\beta, y\rangle$, $\alpha\in k$, $\beta\in k^{*}$. Thus, $I = \langle y^3, xy^2, xy, x^2-\alpha\beta y^2-\beta x\rangle$, and we see that $W_2\cap U_{\langle x^2, xy, y^3\rangle}$ is given by $\langle a_1, a_2, b_1, b_2, b_3, b_4\rangle$.

Let $Y = \textrm{Proj}(k[w_0, w_1,w_2])$ be as in item 3. of the statement of the proposition. Any $I\in Y\cap U_{\langle x^2, xy, y^3\rangle}$ is given by an ideal of the form $\langle y^3, xy^2, xy-\alpha y^2, x^2-\beta y^2\rangle$ for some $\alpha, \beta \in k$. So, $Y\cap U_{\langle x^2, xy, y^3\rangle}$ is given by the ideal $\langle a_1, a_4, b_1,b_2,b_3,b_4\rangle$.

The intersections $(W_1\cap U_{\langle x^2,xy, y^3\rangle})\cap (Y\cap U_{\langle x^2, xy, y^3\rangle})$ and $(W_2\cap U_{\langle x^2,xy, y^3\rangle})\cap (Y\cap U_{\langle x^2, xy, y^3\rangle})$ are compatibly split subvarieties of $U_{\langle x^2, xy, y^3\rangle}$ and are given by the ideals $\langle a_3-a_2^2, a_1, a_4, b_1,b_2,b_3,b_4\rangle$ and $\langle a_1, a_2, a_4, b_1, b_2, b_3, b_4\rangle$. 

Now, $U_{\langle x^2, xy, y^3\rangle}\cap Y$ is the open patch of $Y$ given by $\{w_0 = 1\}$. Notice that, on this patch, $w_1 = -a_2$ and $w_2 = a_3$. Thus,  $W_1\cap Y$ is the subvariety of $Y =\textrm{Proj}(k[w_0,w_1,w_2])$ given by $\{w_0w_2-w_1^2\}$ and $W_2\cap Y$ is given by $\{w_1 = 0\}$. This completes the proof that all subvarieties listed in the statement of the proposition are compatibly split.

As in the 3-point situation, we can use Figure \ref{fig;hilb4polytope} and Lemma \ref{l;puncSplit} to conclude that any compatibly split subvariety of $\Hilb^4(\bA^2_k)$ which is contained in $\Hilb^4_0(\bA^2_k)$ and which does not appear in the list given in the statement of Proposition \ref{p;4pts} must be a subvariety of $Y = \textrm{Proj}(k[w_0,w_1,w_2])$. Because $\{w_1(w_0w_2-w_1^2)=0\}$ is an anticanonical divisor which determines a splitting of $\bP^2$, we may apply \cite[Proposition 2.1]{KM} to see that all compatibly split subvarieties of $\bP^2_k$ are contained inside of $\{w_1(w_0w_2-w_1^2)=0\}$. As all $T^2$-invariant subvarieties of $\{w_1(w_0w_2-w_1^2)=0\}$ are $T^2$-fixed points already appearing in the list given in the statement of the proposition, there do not exist any compatibly split subvarieties of $\Hilb^4(\bA^2_k)$ which are both contained in $\Hilb^4_0(\bA^2_k)$ and which do not appear in the list given in the statement of Proposition \ref{p;4pts}.
\end{proof}

The stratification of $\Hilb^4(\bA^2_k)$ by all compatibly split subvarieties is drawn in Figure \ref{fig;4pts}. (See the last page of this subsection.) Most of the stratum representative pictures are of the form ``stratum representative picture from the $n=3$ case with a point added on one of the two axes or in $\bA^2_k\setminus \{xy=0\}$''. So, it remains to explain the pictures representing those compatibly split subvarieties $Y\subseteq \Hilb^4_0(\bA^2_k)$. The stratum representative second from the left (respectively second from the right) in the line containing the dimension-$1$ compatibly split subvarieties represents the bone of the edge connecting $\tableau{{ \ }\\{ \ }\\{ \ }\\{ \ }}$ and $\tableau{{ \ }\\{ \ }\\{ \ }&{ \ }}$ (resp. $\tableau{{ \ }\\{ \ }&{ \ }&{ \ }}$ and $\tableau{{ \ }&{ \ }&{ \ }&{ \ }}$). The vertical line (resp. horizontal line) in the picture is to indicate that $\frac{\partial f}{\partial y} (0,0) = 0$ and $\frac{\partial^2 f}{\partial y^2} (0,0) = 0$ (resp. $\frac{\partial f}{\partial x} (0,0) = 0$ and $\frac{\partial^2 f}{\partial x^2}(0,0) = 0$) for all $f$ in any colength-$4$ ideal $I\in Y$. Next, let $Y = \textrm{Proj}(k[w_0,w_1,w_2])$ be the subvariety from item 3. of Proposition \ref{p;4pts}. Each $I\in Y\cap U_{\langle x^2, xy, y^3\rangle}$ has the form $\langle y^3, xy^2, xy-\lambda_1y^2, x^2 - \lambda_2y^2\rangle$. Each such point is the limit, under dilation toward the origin, of an element of the form $\langle x,y\rangle\cap \langle x-a,y\rangle\cap \langle x, y-b\rangle\cap \langle x-c,y-d\rangle$, $a,b,c,d\in k^*$. That is, \[\lim_{t\rightarrow 0} (\langle x,y\rangle\cap \langle x-at,y\rangle\cap \langle x, y-bt\rangle\cap \langle x-ct,y-dt\rangle) = \langle y^3, xy^2, xy + \frac{c}{b-d}y^2, x^2-\frac{c(a-c)}{d(b-d)}\rangle.\] This explains the stratum representative picture appearing in the center of the line with the $2$-dimensional compatibly split subvarieties of $\Hilb^4(\bA^2_k)$. Let $Y_1 = \{w_1 = 0\}$ and $Y_2 = \{w_1^2-w_0w_2\}$. Elements of $Y_1\cap U_{\langle x^2, xy, y^3\rangle}$ have the form $\langle y^3, xy, x^2 - \lambda y^2\rangle$ and elements of $Y_2\cap U_{\langle x^2, xy, y^3\rangle}$ have the form $\langle y^3, xy - \lambda y^2, x^2-\lambda^2 y^2\rangle$. Notice that \[\lim_{t\rightarrow 0} (\langle x-at, y\rangle \cap \langle x-bt, y\rangle\cap \langle x,y-ct\rangle\cap \langle x, y-dt\rangle) = \langle y^3, xy, x^2+\frac{ab}{cd}y^2\rangle,\] and \[\lim_{t\rightarrow 0}(\langle x,y\rangle\cap \langle x-dt, y-cdt\rangle \cap \langle x+dt, y+cdt\rangle\cap \langle x-at, y-bt\rangle) = \langle y^3, xy - \frac{1}{c}y^2, x^2 - \frac{1}{c^2}y^2\rangle.\] This justifies the stratum representative pictures for $Y_1$ and $Y_2$ (see the two pictures in the middle of the line containing dimension-$1$ compatibly split subvarieties). Finally, for each $0$-dimensional compatibly split subvariety, we draw the standard set associated to the relevant monomial ideal.

\begin{sidewaysfigure}[!h]
\includegraphics[scale = 0.26]{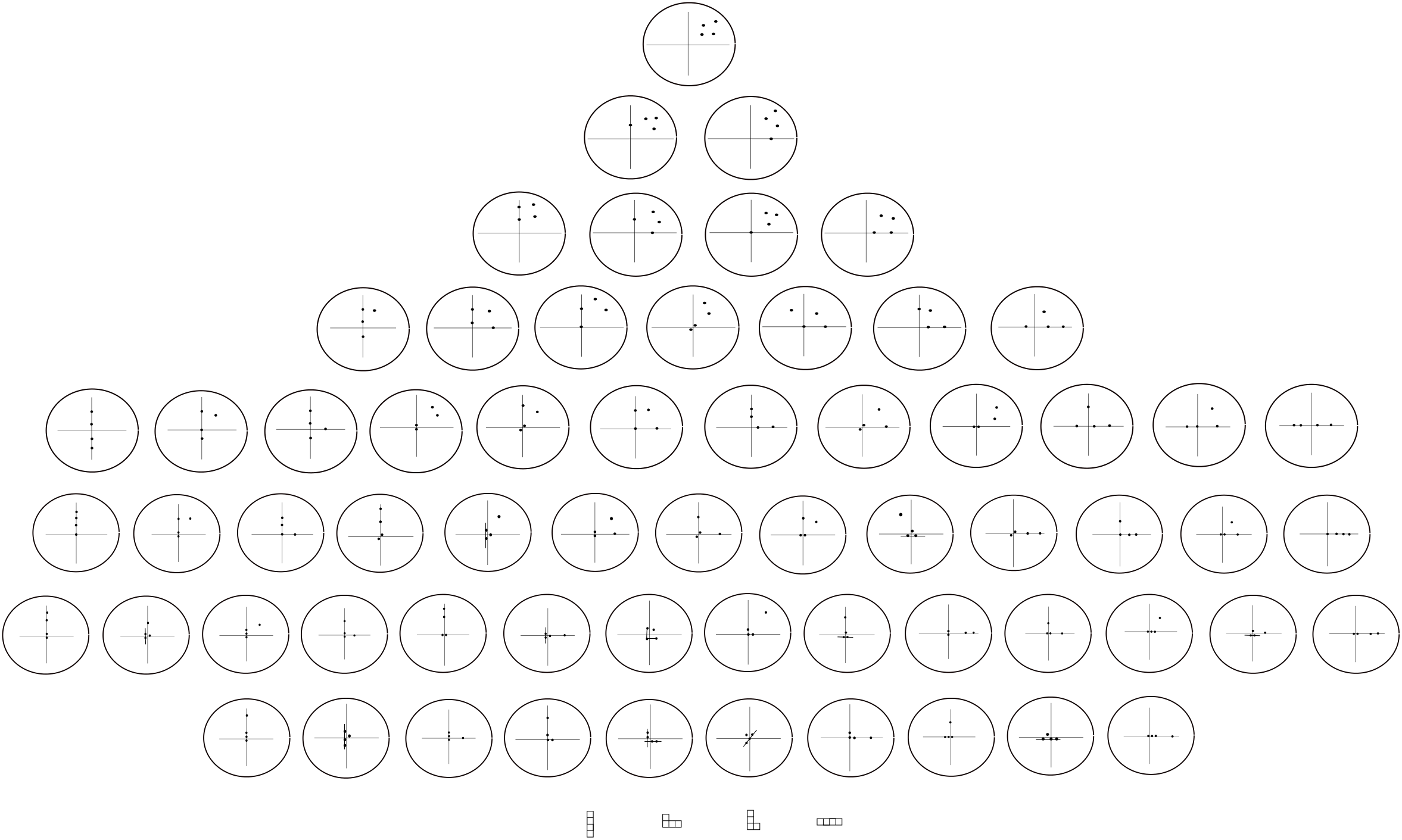}
\caption{The compatibly split subvarieties of $\Hilb^4(\bA^2_k)$}
\label{fig;4pts}
\end{sidewaysfigure}

We now turn our attention to the $n=5$ case. We note that this case uses more of the ideas discussed in Section 2.3. So, it may be preferable to read Section 2.3 before reading the material that appears in the remainder of this subsection. 

\begin{conjecture}\label{p;5pts}
Let $\tableau{{ \ }\\{ \ }\\{ \ }\\{ \ }\\{ \ }}~,$ $\tableau{{ \ }\\{ \ }\\{ \ }\\{ \ }&{ \ }}~,$ $\tableau{{ \ }\\{ \ }&{ \ }\\{ \ }&{ \ }}~,$ $\tableau{{ \ }\\{ \ }\\{ \ }&{ \ }&{ \ }}~,$ $\tableau{{ \ }&{ \ }\\{ \ }&{ \ }&{ \ }}~,$ $\tableau{{ \ }\\{ \ }&{ \ }&{ \ }&{ \ }}~,$ and $\tableau{{ \ }&{ \ }&{ \ }&{ \ }&{ \ }}$ denote the moment polytopes of the $T^2$-fixed points $\langle x,y^5\rangle$, $\langle x^2,xy,y^4\rangle$, $\langle x^2,xy^2, y^3\rangle$, $\langle x^3,xy,y^3\rangle$, $\langle x^3,x^2y,y^2\rangle$, $\langle x^4,xy,y^2\rangle$, and $\langle x^5,y\rangle$ respectively. The compatibly split subvarieties of $\Hilb^5(\bA^2_k)$ which are contained inside the punctual Hilbert scheme, $\Hilb^5_0(\bA^2_k)$, are precisely those $Y$ appearing in the following list.
\begin{enumerate}
\item $Y$ is the bone of the edge in the moment polyhedron of $\Hilb^5(\bA^2_k)$ connecting $\tableau{{ \ }\\{ \ }\\{ \ }\\{ \ }\\{ \ }}$ and $\tableau{{ \ }\\{ \ }\\{ \ }\\{ \ }&{ \ }}~.$ That is, $Y=\{\langle y^5,xy,x^2,ay^4+bx\rangle~|~[a,b]\in \bP^1_k\}$.
\item $Y$ is the bone of the edge in the moment polyhedron of $\Hilb^5(\bA^2_k)$ connecting $\tableau{{ \ }\\{ \ }&{ \ }&{ \ }&{ \ }}$ and $\tableau{{ \ }&{ \ }&{ \ }&{ \ }&{ \ }}~.$ That is, $Y=\{\langle y^2,xy,x^5,ax^4+by\rangle~|~[a,b]\in \bP^1_k\}$.
\item  $Y=\{\langle y^4,xy^2,x^2y,x^3,ax^2+bxy+cy^3, dx^2+exy+fy^3 \rangle~|~\textrm{rank} \begin{pmatrix}a & b & c \\d & e & f \end{pmatrix} = 2 \}$. Note that the moment polytope of $Y$ is the triangle with vertices $\tableau{{ \ }\\{ \ }\\{ \ }\\{ \ }&{ \ }}~,$ $\tableau{{ \ }\\{ \ }&{ \ }\\{ \ }&{ \ }}~,$ and $\tableau{{ \ }\\{ \ }\\{ \ }&{ \ }&{ \ }}~.$ 
\item $Y$ is the bone of the edge (in the moment polytope of the subvariety described in item 3.) connecting $\tableau{{ \ }\\{ \ }\\{ \ }\\{ \ }&{ \ }}$ and $\tableau{{ \ }\\{ \ }&{ \ }\\{ \ }&{ \ }}~,$ or $\tableau{{ \ }\\{ \ }&{ \ }\\{ \ }&{ \ }}$ and $\tableau{{ \ }\\{ \ }\\{ \ }&{ \ }&{ \ }}~,$ or $\tableau{{ \ }\\{ \ }\\{ \ }\\{ \ }&{ \ }}$ and $\tableau{{ \ }\\{ \ }\\{ \ }&{ \ }&{ \ }}~.$ In the first case, $Y = \{\langle y^4, xy^2, x^2, ay^3+bxy\rangle~|~[a,b]\in \bP^1_k\}$. In the second case, $Y = \{\langle y^3, xy^2, x^2y, x^3, ax^2+bxy\rangle~|~[a,b]\in \bP^1_k\}$. In the third case, $Y = \{\langle y^4, xy, x^3, ay^3+bx^2\rangle~|~[a,b]\in \bP^1_k\}$. 
\item $Y=\{\langle y^3,xy^2,x^2y,x^4,ay^2+bxy+cx^3, dy^2+exy+fx^3 \rangle~|~\textrm{rank} \begin{pmatrix}a & b & c \\d & e & f \end{pmatrix} = 2 \}$. Note that the moment polytope of $Y$ is the triangle with vertices $\tableau{{ \ }\\{ \ }\\{ \ }&{ \ }&{ \ }}~,$ $\tableau{{ \ }&{ \ }\\{ \ }&{ \ }&{ \ }}~,$ and $\tableau{{ \ }\\{ \ }&{ \ }&{ \ }&{ \ }}~.$ 
\item $Y$ is the bone of the edge (in the moment polytope of subvariety described in item 5.) connecting either $\tableau{{ \ }\\{ \ }\\{ \ }&{ \ }&{ \ }}$ and $\tableau{{ \ }&{ \ }\\{ \ }&{ \ }&{ \ }}~,$ or $\tableau{{ \ }&{ \ }\\{ \ }&{ \ }&{ \ }}$ and $\tableau{{ \ }\\{ \ }&{ \ }&{ \ }&{ \ }}~,$ or $\tableau{{ \ }\\{ \ }\\{ \ }&{ \ }&{ \ }}$ and $\tableau{{ \ }\\{ \ }&{ \ }&{ \ }&{ \ }}~.$ In the first case, $Y = \{\langle x^3, y^2x, x^2y, y^3, ay^2+bxy\rangle~|~[a,b]\in \bP^1_k\}~.$ In the second case, $Y = \{\langle x^4, x^2y, y^2, ax^3+bxy\rangle~|~[a,b]\in \bP^1_k\}$. In the third case, $Y = \{\langle x^4, xy, y^3, ax^3+by^2\rangle~|~[a,b]\in \bP^1_k\}$. 
\item $Y$ is the bone of the edge in the moment polyhedron of $\Hilb^5(\bA^2_k)$ connecting $\tableau{{ \ }\\{ \ }&{ \ }\\{ \ }&{ \ }}~,$ $\tableau{{ \ }\\{ \ }\\{ \ }&{ \ }&{ \ }}~,$ and $\tableau{{ \ }&{ \ }\\{ \ }&{ \ }&{ \ }}~.$ That is, $Y=\{\langle y^3,xy^2,x^2y,x^3,ay^2+bxy+cx^2\rangle~|~[a,b,c]\in \bP^2_k\}$. 
\item $Y$ is any $T^2$-fixed point.
\end{enumerate}
\end{conjecture}

\begin{figure}[H]
\begin{center}
\includegraphics[scale = 0.4]{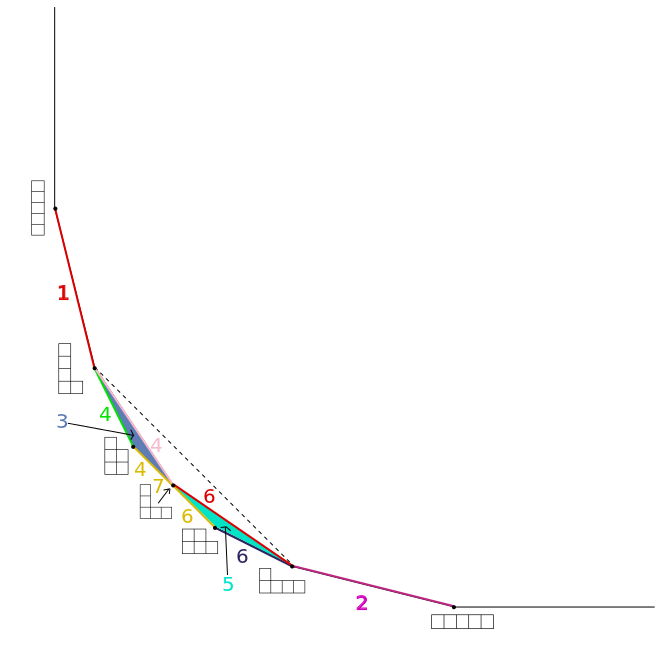}
\end{center}
\caption{The moment polyhedron of $\Hilb^5(\bA^2_k)$, with only some interior edges, is drawn above. The moment polytopes of the subvarieties in items 1. through 7. of Conjecture \ref{p;5pts} are labelled. Notice that all of the vertices, including the interior one, is the moment polytope of a compatibly split $T^2$-fixed point. Note also that there does not exist a compatibly split subvariety of $\Hilb^5(\bA^2_k)$ having the dashed edge as its moment polytope (see Lemma \ref{l;badedge}).}
\label{fig;5ptspolytope}
\end{figure}

We will give a nearly complete proof that the conjecture that Conjecture \ref{p;5pts} is true, is true. In particular, we will show that all subvarieties listed in Conjecture \ref{p;5pts} are compatibly split, and that the only other potential compatibly split subvariety of $\Hilb^5(\bA^2_k)$ is \[Y = \{\langle y^3, xy^2, x^2y, x^3, ay^2+cx^2\rangle~|~[a,c]\in\bP^1_k\}.\] In addition, we'll show that $Y$ is not compatibly split for primes $2<p\leq 23$. 

\begin{lemma}\label{l;badedge}
Let $Y$ be the bone of the edge in the x-ray of $\Hilb^5(\bA^2_k)$ connecting $\tableau{{ \ }\\{ \ }\\{ \ }\\{ \ }&{ \ }}$ and $\tableau{{ \ }\\{ \ }&{ \ }&{ \ }&{ \ }}~.$ Then $Y$ is not compatibly split.
\end{lemma}

\begin{proof}
We show that $Y\cap U_{\langle x^2, xy,y^4\rangle}$ is not compatibly split in $U_{\langle x^2,xy, y^4\rangle}$ with the induced splitting. Using Macaulay 2, we can apply Buchberger's S-pair criterion (or code by M. Lederer implementing techniques in \cite{Led}) to see that all elements of $U_{\langle x^2, xy, y^4\rangle}$ are generated by polynomials of the form \[\begin{array}{ccccccccccc}y^4&-&b_1y^3&-&b_2y^2&-&b_3y&-&b_4x&-&c_1\\xy^3&-&a_1y^3&-&c_2y^2&-&c_3y&-&c_4x&-&c_5\\xy^2&-&a_2y^3&-&c_6y^2&-&c_7y&-&c_8x&-&c_9\\xy&-&a_3y^3&-&c_{10}y^2&-&c_{11}y&-&b_5x&-&c_{12}\\x^2&-&a_4y^3&-&c_{13}y^2&-&c_{14}y&-&a_5x&-&c_{15}\end{array}\]where

\noindent$c_1~= a_1b_4-a_2b_1b_4+a_2b_4b_5-a_3b_1b_4b_5-a_3b_2b_4+a_3b_4b_5^2-a_5b_4-b_1b_5^3-b_2b_5^2-b_3b_5+b_5^4$\\
$c_2~= -a_2a_3b_4+a_2b_2-a_3^2b_4b_5+a_3b_3-a_3b_5^3+a_4b_4$\\
$c_3~=2a_1a_3b_4-2a_2a_3b_1b_4-2a_2a_3b_4b_5+a_2b_3-a_2b_5^3-a_3^3b_4^2+a_3^2b_1b_4b_5-2a_3^2b_2b_4+a_3^2b_4b_5^2-a_3a_5b_4-$\\$~~~~~~~~~~a_3b_2b_5^2-a_3b_3b_5+a_3b_5^4+a_4b_4b_5$\\
$c_4~=a_2b_4+a_3b_4b_5+b_5^3$\\
$c_5~=a_1a_2b_4-a_1a_3b_4b_5-a_1b_5^3-a_2^2b_1b_4+a_2^2b_4b_5-2a_2a_3^2b_4^2-a_2a_3b_2b_4-a_2a_3b_4b_5^2-a_2a_5b_4-a_2b_2b_5^2-$\\$~~~~~~~~~~a_2b_3b_5+a_2b_5^4+a_3^3b_1b_4^2+a_3^3b_4^2b_5+a_3^2b_1b_4b_5^2+a_3^2b_2b_4b_5+a_3^2b_4b_5^3+a_3a_4b_4^2+a_3b_2b_5^3+a_4b_4b_5^2$\\
$c_6~=a_1-a_2b_1-a_3^2b_4-a_3b_5^2$\\
$c_7~=-2a_2a_3b_4-a_2b_5^2+a_3^2b_1b_4+a_3b_1b_5^2+a_3b_3+a_4b_4$\\
$c_8~=a_3b_4+b_5^2$\\
$c_9~=a_1a_3b_4-a_1b_5^2-a_2a_3b_1b_4-a_2a_3b_4b_5+a_2b_1b_5^2-a_3^2b_2b_4+2a_3^2b_4b_5^2-a_3a_5b_4-a_3b_1b_5^3-$\\$~~~~~~~~~~a_3b_3b_5+a_3b_5^4+a_4b_4b_5$\\
$c_{10}=a_2-a_3b_1-a_3b_5$\\
$c_{11}= a_1-a_2b_1-a_2b_5-a_3^2b_4+a_3b_1b_5-a_3b_2$ \\
$c_{12}=-a_1b_5-2a_2a_3b_4+a_2b_1b_5+a_3^2b_1b_4+a_3^2b_4b_5+a_3b_2b_5+a_4b_4$\\
$c_{13}= 2a_1a_3+a_2^2-2a_2a_3b_1-2a_2a_3b_5-a_3^3b_4+a_3^2b_1b_5-a_3^2b_2-a_3a_5-a_4b_1+a_4b_5$ \\
$c_{14}=2a_1a_2-2a_1a_3b_1-2a_2^2b_1-3a_2a_3^2b_4+2a_2a_3b_1^2+2a_2a_3b_1b_5-2a_2a_3b_5^2-a_2a_5+$\\$~~~~~~~~~~2a_3^3b_1b_4-a_3^2b_1^2b_5+a_3^2b_1b_2+a_3^2b_1b_5^2+a_3^2b_3+a_3a_4b_4+a_3a_5b_1-a_4b_1b_5-a_4b_2+a_4b_5^2$\\
$c_{15}= a_1^2-2a_1a_2b_1-2a_1a_3b_2-a_1a_5-2a_2^2a_3b_4+a_2^2b_1^2+3a_2a_3^2b_1b_4-a_2a_3^2b_4b_5+2a_2a_3b_1b_2+2a_2a_3b_1b_5^2+$\\$~~~~~~~~~~~2a_2a_3b_2b_5+2a_2a_3b_3-2a_2a_3b_5^3+a_2a_4b_4+a_2a_5b_1-a_3^3b_1^2b_4+a_3^3b_4b_5^2-a_3^2b_1^2b_5^2-a_3^2b_1b_2b_5-a_3^2b_1b_3+$\\$~~~~~~~~~~a_3^2b_2^2-a_3^2b_2b_5^2-a_3^2b_3b_5+a_3^2b_5^4-a_3a_4b_1b_4+a_3a_4b_4b_5+a_3a_5b_2-a_4b_1b_5^2-a_4b_2b_5-a_4b_3+a_4b_5^3$\\
Now, consider \[M_1 = \begin{pmatrix} -b_4&-c_1\\-b_5&-c_{12}\end{pmatrix},~~\textrm{and}~~M_2 = \begin{pmatrix}-a_1&-c_2&-c_3&-c_5\\-a_2&-c_6&-c_7&-c_9\\-a_3&-c_{10}&-c_{11}&-c_{12}\\-a_4&-c_{13}&-c_{14}&-c_{15}\end{pmatrix}.\] Weight the variables $a_1,\dots,a_5$,$b_1,\dots,b_5$ by $0,0,0,\beta,\beta-\delta,0,0,0,-\alpha,-\alpha+\gamma$ where $\alpha\gg \beta\gg \gamma\gg \delta\gg 0$. Using Macaulay 2, we see that $\textrm{init}(\textrm{det}(M_1)) = a_5b_4b_5$ and $\textrm{init}(\textrm{det}(M_2)) = -a_4f_3$ where $\textrm{Tr}(f_3^{p-1}\cdot)$ is the splitting of $U_{\langle x,y^3\rangle}$ (see Subsection \ref{s;inducedsplitting}). By Proposition \ref{p;eqnRNC}, there is a weighting of the variables $>$, which refines the weighting given above, such that $\textrm{init}_{>}(-\textrm{det}(M_1)\textrm{det}(M_2)) = a_1\cdots a_5b_1\cdots b_5$. Furthermore, $\textrm{Tr}((-\textrm{det}(M_1)\textrm{det}(M_2))^{p-1}\cdot)$ is the splitting of $U_{\langle x^2,xy,y^4\rangle}$ (see Subsection \ref{s;inducedsplitting} for further explanation). 

Now, let $Y$ be the bone of the edge connecting $\tableau{{ \ }\\{ \ }\\{ \ }\\{ \ }&{ \ }}$ and $\tableau{{ \ }\\{ \ }&{ \ }&{ \ }&{ \ }}~.$ $Y\cap U_{\langle x^2, xy,y^4\rangle}$ is given by the ideal $I = \langle b_1,\dots,b_5,a_1,a_3,a_4,a_5\rangle$. By Lemma \ref{l;puncSplit}, this is not a compatibly split ideal of $U_{\langle x,y^3\rangle}\times \bA^4_{a_4,a_5,b_4,b_5}$ with the splitting given by $\textrm{Tr}(f_3^{p-1}(a_4b_4a_5b_5)^{p-1}\cdot)$. By \cite[Theorem 2]{K} (see Theorem \ref{t;knutson}), $I$ is also not a compatibly split ideal of $U_{\langle x^2,xy,y^4\rangle}$.
\end{proof}

\begin{conjecture}\label{l;b=0}
$Y = \{\langle y^3, xy^2, x^2y, x^3, ay^2+cx^2\rangle~|~[a,c]\in\bP^1_k\}$ is not compatibly split.
\end{conjecture}

\begin{proof}[Evidence (i.e. Proof when $2<p\leq 23$)]
We now provide some evidence supporting Conjecture \ref{l;b=0} by showing that the open patch $Y\cap U_{\langle x^2,xy^2,y^3\rangle}$ is not a compatibly split subvariety of $U_{\langle x^2,xy^2,y^3\rangle}$ for those primes $p$ satisfying $2<p\leq 23$.

First note that every ideal in $U_{\langle x^2,xy^2,y^3\rangle}$ is generated by polynomials of the form \[\begin{array}{ccccccccccc}y^3&-&b_1y^2&-&b_2xy&-&c_1y&-&b_4x&-&c_2\\ xy^2&-&a_1y^2&-&b_3xy&-&c_3y&-&b_5x&-&c_4\\x^2y&-&a_2y^2&-&a_3xy&-&c_5y&-&c_6x&-&c_7\\x^2&-&a_4y^2&-&a_5xy&-&c_8y&-&c_9x&-&c_{10}\end{array}\] where

\noindent$c_1 = a_1b_2-a_3b_2-a_5b_4-b_1b_3+b_3^2+b_5$\\
$c_2 = a_1b_4-a_3b_4-a_5b_2b_5+a_5b_3b_4-b_1b_5+b_3b_5$\\
$c_3 = -a_1b_3+a_2b_2+a_4b_4$\\
$c_4 =  -a_1b_5+a_2b_4+a_4b_2b_5-a_4b_3b_4$\\
$c_5 = a_1^2-a_1a_3-a_2b_1+a_2b_3+a_4b_5$\\
$c_6 =  a_4b_4+a_5b_5$\\
$c_7 =a_1a_4b_4-a_1a_5b_5+a_2a_5b_4-a_3a_4b_4-a_4b_1b_5+a_4b_3b_5$\\
$c_8 = -a_1a_5+a_2-a_4b_1$\\
$c_9 = a_3-a_4b_2-a_5b_3$\\
$c_{10} = a_1^2-a_1a_3-a_1a_4b_2+a_1a_5b_3-a_2a_5b_2-a_2b_1+a_2b_3+a_3a_4b_2+a_4b_1b_3-a_4b_3^2$.\\
Thus, $U_{\langle x^2, xy^2,y^3\rangle}\cong \Spec(k[a_1,\dots,a_5,b_1,\dots,b_5])$. Consider the matrices \[M_1 = \begin{pmatrix}-b_4&-c_2\\-b_5&-c_4\end{pmatrix}\textrm{ and } M_2 =\begin{pmatrix}-a_1&-c_3&-c_4\\-a_2&-c_5&-c_7\\-a_4&-c_8&-c_{10}\end{pmatrix}.\] Let $f_1 = \textrm{det}(M_1)$ and let $f_2 = \textrm{det}(M_2)$. Weight $a_1,a_2,a_3,a_4,a_5$, $b_1,b_2,b_3,b_4,b_5$ by \[1,10^6,10^6-10^2,10^8,10^8-10^3,-10,-10^7,-10^7+10^5,-10^9,-10^9+10^4.\] Then, $\textrm{init}(-f_1f_2) = a_1\cdots a_5b_1\cdots b_5$ and the induced splitting of $U_{\langle x^2, xy^2, y^3\rangle}$ is $\textrm{Tr}((-f_1f_2)^{p-1}\cdot)$.

Next notice that $Y\cap U_{\langle x^2, xy^2, y^3\rangle}$ is given by the ideal $I = \langle a_1,a_2,a_3,a_5,b_1,b_2,b_3,b_4,b_5\rangle$. We will show, for small primes $p$, that $(-f_1f_2)^{p-1}$ has the term $c(a_1a_2a_3)^{p-1}a_4^{3(p-1)/2}(b_1b_2b_3b_4b_5)^{p-1}$ for some non-zero constant $c\in k^*$. It will then follow that $I$ is not compatibly split (for these $p$) since $a_4^{(p+1)/2}a_5^{p-1}\in I$ but $\textrm{Tr}((-f_1f_2)^{p-1}(a_4^{(p+1)/2}a_5^{p-1}))\notin I$. That is, $\textrm{Tr}((-f_1f_2)^{p-1}(a_4^{(p+1)/2}a_5^{p-1})) = c'a_4+R$ for some $c'\in k^{*}$ and some polynomial $R$ containing no linear terms in $a_4$.

We now show, for $2<p\leq 23$, that $(-f_1f_2)^{p-1}$ has the term $c(a_1a_2a_3)^{p-1}a_4^{3(p-1)/2}(b_1b_2b_3b_4b_5)^{p-1}$ for some $c\in k^*$. To begin, re-weight the variables $a_1,a_2,a_3,a_4,a_5,b_1,b_2,b_3,b_4,b_5$ by $-10^2,-10^2,0$, $0,-10^5$, $0,0$, $0$, $-10^5$, $-10^5$ and note that if the desired term shows up in $(\textrm{init}(-f_1f_2))^{p-1}$ then it shows up in $(-f_1f_2)^{p-1}$ with the same coefficient.

Using Macaulay 2, we can see that, with respect to the given weighting, \[\textrm{init}(f_1) = -a_4b_3b_4^2+a_4b_2b_4b_5+a_3b_4b_5+b_1b_5^2-b_3b_5^2,\] and \[\textrm{init}(f_2) = (a_1^2a_3+a_1a_2b_1+a_2^2b_2-2a_1a_2b_3)(a_4)(a_3b_2+b_1b_3-b_3^2)\]
Using Sage with the FLINT library, we can check that for $2<p\leq 23$, the desired term appears in $(p-1)^{st}$ power of $\textrm{init}(f_1)\textrm{init}(f_2)$ with coefficient $1~(\textrm{mod } p)$.
\end{proof}

\begin{lemma}
The bones of the various subpolytopes mentioned Conjecture \ref{p;5pts} are as stated in Conjecture \ref{p;5pts}. In addition, \[Y_1:=\left\{\langle y^4,xy^2,x^2y,x^3,ax^2+bxy+cy^3, dx^2+exy+fy^3 \rangle~\bigg|~\textrm{rank} \begin{pmatrix}a & b & c \\d & e & f \end{pmatrix} = 2\right\}\] is the unique $T^2$-invariant subvariety of $\Hilb^5_0(\bA^2_k)$ whose moment polytope is the triangle with vertices $\tableau{{ \ }\\{ \ }\\{ \ }\\{ \ }&{ \ }}~,~\tableau{{ \ }\\{ \ }&{ \ }\\{ \ }&{ \ }}~,$ and $\tableau{{ \ }\\{ \ }\\{ \ }&{ \ }&{ \ }}~.$ Similarly, \[Y_2:=\left\{\langle y^3,xy^2,x^2y,x^4,ay^2+bxy+cx^3, dy^2+exy+fx^3 \rangle~\bigg|~\textrm{rank} \begin{pmatrix}a & b & c \\d & e & f \end{pmatrix} = 2 \right\}\] is the unique $T^2$-invariant subvariety whose moment polytope is the triangle with vertices $\tableau{{ \ }\\{ \ }\\{ \ }&{ \ }&{ \ }}~,$ $\tableau{{ \ }&{ \ }\\{ \ }&{ \ }&{ \ }}~,$ and $\tableau{{ \ }\\{ \ }&{ \ }&{ \ }&{ \ }}~.$
\end{lemma}

\begin{proof}
Most of the arguments are nearly identical to ones seen previously. So we just show that $Y_1$ is the unique subvariety whose moment polytope is the triangle with vertices $\tableau{{ \ }\\{ \ }\\{ \ }\\{ \ }&{ \ }}~,~\tableau{{ \ }\\{ \ }&{ \ }\\{ \ }&{ \ }}~,$ and $\tableau{{ \ }\\{ \ }\\{ \ }&{ \ }&{ \ }}~.$ To do so we apply a theorem of Bia\l ynicki-Birula (see \cite[Section 4]{BB}).

First note that the bone of the edge connecting $\tableau{{ \ }\\{ \ }\\{ \ }\\{ \ }&{ \ }}$ and $\tableau{{ \ }\\{ \ }\\{ \ }&{ \ }&{ \ }}$ is isomorphic to $\bP^1_k$. Indeed, on the open patch $U_{\langle x^2,xy,y^4\rangle}$, the set of ideals pointwise fixed by the subtorus $T^1 = \{(t^3,t^2)~|~t\in \bG_m\}$ is $\{\langle y^4, xy, x^2-\lambda y^3\rangle~|~\lambda\in k\}$. On the open set $U_{\langle x^3,xy,y^3\rangle}$, the set of ideals pointwise fixed by $T^1$ is $\{\langle y^4, xy, y^3-\mu x^2\rangle~|~\mu\in k\}$. Thus, $\{\langle y^4, xy, ax^2+by^3\rangle~|~[a,b]\in \bP^1_k\}$ as desired.

Now, consider the point $w = \langle y^4, xy, x^2-y^3\rangle$. This is a general point of $\{\langle y^4, xy, ax^2+by^3\rangle~|~[a,b]\in \bP^1_k\}$ (i.e. it isn't one of the $T^2$-fixed points). Recall that $U_{\langle x^2, xy, y^4\rangle}$ is isomorphic to $k[a_1,\dots,a_5,b_1,\dots,b_5]$ (as explained in the proof of the previous lemma). The Zariski cotangent space $\frak{m}_w/\frak{m}_w^2$ of $\Hilb^5_0(\bA^2_k)$ at $w$ is $4$-dimensional with generators $a_2, a_3, a_4-1$, $b_4$ $(\textrm{mod } \frak{m}_w^2)$. These generators are $T^1$-weight vectors with weights $(-1,1)\cdot (3,2) = -1$, $(-1,2)\cdot (3,2) = 1$, $0$, and $(1,-4)\cdot (3,2) = -5$. Thus, the dual vectors in $T_w(\Hilb^5_0(\bA^2_k))$ have weights $1,-1,0$, and $5$. As there is 1 negative weight, and 1 zero weight, the locally closed Bia\l ynicki-Birula (B-B) stratrum (in the B-B decomposition of $\Hilb^5_0(\bA^2_k)$) with B-B sink $\{\langle y^4, xy, ax^2+by^3\rangle~|~[a,b]\in \bP^1_k\}$ is $(1+1)$-dimensional. 

Since $Y_1$ is a $2$-dimensional irreducible subvariety with fixed points whose moment polytopes are the vertices $\tableau{{ \ }\\{ \ }\\{ \ }\\{ \ }&{ \ }}~,$ $\tableau{{ \ }\\{ \ }&{ \ }\\{ \ }&{ \ }}~,$ and $\tableau{{ \ }\\{ \ }\\{ \ }&{ \ }&{ \ }}~,$ $Y_1$ must be the unique subvariety whose moment polytope is the triangle with vertices $\tableau{{ \ }\\{ \ }\\{ \ }\\{ \ }&{ \ }}~,$ $\tableau{{ \ }\\{ \ }&{ \ }\\{ \ }&{ \ }}~,$ and $\tableau{{ \ }\\{ \ }\\{ \ }&{ \ }&{ \ }}~.$ 
\end{proof}

We now prove that Conjecture \ref{p;5pts} is true under the assumption that the subvariety \[Y = \{\langle y^3, xy^2, x^2y, x^3, ay^2+cx^2\rangle~|~[a,c]\in\bP^1_k\}\] discussed in Conjecture \ref{l;b=0} is \emph{not} compatibly split.

\begin{proof}[Proof of Conjecture \ref{p;5pts} assuming Conjecture \ref{l;b=0}]
To begin, let \[Y=\{\langle y^4,xy^2,x^2y,x^3,ax^2+bxy+cy^3, dx^2+exy+fy^3 \rangle~|~\textrm{rank} \begin{pmatrix}a & b & c \\d & e & f \end{pmatrix} = 2\}.\] (Recall that the moment polytope of $Y$ is the triangle with vertices $\tableau{{ \ }\\{ \ }\\{ \ }\\{ \ }&{ \ }}~,~\tableau{{ \ }\\{ \ }&{ \ }\\{ \ }&{ \ }}~,$ and $\tableau{{ \ }\\{ \ }\\{ \ }&{ \ }&{ \ }}~.$) We will show that $Y$ is compatibly split by showing that it is a component of the intersection of two known compatibly split subvarieties.

Let $Z_1 = \{\langle x,y^3\rangle\}\subseteq \Hilb^3(\bA^2_k)$ and let $W_1$ be the closure of the image of \[i_{1,0,1,Z_1}:\Hilb^1(\textrm{punctured x-axis})\times\Hilb^1(\bA^2_k\setminus\{xy=0\})\times Z_1\rightarrow \Hilb^4(\bA^2_k)\] where $i_{1,0,1,Z_1}$ is as in Proposition \ref{p;intuition}. Then $W_1$ is compatibly split and $W_1$ contains the point $I = \langle x,y^3\rangle\cap \langle x-b^2, y-b\rangle\cap \langle x-b, y\rangle$, $b\in k^{*}$. With respect to the GRevLex term order, $I$ has a Gr\"{o}bner basis $\{x^2+(-b+1)xy-bx, y^3-xy, xy^2-bxy\}$. Thus, $W_1$ contains \[I' = \textrm{lim}_{b\rightarrow 0}I = \langle x^2+xy, y^3-xy, xy^2\rangle.\] 
Notice that $I'\in Y$ and that $I'$ is not fixed by any subtorus of $T^2$. Because, $W_1$ is closed and $T^2$-invariant, $W_1$ contains the $T^2$-orbit closure of $I'$. Thus, $W_1$ contains the toric variety $Y$. 

Let $Z_2 =\{\langle y^3,xy^2,x^2y,x^3,ax^2+bxy+cy^2, dx^2+exy+fy^2 \rangle~|~\textrm{rank} \begin{pmatrix}a & b & c \\d & e & f \end{pmatrix} = 2 \}$ and let $W_2$ be the closure of the image of \[i_{0,1,0,Z_2}:\Hilb^1(\textrm{punctured y-axis})\times Z_2\rightarrow \Hilb^5(\bA^2_k)\] where $i_{0,1,0,Z_2}$ is as in Proposition \ref{p;intuition}. Then $W_2$ is compatibly split and $W_2$ contains the point $I = \langle y^3, xy^2, x^2-by^2,xy-by^2\rangle\cap \langle x,y-b\rangle$. With respect to the GRevLex term order, $I$ has a Gr\"{o}bner basis $\{x^2-xy, y^3+xy-by^2, xy^2\rangle\}$. Thus, \[I' = \textrm{lim}_{b\rightarrow 0}I = \langle x^2-xy, y^3+xy, xy^2\rangle.\] By the same reasoning as above, $Y$ is a subvariety of $W_2$. It follows that $Y$ is a component of the intersection $W_1\cap W_2$ and thus $Y$ is compatibly split. Similarly, the unique subvariety whose moment polytope is the triangle with vertices $\tableau{{ \ }\\{ \ }\\{ \ }&{ \ }&{ \ }}~,$ $\tableau{{ \ }&{ \ }\\{ \ }&{ \ }&{ \ }}~,$ and $\tableau{{ \ }\\{ \ }&{ \ }&{ \ }&{ \ }}$ is compatibly split.

Each remaining subvariety listed in the statement of the conjecture is the bone of an exterior face of the moment polytope of a compatibly split subvariety. (See Figure \ref{fig;5ptspolytope}.) Thus, each remaining subvariety listed in the statement of the conjecture is compatibly split.

By Lemma \ref{l;puncSplit}, Lemma \ref{l;badedge}, and an argument (using the moment polyhedron of $\Hilb^5(\bA^2_k)$) which is similar to the argument given in each of the $n=3$ and $n=4$ cases, any additional compatibly split subvarieties of $\Hilb^5_0(\bA^2_k)$ must be contained inside of \[\{\langle y^3,xy^2,x^2y,x^3,ay^2+bxy+cx^2\rangle~|~[a,b,c]\in \bP^2_k\}\] (which is the bone of the edge connecting $\tableau{{ \ }\\{ \ }&{ \ }\\{ \ }&{ \ }}~,$ $\tableau{{ \ }\\{ \ }\\{ \ }&{ \ }&{ \ }}~,$ and $\tableau{{ \ }&{ \ }\\{ \ }&{ \ }&{ \ }}$ ).

We already know that $\{a=0\}$ and $\{c=0\}$ are compatibly split subvarieties of this $\bP^2_k$. Thus, any additional $1$-dimensional compatibly split subvariety must be a $T^2$-invariant degree-$1$ curve. The only one left is $\{b = 0\}$ and this is not compatibly split if we assume that the conjecture that Conjecture \ref{l;b=0} is true, is correct.    
\end{proof}

\begin{remark}
Notice that when $n=4$, only the attractive (in the sense of Bia\l ynicki-Birula) fixed points are compatibly split. However when $n=5$, all fixed points (including the non-attractive one, $\langle x^3, xy, y^3\rangle$) are compatibly split.
\end{remark}


 \section{The affine patch $U_{\langle x,y^n\rangle}$}\label{s;goodPatch}

In this section we consider the open affine patch $U_{\langle x,y^n\rangle}\subseteq \Hilbn$ with the induced Frobenius splitting. In particular, we describe all compatibly split subvarieties of $U_{\langle x,y^n\rangle}$ and their scheme-theoretic defining ideals. We find degenerations of these subvarieties to Stanley-Reisner schemes, explicitly describe the associated simplicial complexes, and use these complexes to prove that certain compatibly split subvarieties of $U_{\langle x,y^n\rangle}$ are Cohen-Macaulay.

\subsection{A choice of coordinates}

In this subsection we describe our preferred coordinates on $U_{\langle x,y^n\rangle}\subseteq \Hilbn$ and show that these coordinates are a natural choice with respect to the $T^2$-action. 

\begin{lemma}\label{l;cij}
Every colength-$n$ ideal in $U_{\langle x,y^n\rangle}$ is generated by a set of polynomials of the form
\[\begin{array}{lllllllllllll}
  f_1& := & y^n&-&b_1y^{n-1}&-&b_2y^{n-2}&-&\cdots&-&b_{n-1}y&-&b_n\\
  f_2& := & xy^{n-1}&-&a_1y^{n-1}&-&c_{12}y^{n-2}&-&\cdots&-&c_{1(n-1)}y&-&c_{1n}\\
  f_3& := & xy^{n-2}&-&a_2y^{n-1}&-&c_{22}y^{n-2}&-&\cdots&-&c_{2(n-1)}y&-&c_{2n}\\
  &&&&\vdots&&&&&&\\
  f_n& := & xy&-&a_{n-1}y^{n-1}&-&c_{(n-1)2}y^{n-2}&-&\cdots&-&c_{(n-1)(n-1)}y&-&c_{(n-1)n}\\
  f_{n+1}& := & x&-&a_n y^{n-1}&-&c_{n2}y^{n-2}&-&\cdots&-&c_{n(n-1)}y&-&c_{nn}\end{array} \]
 where each $c_{ij}$ is the following polynomial in $a_1,\dots,a_n,b_1,\dots,b_n$:
\begin{enumerate}
\item If $i<j$ then $c_{ij} = \sum_{k=1}^{n-j+1}a_{k+i}b_{k+j-1}$.
\item If $i\geq j$ then $c_{ij} = a_{i-j+1} - \sum_{k=1}^{j-1}a_{k+i-j+1}b_k$.
\end{enumerate}
Thus, $U_{\langle x,y^n\rangle}\cong \Spec k[a_1,\dots,a_n,b_1,\dots,b_n]$. 
\end{lemma}

\begin{proof}
To prove the lemma, we show that the set $G = \{f_1,\dots,f_{n+1}\}$ is a Gr\"{o}bner basis under the Lex term order with $x\gg y$. Indeed, if $G$ is a Gr\"{o}bner basis then \[\textrm{init}\langle G\rangle = \langle \textrm{init}~G \rangle = \langle x,y^n\rangle\] and so $\langle G\rangle\in U_{\langle x,y^n\rangle}$. Thus, \[\Spec k[a_1,\dots,a_n,b_1,\dots,b_n]\cong \{\langle G\rangle~|~(a_1,\dots,a_n,b_1,\dots,b_n)\in \bA^{2n}_k\}\subseteq U_{\langle x,y^n\rangle}.\] As $U_{\langle x,y^n\rangle}$ is a $2n$-dimensional irreducible affine variety, it follows that \[\Spec k[a_1\dots,a_n,b_1,\dots,b_n]\cong U_{\langle x,y^n\rangle}.\]

It remains to show that $G = \{f_1,\dots,f_{n+1}\}$ is a Gr\"{o}bner basis. To begin, notice that each $S$-polynomial, $S(f_i,f_j)$ for $f_i,f_j\in G$, can be written as follows:
\begin{itemize}
\item If $1<i<j\leq n+1$ then \[S(f_i,f_j) =y^{j-i}f_j-f_i =-\sum_{k=1}^{j-i}a_{j-k}y^{j-i-k}f_1\]
\item If $i=1$ and $1<j\leq n+1$ then \[S(f_1,f_j) = y^{j-1}f_j-xf_1 = \sum_{k=1}^{j-1}b_{j-k}f_{j+1-k} -\sum_{k=1}^{j-1}a_{j-k}y^{j-1-k}f_1\]
\end{itemize}

Therefore, each $S$-polynomial can be written as a sum of the form $\sum_{k}m_kg_k$ where each $m_k$ is a monomial, each $g_k\in G$, and the initial terms of the summands are strictly decreasing (i.e. $\textrm{init}(m_ig_i) > \textrm{init}(m_jg_j)$ when $i<j$). Thus, each $S$-polynomial reduces to $0$ upon division by $G$ and so $G$ is a Gr\"{o}bner basis. 
\end{proof}

\begin{remark}
Recall that $\frak{m}_{\langle x,y^n\rangle}/ \frak{m}_{\langle x,y^n\rangle}^2$, the Zariski cotangent space of $U_{\langle x,y^n\rangle}$ at the $T^2$-fixed point $\langle x,y^n\rangle$, is generated by $a_1,\dots,a_n$, $b_1,\dots,b_n$ where each $a_i~(\textrm{mod}~\frak{m}^2_{\langle x,y^n\rangle})$ is a $T^2$-weight vector with weight $(-1,i-1)$ and each $b_i~(\textrm{mod}~\frak{m}^2_{\langle x,y^n\rangle})$ is a $T^2$-weight vector with weight $(0,-i)$ (see \cite{H2}). Thus, the coordinates $a_1,\dots,a_n,b_1,\dots,b_n$ on $U_{\langle x,y^n\rangle}$ yield generators of the $T^2$-weight spaces of the tangent space of $\Hilbn$ at $\langle x,y^n\rangle$. In this way, we see that our preferred coordinates are natural from the perspective of the torus action. See Figure \ref{fig;goodarrows} for a pictorial representation of $a_1,\dots,a_n$, $b_1,\dots,b_n$ using Haiman's arrows (see \cite{H2}).
\end{remark}

\begin{figure}[!h]
\begin{center}
\includegraphics[scale = 0.33]{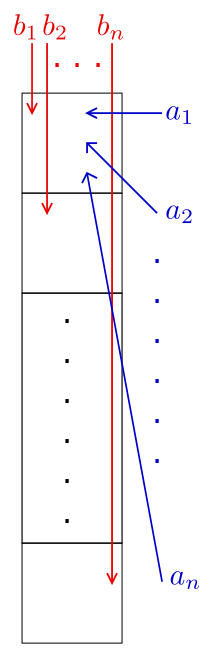}
\caption{Associating $a_1,a_2,\dots,a_n$, $b_1,b_2,\dots,b_n$ to Haiman's arrows (see \cite{H2}). Note that the vertical strip consists of $n$ boxes and is the standard set of the monomial ideal $\langle x,y^n\rangle$. The $T^2$-weights of $a_1,a_2,\dots,a_n,b_1,b_2,\dots,b_n$ are $(-1,0)$, $(-1,-1),\dots,(-1,-n+1)$, $(0,-1),(0,-2),\dots,(0,-n)$ as indicated by the arrows.}
\label{fig;goodarrows}
\end{center}
\end{figure}

\begin{definition}\label{d;mat}
Let $M_n$ denote the $(n\times n)$-matrix of coefficients $(-c_{ij})_{1\leq i,j\leq n}$ appearing in Lemma \ref{l;cij}. Set $c_{i1} := a_i$.
\end{definition}

\begin{example}
\[M_2 = \begin{pmatrix}-a_1&-a_2b_2\\-a_2&-(a_1-b_1a_2)\end{pmatrix},~~M_3 = \begin{pmatrix}-a_1&-(a_2b_2+a_3b_3)&-a_2b_3\\-a_2&-(a_1-b_1a_2)&-a_3b_3\\-a_3&-(a_2-b_1a_3)&-(a_1-b_1a_2-b_2a_3)\end{pmatrix}\]
\[M_4 = \begin{pmatrix}-a_1&-(a_2b_2+a_3b_3+a_4b_4)&-(a_2b_3+a_3b_4)&-a_2b_4\\-a_2&-(a_1-b_1a_2)&-(a_3b_3+a_4b_4)&-a_3b_4\\-a_3&-(a_2-b_1a_3)&-(a_1-b_1a_2-b_2a_3)&-a_4b_4\\-a_4&-(a_3-b_1a_4)&-(a_2-b_1a_3-b_2a_4)&-(a_1-b_1a_2-b_2a_3-b_3a_4)\end{pmatrix}\]
Notice that when $b_3=0$, $M_2$ is the upper left $2\times 2$ submatrix of $M_3$. Similarly, when $b_4=0$, $M_3$ is the upper left $3\times 3$ submatrix of $M_4$. By Lemma \ref{l;cij}, this nesting of matrices occurs in general. We make use of it in the next subsection.
\end{example}
 
\subsection{The induced splitting}\label{s;inducedsplitting}

Because $U_{\langle x,y^n\rangle}\cong \bA^{2n}$ and our splitting of $\Hilbn$ is a $(p-1)^{st}$ power, the induced splitting of $U_{\langle x,y^n\rangle}$ has the form $\textrm{Tr}(f_n^{p-1}\cdot)$, for some polynomial $f_n\in k[a_1,b_1,\dots,a_n,b_n]$. In this subsection, we compute $f_n$ and show that, under an appropriate weighting, the initial term of $f_n$ is the product of the variables $a_1b_1\cdots a_nb_n$.

To begin, consider $U_{\langle x,y^2\rangle}\cong \Spec k[a_1,b_1,a_2,b_2]$. Each closed point of $U_{\langle x,y^n\rangle}$ is an ideal $I$ generated by polynomials of the form
\[\begin{array}{lllll}
 y^2&-&b_1y&-&b_2\\
 xy&-&a_1y&-&a_2b_2\\
 x&-&a_2y&-&(a_1-b_1a_2)\end{array}\]
If $\Spec(k[a_1,b_1,a_2,b_2]/I)$ has non-trivial intersection with the $x$-axis, then $I+\langle y\rangle \neq \langle 1\rangle$. Thus, $-b_2 = 0$. Similarly, if $\Spec(k[a_1,b_1,a_2,b_2]/I)$ has non-trivial intersection with the $y$-axis, then $I+\langle x\rangle \neq \langle 1\rangle$. Thus,
\[\left|\begin{array}{ll}
 -a_1&-a_2b_2\\
 -a_2&-(a_1-b_1a_2)\end{array}\right| = -a_1b_1a_2+a_1^2-a_2^2b_2 = 0\]
Let $f_2 := a_1b_1a_2b_2-a_1^2b_2+a_2^2b_2^2$. Because there is a weighting such that the $\textrm{init}(f_2) = a_1b_1a_2b_2$, $\textrm{Tr}(f_2^{p-1}\cdot)$ is a splitting of $k[a_1,b_1,a_2,b_2]$ (see Theorem \ref{t;LMPinit}, i.e. \cite[Proposition 1.8]{LMP}, or Theorem \ref{t;knutson}, i.e. \cite[Theorem 2]{K}). 

Now, let $s_2$ denote the anticanonical section that determines the $T^2$-invariant splitting of $\Hilb^2(\bA^2_k)$ and let $D_2:=\{s_2=0\}$. The vanishing set of $f_2$ is $D_2 \cap U_{\langle x,y^2\rangle}$. As there are no non-constant, non-vanishing functions on $U_{\langle x,y^2\rangle}$, it follows that $\textrm{Tr}(f_2^{p-1}\cdot)$ must be the induced splitting of $U_{\langle x,y^2\rangle}$.  

More generally, we have the following situation.

\begin{proposition}\label{p;eqnRNC}
Let $M_n$ be as in Definition \ref{d;mat}. Let $f_n:=-b_n\textrm{det}(M_n)$.
\begin{enumerate}
\item The induced splitting of $U_{\langle x,y^n\rangle}\subseteq \Hilbn$ is given by $\textrm{Tr}(f_n^{p-1}\cdot)$.
\item For any $n\geq 1$, $f_n$ is a degree $2n$ polynomial such that $\textrm{Lex}_{a_n} \textrm{Revlex}_{b_n}(f_n) = b_na_nf_{n-1}$. Thus, under the weighting \[\textrm{Revlex}_{b_n},~\textrm{Lex}_{a_n},\dots,~\textrm{Revlex}_{b_1},~\textrm{Lex}_{a_1}\]
$\textrm{init}(f_n) = b_na_n\cdots b_1a_1$. In other words, $\{f_n = 0\}$ is a residual normal crossings divisor as defined in \cite{LMP}.
\end{enumerate}
\end{proposition}

\begin{remark}
By the weighting $\textrm{Revlex}_{b_n},~\textrm{Lex}_{a_n},\dots,~\textrm{Revlex}_{b_1},~\textrm{Lex}_{a_1}$ we mean the following: Let $m,n\in k[a_1,b_1,\dots,a_n,b_n]$ be monomials. We first compare $m$ and $n$ using $\textrm{Revlex}_{b_n}$. If they are indistinguishable then we compare $m$ and $n$ using $\textrm{Lex}_{a_n}$. If $m$ and $n$ are still indistinguishable, we compare them using $\textrm{Revlex}_{b_{n-1}}$, and so on. 

Put differently, $\textrm{Revlex}_{b_n},~\textrm{Lex}_{a_n},\dots,~\textrm{Revlex}_{b_1},~\textrm{Lex}_{a_1}$ is equivalent to a weighting of the variables where variables $a_1,\dots,a_n$ are weighted by $A_1,\dots,A_n$, variables $b_1,\dots, b_n$ are weighted by $B_1,\dots,B_n$ and \[-B_n\gg A_n \gg -B_{n-1}\gg {A_{n-1}}\gg\cdots\gg -B_1\gg A_1 \gg 0.\]
\end{remark}

\begin{proof}[Proof of Proposition \ref{p;eqnRNC}]
The proof is nearly identical to the $n=2$ case explained above. 

Let $I\in U_{\langle x,y^n\rangle}$. If $\Spec (k[x,y]/I)$ has non-trivial intersection with the x-axis then $I+\langle y \rangle\neq \langle 1\rangle$ and so $-b_n=0$. If $\Spec (k[x,y]/I)$ has non-trivial intersection with the y-axis then $I+\langle x \rangle \neq \langle 1\rangle$ and so the polynomials
\[\begin{array}{llllllllll}
 -&a_1y^{n-1}&-&c_{12}y^{n-2}&-&\cdots&-&c_{1(n-1)}y&-&c_{1n}\\
 -&a_2y^{n-1}&-&c_{22}y^{n-2}&-&\cdots&-&c_{2(n-1)}y&-&c_{2n}\\
 &&&&\vdots&&&&&\\
 -&a_{n-1}y^{n-1}&-&c_{(n-1)2}y^{n-2}&-&\cdots&-&c_{(n-1)(n-1)}y&-&c_{(n-1)n}\\
 -&a_n y^{n-1}&-&c_{n2}y^{n-2}&-&\cdots&-&c_{n(n-1)}y&-&c_{nn}\end{array} \]
must have a common solution. Thus, $\textrm{det}(M_n) = 0$. 

Let $f_n := -b_n\textrm{det}(M_n)$. We now show that there is a weighting such that the initial term of $f_n$ is the product of the variables $a_1b_1\cdots a_nb_n$. This will prove that $\textrm{Tr}(f_n^{p-1}\cdot)$ is a splitting of $U_{\langle x,y^n\rangle}$. 

By Lemma \ref{l;cij}, if we compute $\textrm{det}(M_n)$ using cofactors along the last column and group together terms involving $b_n$, we have 
\[\textrm{det}(M_n) = c_{nn}\textrm{det}(M_{n-1})+b_nR\]
where $R$ is a polynomial in $a_1,\dots,a_n,b_1,\dots,b_n$. Notice that $b_n$ does not appear anywhere in the polynomial $c_{nn}\textrm{det}(M_{n-1})$ and that $a_n$ does not appear anywhere in $\textrm{det}(M_{n-1})$. Thus, 
\[\textrm{Lex}_{a_n}\textrm{Revlex}_{b_n}(\textrm{det}(M_n)) = \textrm{Lex}_{a_n}(c_{nn}\textrm{det}(M_{n-1})) = a_nb_{n-1}\textrm{det}(M_{n-1}) = -a_nf_{n-1}\]
and so $\textrm{Lex}_{a_n}\textrm{Revlex}_{b_n}(f_n) = b_na_nf_{n-1}$. Therefore $\textrm{init}(f_n) = b_na_n\cdots b_1a_1$, with respect to the weighting \[\textrm{Revlex}_{b_n},~\textrm{Lex}_{a_n},\dots,~\textrm{Revlex}_{b_1},~\textrm{Lex}_{a_1}.\]

Let $s_n$ denote the anticanonical section that determines the $T^2$-invariant splitting of $\Hilb^n(\bA^2_k)$ and let $D_n:=\{s_n=0\}$. By construction, $f_n$ vanishes along $D_n \cap U_{\langle x,y^n\rangle}$. As there are no non-constant, non-vanishing functions on $U_{\langle x,y^n\rangle}$, $\textrm{Tr}(f_n^{p-1}\cdot)$ must be the induced splitting of $U_{\langle x,y^n\rangle}$.  

Finally, notice that $\textrm{deg}(\textrm{det}(M_n))\leq 2n-1$ (since the first column of $M_n$ consists of linear entries and the rest of $M_n$ has quadratic entries). Therefore, $\textrm{deg}(f_n)\leq 2n$. Because $a_1b_1\cdots a_nb_n$ is a term of $f_n$, $\textrm{deg}(f_n) = 2n$.
\end{proof}

\begin{remark}
In the proof, we use that $\textrm{init}(f_n) = a_1b_1\dots a_nb_n$ in order to prove that $\textrm{Tr}(f_n^{p-1}\cdot)$ is the induced splitting of $U_{\langle x,y^n\rangle}$. Though this isn't necessary, it is a convenient replacement for showing both (i) that $f_n$ does not vanish anywhere other than $\{s_n = 0\}$, and (ii) that $\textrm{Tr}(f_n^{p-1}1) = 1$, rather than some other constant $c\in k$. 
\end{remark}

\subsection{The compatibly split subvarieties of $U_{\langle x,y^n\rangle}$}\label{splitGoodPatch}

In this subsection, we describe all compatibly split subvarieties of $U_{\langle x,y^n\rangle}$. By \cite[Lemma 1.1.7]{BK}, it suffices to find all compatibly split subvarieties of $\Hilbn$ that have non-trivial intersection with $U_{\langle x,y^n\rangle}$.

 \begin{figure}[!h]
 \begin{center}
 \includegraphics[scale = 0.33]{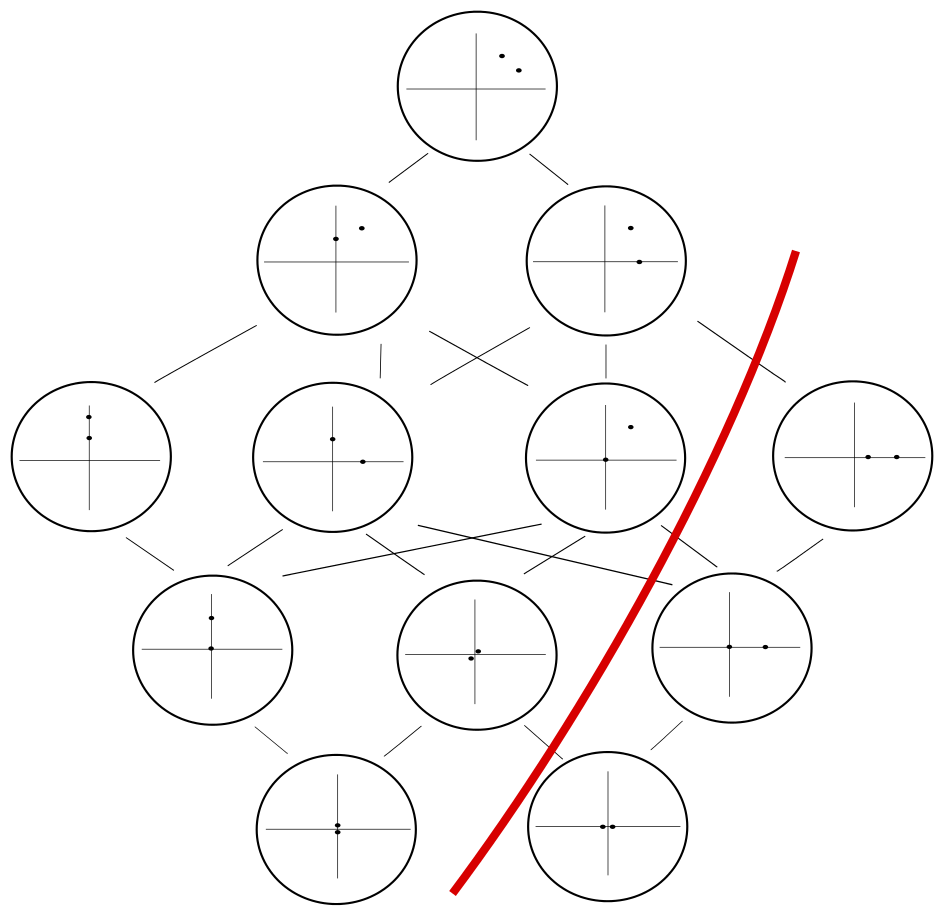}
 \caption{Intersecting the compatibly split subvarieties of $\Hilb^2(\bA^2_k)$ with $U_{\langle x,y^2\rangle}$. The subvarieties with stratum representatives appearing to the left of the curve have non-trivial intersection with $U_{\langle x,y^n\rangle}$. }
 \label{fig;linesplit}
 \end{center}
 \end{figure}

To begin, we introduce some notation which will be used to label the compatibly split subvarieties of $\Hilbn$ that have non-trivial intersection with $U_{\langle x,y^n\rangle}$.  

\begin{definition}\label{d;labels}
Consider the following two types of 4-tuples: $(s,u,t,+0)$ and $(s,u,t,+1)$, $s,u,t\in \bZ_{\geq 0}$. Label $Y\subseteq \Hilbn$ by $(s,u,t,+0)$ or by $(s,u,t,+1)$ if $Y$ is the closure, in $\Hilbn$, of the image of the map 
\[i_{s,t,W}: \Hilb^s(\textrm{y-axis}\setminus \{(0,0)\})\times \Hilb^t(\bA^2_k\setminus \{xy=0\})\times W\rightarrow \Hilbn\]
\[(I_1,I_2,I_3)\mapsto I_1\cap I_2\cap I_3\] for some $W\subseteq \Hilb^{u}(\bA^2_k)$ of the following form: 
\begin{enumerate}
\item Suppose the last index of the 4-tuple is $+0$.
\begin{enumerate}
\item If $u=0$ then $W$ is empty.
\item If $u\geq 1$ then $W$ is the $0$-dimensional subscheme consisting of the $T^2$-fixed point $\langle x,y^u\rangle\in \Hilb^u(\bA^2_k)$.
\end{enumerate}
\item Suppose the last index of the 4-tuple is $+1$.
\begin{enumerate}
\item If $u=0$ then $W = \Hilb^1(\textrm{x-axis}\setminus\{(0,0)\})$.
\item If $u\geq 2$ then $W$ is the subvariety of $\Hilb^u_0(\bA^2_k)$ that is pointwise fixed by the subtorus $T^1 = \{(t^{u-1},t)~|~t\in \bG_m\}$. That is, $W$ is the bone (see Definition \ref{d;bone}) of the exterior edge of the moment polyhedron of $\Hilb^u(\bA^2_k)$ connecting the moment polytope of $\langle x,y^u\rangle$ to the moment polytope of $\langle x^2, xy, y^{u-1}\rangle$. Working on the two open patches $U_{\langle x,y^u\rangle}$ and $U_{\langle x^2, xy, y^{u-1}\rangle}$, we see that $W\cong \bP^1_k$.
\end{enumerate}
We \emph{do not} associate any subvariety to the $4$-tuple $(s,1,t,+1)$.
\end{enumerate}
\end{definition}

\begin{remarks}
\begin{enumerate}
\item The last entry of the $4$-tuple is $+0$ when $\textrm{dim}(W)=0$ and is $+1$ when $\textrm{dim}(W)=1$. See Figure \ref{fig;stratreps} for a picture of a stratum representative of each subvariety $Y\subseteq \Hilbn$ labelled by $(s,u,t,+0)$ or by $(s,u,t,+1)$.
\item All subvarieties of $\Hilbn$ which are labelled by $(s,u,t,+0)$ or by $(s,u,t,+1)$ are irreducible. (Proof as before.) 
\item All subvarieties of $\Hilbn$ which are labelled by $(s,u,t,+0)$ or by $(s,u,t,+1)$ are compatibly split. (Proof: Apply Proposition \ref{p;intuition}.) 
\item All subvarieties $\Hilbn$ which are labelled by $(s,u,t,+0)$ or by $(s,u,t,+1)$ have non-trivial intersection with $U_{\langle x,y^n\rangle}$. Indeed, they all contain the point $\langle x,y^n\rangle$.
\end{enumerate}
\end{remarks}

\begin{figure}
\begin{center}
\includegraphics[scale = 0.4]{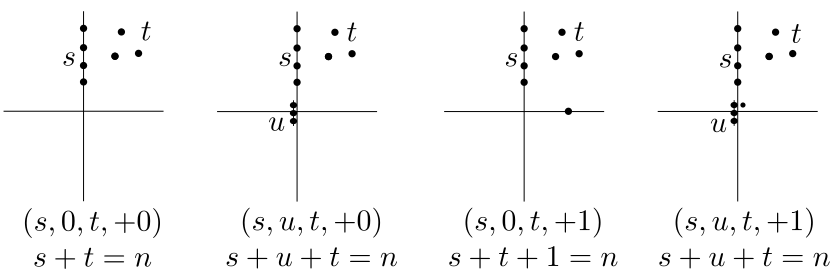}
\caption{Stratum representative pictures for subvarieties of $\Hilbn$ labelled by $(s,u,t,+0)$ or by $(s,u,t,+1)$.}
\label{fig;stratreps}
\end{center}
\end{figure}

Thus, if $Y\subseteq \Hilbn$ is of type $(s,u,t,+0)$ or $(s,u,t,+1)$, then $Y\cap U_{\langle x, y^n\rangle}$ is a (non-empty) compatibly split subvariety of $U_{\langle x,y^n\rangle}$. In fact, as the next proposition indicates, these are the only compatibly split subvarieties of $U_{\langle x,y^n\rangle}$.

\begin{proposition}\label{p;allsplit}
A subvariety $Z\subseteq U_{\langle x,y^n\rangle}$ is compatibly split if and only if $Z = Y\cap U_{\langle x,y^n\rangle}$ where $Y\subseteq\Hilbn$ is labelled by some $(s,u,t,+0)$ or some $(s,u,t,+1)$ as in Definition \ref{d;labels}.
\end{proposition}

Before proving the proposition, we consider a few lemmas.

\begin{lemma}\label{l;vertex}
A closed, $T^2$-invariant closed subvariety $Y\subseteq \Hilbn$ has non-trivial intersection with $U_{\langle x,y^n\rangle}$ if and only if $\langle x,y^n\rangle\in Y$.
\end{lemma}

\begin{proof}
The backwards direction is clear. For the forward direction, let $I\in Y \cap U_{\langle x,y^n\rangle}$. Consider the subtorus $S:=\{(\alpha^N,\alpha)~|~N>n, \alpha\in \bG_m\}\subseteq T^2$. Then, $\langle x,y^n\rangle$ lies in the $S$-orbit closure of $I$. As $Y$ is both closed and $T^2$-invariant, $\langle x,y^n\rangle\in Y$. 
\end{proof}

\begin{lemma}\label{l;notsplit}
Let $1\leq i\leq n-2$ and let $Y_i\subseteq U_{\langle x,y^n\rangle}$ be the subvariety defined by the ideal containing all of the variables except for $a_{i+1}$. That is, $J(Y_i) = \langle a_1,\dots,a_{i},a_{i+2},\dots,a_n$, $b_1,\dots,b_n\rangle\subseteq k[a_1,\dots,a_n,b_1,\dots,b_n]$. Then $Y_i$ is not compatibly split. 
\end{lemma}

\begin{proof}
Let $X$ be the subvariety of $\Hilbn$ labelled by $(0,n-1,1,+0)$. Let $X^{o}$ denote the image of \[\Hilb^1(\bA^2_k\setminus\{xy=0\})\times \{\langle x,y^{n-1}\rangle\}\rightarrow \Hilb^{n+1}(\bA^2_k).\] Then $X^{o}$ is open in $X$ and any $I\in X^{o}$ has the form $I = \langle x,y^{n-1}\rangle\cap \langle x-a,y-b\rangle$, $a,b\in k^{*}$. That is, $I$ is generated by the following polynomials:
\[y^n-by^{n-1},~~ xy^{n-1}-ay^{n-1},~~ xy^{n-2}-\frac{a}{b}y^{n-1}, \dots, xy^{n-1-i}-\frac{a}{b^{i}}y^{n-1}, \dots,~~ x-\frac{a}{b^{n-1}}y^{n-1}.\]

\noindent Thus the ideal defining $X$, $J(X)\subseteq k[a_1,\dots,a_n,b_1,\dots,b_n]$, is given by \[J(X) = \langle b_2,\dots,b_n, b_1a_2-a_1, b_1a_3-a_2,\dots, b_1a_n-a_{n-1}\rangle.\] Using Macaulay 2, we can compute that $J(X)$ has Gr\"{o}bner basis \[\{b_2,\dots,b_n\}\cup \{b_1a_j-a_{j-1}~|~2\leq j\leq n\}\cup \{a_ja_k-a_{j+1}a_{k-1}~|~1\leq j\leq n-2, ~3\leq k\leq n,~ j+1\leq k-1\}\] with respect to the weighting $\textrm{Revlex}_{b_n}, \textrm{Lex}_{a_n}, \cdots, \textrm{Revlex}_{b_1},\textrm{Lex}_{a_1}$. Therefore, \[\textrm{init}(J(X)) = \langle b_2,\dots,b_n, b_1a_2,\dots,b_1a_n, \{a_ja_k~|~1\leq j\leq n-2, ~3\leq k\leq n, j < k\}\rangle.\] 

Now let $Y_i$, $1\leq i\leq n-2$, be as in the statement of the proposition. Then $Y_i\nsubseteq X$. (In particular, $b_1a_{i+2}-a_{i+1}\in J(X)$ is not an element of $J(Y_i)$.) However, $\textrm{init}(Y_i)\subseteq \textrm{init}(X)$. By a theorem of Knutson (see Theorem \ref{t;yay}), $Y_i$ is not compatibly split.
\end{proof}

\begin{lemma}\label{l;puncSplit}
Let $Y\subseteq \Hilbn$ be a compatibly split subvariety that has non-trivial intersection with $U_{\langle x,y^n\rangle}$. If $Y$ is contained inside the punctual Hilbert scheme $\Hilb^n_0(\bA^2_k)$, then $Y$ is either the $0$-dimensional subvariety $\{\langle x,y^n\rangle\}$ or the $1$-dimensional subvariety that is pointwise fixed by the subtorus $T^1 = \{(t^{n-1},t)~|~t\in \bG_m\}$. In other words, $Y$ can be labelled by $(0,n,0,+0)$ or $(0,n,0,+1)$.
\end{lemma}

\begin{proof}
When $n\leq 2$, the result holds trivially. So, assume that $n\geq 3$.

$Y$ is projective because $Y$ is a closed subvariety of the punctual Hilbert scheme $\Hilb^{n}_0(\bA^2_k)$. Thus, $Y$'s moment polyhedron is a (compact) polytope, $P$. By Lemma \ref{l;vertex}, $P$ contains a vertex $v$ which is the moment polytope of the $T^2$-fixed point $\langle x,y^n\rangle$. Since $v$ is external in the moment polyhedron of $\Hilbn$, $v$ is an external vertex of $P$. 

Recall that an edge in the x-ray of the punctual Hilbert scheme $\Hilb^n_0(\bA^2_k)$ which is connected to $v$ must be in one of the following directions: \[(1,-1), (1,-2),\dots,(1,-(n-1))\] (since these are the weights of the $T^2$-action on the tangent space $T_{\langle x,y^n\rangle}\Hilb^n_0(\bA^2_k)$). Thus, any exterior edge of $P$ which is connected to $v$ must also be in one of these directions. Suppose that $Y_i\subseteq\Hilbn$ is the preimage of the edge in the direction $(1,-i)$, $1\leq i\leq n-2$. Then, $Y_i$ is pointwise fixed by the subtorus $T^1 = \{(\alpha^{i},\alpha)~|~\alpha\in \bG_m\}$. By Lemma \ref{l;cij}, $Y_i\cap U_{\langle x,y^n\rangle}$ is defined by the ideal $J_i := \langle b_1,\dots,b_n,a_1,\dots,a_{i},a_{i+2},\dots,a_n\rangle$. By Lemma \ref{l;notsplit}, $J_i$, $1\leq i\leq n-2$ is not compatibly split. Therefore, $Y_i$, $1\leq i\leq n-2$, is not compatibly split.

Next, recall that a subvariety which is the bone of any exterior face of $P$ is compatibly split (by Proposition \ref{p;extedge}). Therefore, by the above argument, the edges connected to $v$ in the directions of $(1,-1)$, $(1,-2)$, $\dots$, $(1,-(n-2))$ cannot be exterior faces of $P$. By compactness of $P$ and Lemma \ref{l;vertex}, $P$ is either (i) the point $v$ or (ii) the edge attached to $v$ in the direction $(1,-(n-1))$. 
\end{proof}

We are now ready to prove Proposition \ref{p;allsplit}.

\begin{proof}[Proof of Proposition \ref{p;allsplit}]
As discussed above, if $Y\subseteq \Hilbn$ is labelled by $(s,u,t,+0)$ or by $(s,u,t,+1)$, then $Y\cap U_{\langle x,y^n\rangle}$ is compatibly split. It remains to prove the converse of the proposition.

Suppose that $Y\subseteq \Hilbn$ is compatibly split and that $Y\cap U_{\langle x,y^n\rangle}\neq \emptyset$. We show that $Y$ can be labelled by $(s,u,t,+0)$ or by $(s,u,t,+1)$. 

By Proposition \ref{p;intuition}, $Y$ must be the closure of the image of the map 
\[i_{r,s,t,Z}:\Hilb^r(\textrm{x-axis}\setminus\{(0,0)\})\times \Hilb^s(\textrm{y-axis}\setminus\{(0,0)\})\times \Hilb^t(\bA^2_k\setminus \{xy=0\})\times Z\rightarrow \Hilbn\]
\[(I_1,I_2,I_3,I_4)\mapsto I_1\cap I_2\cap I_3\cap I_4\] 
for some integers $r,s,t\geq 0$ with $r+s+t\leq n$ and some compatibly split subvariety $Z\subseteq \Hilb^{n-r-s-t}(\bA^2_k)$ that is contained inside of $\Hilb_0^{n-r-s-t}(\bA_k^2)$. Note the following:

\begin{enumerate}
\item $Y\cap U_{\langle x,y^n\rangle}\neq \emptyset$ by assumption. Therefore, if $Z\neq\emptyset$, then $r=0$. If $Z = \emptyset$, then $r\leq 1$. 
\item If $Z\neq\emptyset$, then $Z\cap U_{\langle x,y^{n-r-s-t}\rangle}\neq \emptyset$.
\end{enumerate}

By Lemma \ref{l;puncSplit}, the only compatibly split subvarieties of $\Hilb^{n-r-s-t}(\bA^2_k)$ which are contained in $\Hilb^{n-r-s-t}_0(\bA^2_k)$ and have non-empty intersection with $U_{\langle x,y^{n-r-s-t}\rangle}$ are those be labelled by $(0,n-r-s-t,0,+0)$ or $(0,n-r-s-t,0,+1)$. This, along with 1. and 2., above gives the desired result.
\end{proof}

\begin{corollary}
The stratification of $U_{\langle x,y^n\rangle}$ by all of its compatibly split subvarieties has the shape of a square. That is, there are $d+1$ strata of dimension $d$ when $d\leq n$ and $(2n-d)+1$ strata of dimension $d$ when $d>n$. 
\end{corollary}

\begin{proof}
The dimension of the subvariety labelled by $(s,u,t,+0)$ (where $s+u+t=n$) is $s+2t$ and the dimension of a subvariety labelled by $(s,u,t,+1)$ (where $s+u+t = n$ if $u \geq 2$ and $s+u+t+1=n$ if $u=0$) is $s+2t+1$. So, when $d\leq n$, there is exactly one compatibly split subvariety of dimension $d$ for each value of $s$, $0\leq s\leq d$. This yields $d+1$ compatibly split subvarieties of dimension $d$. When $d>n$, there is exactly one compatibly split subvariety for each value of $s$, $0\leq s\leq 2n-d$. This yields $(2n-d)+1$ compatibly split subvarieties of dimension $d$.
\end{proof}

\subsection{A Gr\"{o}bner degeneration of the compatibly split strata}\label{s;degens}

We begin with the $n=2$ case. Recall that the induced splitting of $U_{\langle x,y^2\rangle}$ is $\textrm{Tr}(f_2^{p-1}\cdot)$ where $f_2 = a_1b_1a_2b_2-a_1^2b_2+a_2^2b_2^2$ and that, under an appropriate weighting, $\textrm{init}(f_2) = a_1b_1a_2b_2$. It follows by a theorem of Knutson (see Theorem \ref{t;knutson} or the original source \cite{K}) that every compatibly split ideal of $k[a_1,b_1,a_2,b_2]$ Gr\"{o}bner degenerates to a squarefree monomial ideal. This is shown explicitly in Figure \ref{fig;2ptsvars}. The leading terms with respect to the weighting $\textrm{Revlex}_{b_2}, \textrm{Lex}_{a_2}, \textrm{Revlex}_{b_1}, \textrm{Lex}_{a_1}$ are underlined.

\begin{figure}[H]
\begin{center}
\includegraphics[scale = 0.31]{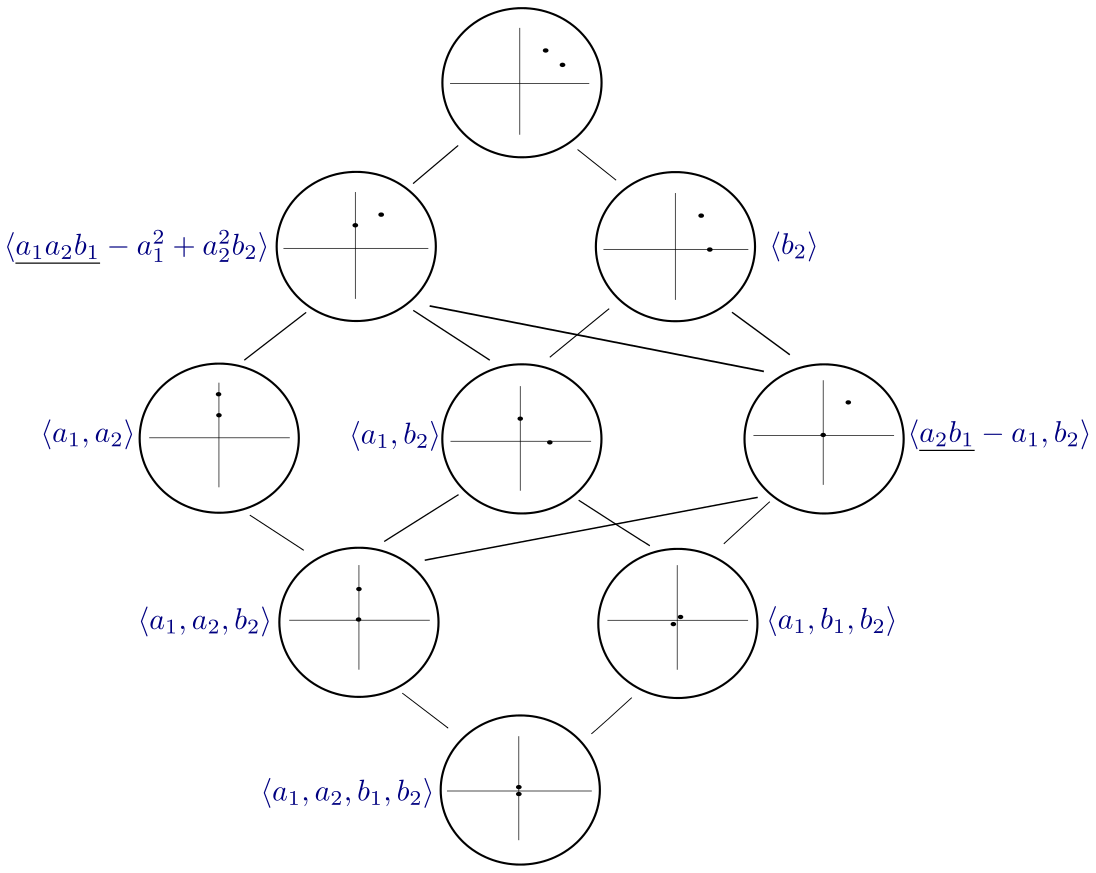}
\caption{Compatibly split ideals of $k[a_1,b_1,a_2,b_2]$ with leading terms underlined.}
\label{fig;2ptsvars}
\end{center}
\end{figure}

Let $f_n$ be as in Proposition \ref{p;eqnRNC} and fix the weighting \[\textrm{Revlex}_{b_n}, \textrm{Lex}_{a_n},\dots, \textrm{Revlex}_{b_1}, \textrm{Lex}_{a_1}.\] By Proposition \ref{p;eqnRNC}, $\textrm{Tr}(f_n^{p-1}\cdot)$ is the induced splitting of $U_{\langle x,y^n\rangle} = \Spec(k[a_1,b_1,\dots,a_n,b_n])$. Furthermore, $f_n$ is a polynomial of degree $2n$ with the property that $\textrm{init}(f_n) = a_1b_1\cdots a_nb_n$. It follows by Knutson's theorem that each compatibly split subvariety of $U_{\langle x,y^n\rangle}$ Gr\"{o}bner degenerates to a Stanley-Reisner scheme. In this subsection, we explicitly describe these degenerations. In addition, we find the defining ideal $J_Y\subseteq k[a_1,b_1,\dots,a_n,b_n]$ of each compatibly split $Y\subseteq U_{\langle x,y^n\rangle}$.

\begin{remark}
 Let $Y'\subseteq \Hilbn$ be compatibly split and labelled by $(s,u,t,+i)$, for $i=0$ or $1$. In the remainder of this subsection, we abuse notation: If $Y = Y'\cap U_{\langle x,y^n\rangle}$ then write $Y = (s,u,t,+0)$.
\end{remark}

\begin{proposition}\label{degens}
Let $Y\subseteq U_{\langle x,y^n\rangle}$ be compatibly split. The scheme-theoretic defining ideal $J_Y$ is given by the following information.
\begin{enumerate}
\item If $Y=(0,0,n-1,+1)$, then $J_Y = \langle b_n\rangle$.
\item If $Y=(s,0,n-s,+0)$, then $J_Y = \langle ((n-s+1)\times (n-s+1))\textrm{-minors of } M_n\rangle$.
\item If $s\geq 1$ and $Y = (s,0,n-s-1,+1)$, then $J_Y = \langle b_n,$ $((n-s)\times (n-s))\textrm{-minors of } M_{n-1}\rangle$.
\item If $u\geq 1$ and $Y = (s,u,n-s-u,+0)$, then $J_Y = \langle b_n,\dots,$ $b_{n-u+1},$ $(M_n)_{n,n},\dots,(M_n)_{n,n-u+1},$ $((n-s-u+1)\times (n-s-u+1))\textrm{-minors of } M_n\rangle$. 
\item If $u\geq 2$ and $Y = (s,u,n-s-u,+1)$, then $J_Y = \langle b_n,\dots$, $b_{n-u+1}$, $(M_n)_{n,n},\dots,(M_n)_{n,n-u+2},$ $((n-s-u+1)\times (n-s-u+1))\textrm{-minors of }M_{n-1}\rangle$. 
\end{enumerate}
Consider $U_{\langle x,y^{n-1}\rangle}\times \bA^1_{b_n}\times \bA^1_{a_n}$ with the splitting $Tr(f_{n-1}^{p-1}b_n^{p-1}a_n^{p-1}\cdot)$. Then, $\textrm{Lex}_{a_n}\textrm{Revlex}_{b_n}(Y)$ is a compatibly split subvariety of $U_{\langle x,y^{n-1}\rangle}\times \bA^1_{b_n}\times \bA^1_{a_n}$. Furthermore, the initial schemes $\textrm{Lex}_{a_n}\textrm{Revlex}_{b_n}(Y)$ are described with the following information: 
\begin{enumerate}
\item If $Y\subseteq \{b_n=0\}$ then $\textrm{Revlex}_{b_n}(Y) = Y\times 0_{b_n}$. (Note: $Y\subseteq \{b_n = 0\}$ if $Y=(s,u,n-s-u,+0)$, $u\geq 1$, or $Y = (s,u,n-s-u,+1)$, $u\geq 2$ or if $Y = (s,0,n-s-1,+1)$.)
\item $\textrm{Revlex}_{b_n}(0,0,n,+0) = (0,0,n-1,+1)\times \bA^1_{b_n}$.
\item $\textrm{Revlex}_{b_n}(s,0,n-s,+0) = [(s,0,n-s-1,+1)\cup (s-1,1,n-s,+0)]\times \bA^1_{b_n}$, for $1\leq s<n$.
\item $\textrm{Revlex}_{b_n}(n,0,0,+0) = (n-1,1,0,+0)\times \bA^1_{b_n}$.
\item $\textrm{Lex}_{a_n}(s,0,n-s-1,+1) = (s,0,n-s-1,+0)\times 0_{b_n}\times \bA^1_{a_n}$.
\item $\textrm{Lex}_{a_n}(s,1,n-s-1,+0) = ((s,0,n-s-1,+0)\times 0_{b_n}\times 0_{a_n})\cup ((s,0,n-s-2,+1)\times 0_{b_n}\times \bA^1_{a_n})$, for $s<n-1$.
\item $\textrm{Lex}_{a_n}(s,u,n-s-u,+0) = ((s,u-1,n-s-u,+0)\times 0_{b_n}\times 0_{a_n})\cup ((s,u,n-s-u-1,+1)\times 0_{b_n}\times \bA^1_{a_n})$, for $2\leq u<n$ and $s<n-u$.
\item $\textrm{Lex}_{a_n}(n-u,u,0,+0) = (n-u,u-1,0,+0)\times 0_{b_n}\times 0_{a_n}$, $u\geq 1$.
\item $\textrm{Lex}_{a_n}(s,u,n-s-u,+1) = (s,u-1,n-s-u,+0)\times 0_{b_n}\times \bA^1_{a_n}$, for $u\geq 2$.
\end{enumerate}
Note that the $4$-tuples appearing on the right hand side of the last five equations label subvarieties of $U_{\langle x,y^{n-1}\rangle}$, \emph{not} $U_{\langle x,y^n\rangle}$.
\end{proposition}

\begin{remark}
To gain some intuition into the degenerations appearing in the proposition, recall the following (discussed in Subsection \ref{s;lexRev}):
\[\textrm{Revlex}_{b_n}(Y) = \left\{
        \begin{array}{ll}
            Y\times 0_{b_n}, & \quad \textrm{if}~Y\subseteq \{b_n = 0\} \\
            (Y\cap \{b_n = 0\})\times \bA^1_{b_n},& \quad \textrm{if}~Y\nsubseteq \{b_n=0\} 
        \end{array}
    \right.\]

Because both $Y\subseteq U_{\langle x,y^n\rangle}$ and $\{b_n = 0\}$ are compatibly split, so is $Y\cap \{b_n = 0\}$. Therefore, we can easily check that each $\textrm{Revlex}_{b_n}(Y)$ is as claimed in the proposition as we have already found all compatibly split subvarieties of $U_{\langle x,y^n\rangle}$ (see Proposition \ref{p;allsplit}). 

To describe $\textrm{Lex}_{a_n}\textrm{Revlex}_{b_n}(Y)$ for each compatibly split $Y\in U_{\langle x,y^n\rangle}$, it remains to describe $\textrm{Lex}_{a_n}(Y')$ for each compatibly split subvariety $Y'\subseteq (U_{\langle x,y^n\rangle}\cap\{b_n = 0\})$. See Figure \ref{fig;lexdeg} to visualize the $\textrm{Lex}_{a_n}$ degenerations. 
\end{remark}

\begin{figure}[!h]
 \begin{center}
 \includegraphics[scale = 0.32]{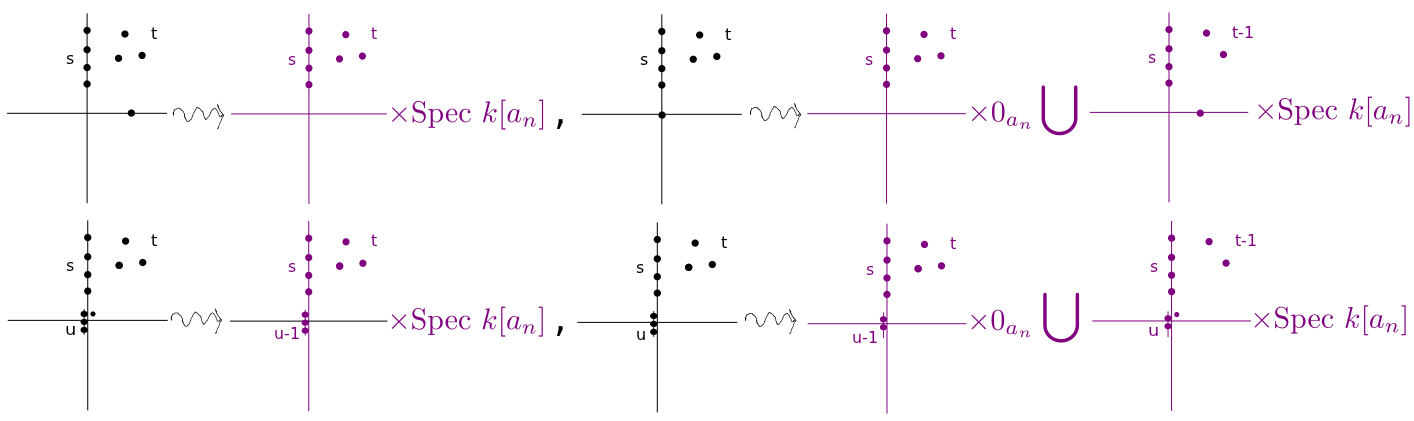}
\caption{The degenerations $\textrm{Lex}_{a_n}(Y')$ described in Proposition \ref{degens}.}
\label{fig;lexdeg}
 \end{center}
 \end{figure}

The proof of the proposition makes use of the following lemma.

\begin{lemma}\label{l;lexinit}
Let $Y\subseteq \{b_n = 0\}$ be a compatibly split subvariety of $U_{\langle x,y^n\rangle}$. Let $J(Y)$ denote the ideal associated to $Y$. Then, with respect to the weighting given in Proposition \ref{p;eqnRNC}, the reduced Gr\"{o}bner basis of $J(Y)$ consists of $T^2$-homogeneous polynomials of the form 
\[G = \{b_n,g_1,\dots,g_r,a_ng_{r+1}+h_1,\dots,a_ng_{r+l}+h_l\}\]
where none of $g_1,\dots,g_{r+l},h_1,\dots,h_l$ have any nonzero terms which are divisible by $a_n$ or $b_n$ and none of $g_1,\dots,g_r$ have any constant terms. Furthermore, if $a_n\notin J(Y)$, then none of $g_{r+1},\dots, g_{r+l}$ have any constant terms.
\end{lemma}

\begin{proof}
Let $G = \{q_1,\dots,q_r\}$ be the reduced Gr\"{o}bner basis for $J(Y)$. Because $Y$ is $T^2$-invariant, we may assume that each $q_i$ is $T^2$-homogeneous. Because $Y\subseteq \{b_n = 0\}$, we may assume that $q_1 = b_n$ and that each $q_i$, for $i>1$, has no terms involving $b_n$. 

By the nature of the weighting in Proposition \ref{p;eqnRNC}, $\textrm{init}(q_i)$ must be one of the terms of $\textrm{Lex}_{a_n}(q_i)$. Suppose that $\textrm{init}(q_i) = a_n^sm$ for some monomial $m$. By Theorem \ref{t;knutson} (i.e. \cite[Theorem 2]{K}), $s = 0$ or $1$ (since $\textrm{init}(I)$ is a squarefree monomial ideal and $G$ is a reduced Gr\"{o}bner basis). Thus,
\[G = \{b_n,g_1,\dots,g_r,a_ng_{r+1}+h_1,\dots,a_ng_{r+l}+h_l\}\]
where none of $g_1,\dots,g_{r+l},h_1,\dots,h_l$ have any nonzero terms which are divisible by $a_n$ or $b_n$.

Next, consider the subtorus $T^1 = \{(t,t^{\epsilon})~|~t\in \bG_m\}$ for $\epsilon>0$ very small. Then, $(t,t^{\epsilon})\cdot a_i = t^{-1+\epsilon(i-1)}a_i$, and $(t,t^{\epsilon})\cdot b_i = t^{-\epsilon i}b_i$. Since each $q_i$ is $T^1$-homogenous, every monomial in $q_i$ has a strictly negative weight. Thus, no term of $q_i$ is constant.

Now suppose that $a_n\notin J(Y)$. We show that none of $g_{r+1},\dots, g_{r+l}$ have any constant terms. So, suppose otherwise and assume that $g_{r+1}$ contains a constant term. Then every term of $a_ng_{r+1}+h_1$ must have $T^2$-weight $(-1,n-1)$ (the $T^2$-weight of $a_n$). As no positive linear combination of the $T^2$-weights of $a_1,\dots,a_{n-1},b_1,\dots,b_n$ is $(-1,n-1)$, it follows that $a_ng_{r+1}+h_1 = a_n$, a contradiction since $a_n\notin J(Y)$.
\end{proof}

We now prove the proposition.

\begin{proof}[Proof of Proposition \ref{degens}] 
We proceed by induction on $n$. The $n=1$ case is trivial. The $n=2$ case has been covered earlier. The outline of the proof is as follows: 

\begin{enumerate}
\item Let $J(Y)\subseteq k[a_1,b_1,\dots,a_n,b_n]$ denote the ideal defining $Y\subseteq U_{\langle x,y^n\rangle}$, and let $J_Y$ denote the ideal of $Y$ \emph{as claimed in the proposition}. We show that $J_Y\subseteq J(Y)$ for each compatibly split $Y\subseteq U_{\langle x,y^n\rangle}$. 
\item Let $Y\subseteq \{b_n = 0\}$ be a compatibly split subvariety of $U_{\langle x,y^n\rangle}$ and let $Y''$ denote the initial scheme $\textrm{Lex}_{a_n}(Y)$ \emph{as claimed in the proposition}. Because $Y''\subseteq U_{\langle x,y^{n-1}\rangle}\times \bA^1_{a_n}\times \bA^1_{b_n}$, we know $J(Y'')$ by induction. We show that $J(Y'') = \textrm{Lex}_{a_n}(J_Y) = \textrm{Lex}_{a_n}(J(Y))$. Because $J_Y\subseteq J(Y)$ (by 1.), it follows that $J_Y = J(Y)$ (see Lemma \ref{l;init}).
\item Let $Y\nsubseteq \{b_n = 0\}$ be a compatibly split subvariety of $U_{\langle x,y^n\rangle}$ and let $Y''$ denote $\textrm{Revlex}_{b_n}(Y)$ \emph{as claimed in the proposition}. By 2., we know $J(Y'')$. We show that $J(Y'') = \textrm{Revlex}_{b_n}J_Y = \textrm{Revlex}_{b_n}J(Y)$. Because $J_Y\subseteq J(Y)$ (by 1.), it follows that $J_Y = J(Y)$.
\end{enumerate}

We begin by showing that $J_Y\subseteq J(Y)$ for each compatibly split $Y\subseteq U_{\langle x,y^n\rangle}$. Recall that each $I\in U_{\langle x,y^n\rangle}$ is an ideal generated by polynomials
\[\begin{array}{lllllllllll}
 y^n&-&b_1y^{n-1}&-&b_2y^{n-2}&-&\cdots&-&b_{n-1}y&-&b_n\\
 xy^{n-1}&-&a_1y^{n-1}&-&c_{12}y^{n-2}&-&\cdots&-&c_{1(n-1)}y&-&c_{1n}\\
 xy^{n-2}&-&a_2y^{n-1}&-&c_{22}y^{n-2}&-&\cdots&-&c_{2(n-1)}y&-&c_{2n}\\
 &&&&\vdots&&&&&&\\
 xy&-&a_{n-1}y^{n-1}&-&c_{(n-1)2}y^{n-2}&-&\cdots&-&c_{(n-1)(n-1)}y&-&c_{(n-1)n}\\
 x&-&a_n y^{n-1}&-&c_{n2}y^{n-2}&-&\cdots&-&c_{n(n-1)}y&-&c_{nn}\end{array} \]
where each $c_{ij}$ is a polynomial in $a_1,\dots,a_n$, $b_1,\dots,b_n$ (described in Lemma \ref{l;cij}). Let $M_n$ be the $n\times n$ matrix where $(M_n)_{i,j} = -c_{ij}$, $i\neq 1$, and $(M_n)_{i,1} = -a_i$. 

\noindent \textbf{Suppose} $\mathbf{Y = (0,0,n-1,+1)}$: For each $I\in Y$, $\Spec k[x,y]/I$ has non-trivial intersection with the x-axis. Thus, $I+\langle y\rangle\neq \langle 1\rangle$ and so $b_n = 0$. Thus, $J_Y:=\langle b_n\rangle \subseteq J(Y)$.

\noindent \textbf{Suppose} $\mathbf{Y = (s,0,n-s,+0),~ s\geq 1}$: Let $U_{\langle x,y^n\rangle}\setminus\Delta$ denote the open set where all of the $n$ points are in distinct locations. For each $I\in Y\cap (U_{\langle x,y^n\rangle}\setminus\Delta)$, $\Spec k[x,y]/I$ intersects the y-axis in $s$ distinct locations. Let $p_1,\dots,p_s$ denote the $y$-coordinates of these locations. Then, the set of equations
\[\begin{array}{llllllllllll}
 -&a_1y^{n-1}&-&c_{12}y^{n-2}&-&\cdots&-&c_{1(n-1)}y&-&c_{1n}&=&0\\
 -&a_2y^{n-1}&-&c_{22}y^{n-2}&-&\cdots&-&c_{2(n-1)}y&-&c_{2n}&=&0\\
 &&&&\vdots&&&&&\\
 -&a_{n-1}y^{n-1}&-&c_{(n-1)2}y^{n-2}&-&\cdots&-&c_{(n-1)(n-1)}y&-&c_{(n-1)n}&=&0\\
 -&a_n y^{n-1}&-&c_{n2}y^{n-2}&-&\cdots&-&c_{n(n-1)}y&-&c_{nn}&=&0\end{array} \]
has (at least) $s$ solutions, $(p_1^{n-1},p_1^{n-2},\dots,p_1,1),\dots,(p_s^{n-1},p_s^{n-2},\dots,p_s,1)$. Therefore, $\textrm{rank}(M_n)\leq n-s$ and so the $((n-s+1)\times (n-s+1))\textrm{-minors of } M_n$ vanish. Thus, $J_Y:=\langle ((n-s+1)\times (n-s+1))\textrm{-minors of } M_n\rangle \subseteq J(Y)$.

\noindent \textbf{Suppose} $\mathbf{Y = (s,0,n-s-1,+1)}$: Then $Y\subseteq (0,0,n-1,+1)$ and $Y\subseteq (s,0,n-s,+0)$. Thus, $\langle b_n, ((n-s+1)\times(n-s+1))\textrm{-minors of }M_n\rangle\subseteq J(Y)$. 

Since $b_n =0$ on $Y$, $c_{in}=0$ for $1\leq i\leq n-1$ and $-c_{ij} = (M_{n-1})_{ij}$ for $1\leq i,j\leq n-1$ (see Lemma \ref{l;cij}). Consider the non-empty open set $\{I\in Y~|~I+\langle x,y\rangle = \langle 1\rangle \}$. Because each $c_{in}$ vanishes, $c_{nn}$ cannot vanish. Furthermore, since the $((n-s+1)\times(n-s+1))\textrm{-minors of }M_n$ vanish and $c_{nn}$ does not, the  $((n-s)\times (n-s))\textrm{-minors of } M_{n-1}$ must all vanish. So, $J_Y := \langle b_n, ((n-s)\times (n-s))\textrm{-minors of } M_{n-1}\rangle \subseteq J(Y)$.

\noindent \textbf{Suppose} $\mathbf{Y = (s,u,n-s-u,+0), u\geq 1}$: Then $Y\subseteq (s+u,0,n-s-u,+0)$ and so the $((n-s-u+1)\times (n-s-u+1))\textrm{-minors of } M_n$ are in $J(Y)$. Now let $I\in Y$. By construction, $I\subseteq\langle x,y^u\rangle$. Thus, for all $f\in I$, we have $\frac{\partial^i f}{\partial y^i}((0,0)) = 0$ for $0\leq i\leq u-1$. It follows that $b_n,\dots,b_{n-u+1}, (M_n)_{n,n},\dots,(M_n)_{n,n-u+1}$ must all be $0$. Thus, $J_Y := \langle b_n,\dots,b_{n-u+1},$ $(M_n)_{n,n},$ $\dots,$ $(M_n)_{n,n-u+1},$ $((n-s-u+1)\times (n-s-u+1))\textrm{-minors of } M_n\rangle\subseteq J(Y)$.

\noindent \textbf{Suppose} $\mathbf{Y = (s,u,n-s-u,+1), u\geq 2}$: Then $Y\subseteq (s+u-1,0,n-s-u+1,+1)$ and $Y\subseteq (s,u-1,n-s-u+1,+0)$. So, $\langle b_n,\dots b_{n-u+2},(M_n)_{n,n},\dots, (M_n)_{n,n-u+2}, ((n-s-u+1)\times (n-s-u+1))\textrm{-minors of } M_{n-1}\rangle \subseteq J(Y)$. Furthermore, because each $I\subseteq Y$ is the ideal of a configuration of points where at least $u$ are supported at the origin, $b_n,\dots,b_{n-u+1}$ must all be $0$. Thus, $J_Y := \langle b_n,\dots,b_{n-u+1},(M_n)_{n,n},\dots,(M_n)_{n,n-u+2},$ $(n-s-u+1)\times (n-s-u+1)-\textrm{minors of }M_{n-1}\rangle\subseteq J(Y)$.

\emph{Therefore, $J_Y\subseteq J(Y)$ for each compatibly split $Y\subseteq U_{\langle x,y^n\rangle}$. This completes step 1. in the outline of the proof.}

Next, suppose that $Y\subseteq \{b_n = 0\}$. Let $Y''$ denote $\textrm{Lex}_{a_n}(Y)$ \emph{as claimed in the statement of the proposition}. By induction, we know $J(Y'')$. We now show that $J(Y'') = \textrm{Lex}_{a_n} J_Y = \textrm{Lex}_{a_n} J(Y)$.

\noindent \textbf{Suppose} $\mathbf{Y = (0,0,n-1,+1)}$: Then $J_Y = \langle b_n\rangle$ and $Y'' = (0,0,n-1,+0)\times 0_{b_n}\times \bA^1_{a_n}$. So, $J(Y'') = \langle b_n\rangle = \textrm{Lex}_{a_n}J_Y\subseteq \textrm{Lex}_{a_n}J(Y)$ and thus, $\textrm{Lex}_{a_n}(Y)\subseteq Y''$. Since (i) $\textrm{Lex}_{a_n}(Y)$ is non-empty, (ii) $\textrm{dim}(Y) = \textrm{dim}(Y'')$, and (iii) $Y''$ is irreducible, it follows that $\textrm{Lex}_{a_n}(Y) = Y''$. Thus, $J(Y'') = \textrm{Lex}_{a_n}J_Y = \textrm{Lex}_{a_n}J(Y)$.

\noindent \textbf{Suppose} $\mathbf{Y = (s,0,n-s-1,+1), s\geq 1}$: Then, $J_Y = \langle b_n, ((n-s)\times (n-s))\textrm{-minors of } M_{n-1}\rangle$ and $Y'' = (s,0,n-s-1,+0)\times 0_{b_n}\times \bA^1_{a_n}$. By induction, $J(Y'') = \langle b_n, ((n-s)\times (n-s))\textrm{-minors of } M_{n-1}\rangle$. So, since no $a_n$ appears in the generating set of $J_Y$, we have $J(Y'') = \textrm{Lex}_{a_n}J_Y \subseteq \textrm{Lex}_{a_n}J(Y)$. As in the previous case, there is equality $J(Y'') = \textrm{Lex}_{a_n}J_Y = \textrm{Lex}_{a_n}J(Y)$ by the irreducibility of $Y''$.

\noindent \textbf{Suppose} $\mathbf{Y = (s,1,n-s-1,+0), s<n-1}$: Then $J_Y$ is the ideal $\langle b_n,(M_n)_{n,n},((n-s)\times (n-s))\textrm{-minors of } M_n\rangle$ and \[Y'' = ((s,0,n-s-1,+0)\times 0_{b_n}\times 0_{a_n})\cup ((s,0,n-s-2,+1)\times 0_{b_n}\times \bA^1_{a_n}).\] 
When $s=0$, $J_Y = \langle b_n, (M_n)_{n,n}, \textrm{det}(M_n)\rangle = \langle b_n, (M_n)_{n,n}\rangle$,
where the second equality can be seen by computing $\textrm{det}(M_n)$ using cofactors along the last column. Thus, \[J(Y'') = \langle a_n,b_n\rangle\cap \langle b_n,b_{n-1}\rangle \subseteq \textrm{Lex}_{a_n} J_Y \subseteq \textrm{Lex}_{a_n} J(Y).\]
Now suppose that $s\geq 1$.  
\[\textrm{Lex}_{a_n}J_Y\supseteq\langle b_n,a_nb_{n-1}, \{\textrm{Lex}_{a_n}m~|~m\in\{((n-s)\times (n-s))\textrm{-minors of } M_n\}\}\rangle.\]
Now, the set $S:=\{((n-s)\times (n-s))\textrm{-minors of } M_n\}$ contains the set of all $((n-s)\times (n-s))\textrm{-minors of } M_{n-1}$. In addition, if $A$ denotes the southwest $((n-1)\times (n-1))$-submatrix of $M_n$, then $S$ contains the set of $((n-s)\times (n-s))\textrm{-minors of } A$. To better understand some of the $((n-s)\times (n-s))\textrm{-minors of } A$, we consider another matrix $\tilde{A}$. 

Let $A_i$ denote the $i^{th}$ column of $A$. Define $\tilde{A}$ to be the $((n-1)\times(n-1))$-matrix whose first column is $A_1$ and whose $i^{th}$ column is $A_i+b_{i-1}A_1$ for $2\leq i\leq n-1$. Then, any $((n-s)\times (n-s))$-minor of $A$ which involves the column $A_1$ agrees with the corresponding minor of $\tilde{A}$. Furthermore, if $m$ is an $((n-s)\times (n-s))$-minor of $\tilde{A}$ involving both the first column and the last row, then we may use Lemma \ref{l;cij} to see that $m$ has the form \[m = a_nm'+a_{n}b_{n}r_1+r_2\] where $m'$ is an $(n-1-s)\times(n-1-s)$-minor of $M_{n-2}$, $r_1, r_2\in k[a_1,\dots,a_n,b_1,\dots,b_n]$, and $r_2$ has no terms involving $a_n$. Using this, we have the following:
\[\begin{array}{lll}
\textrm{Lex}_{a_n}J_Y& \supseteq &\langle b_n,a_nb_{n-1},((n-s)\times (n-s))\textrm{-minors of } M_{n-1}, \\
&&~~~\{{a_n}m'~|~m'\in \{((n-1-s)\times(n-1-s))\textrm{-minors of } M_{n-2}\}\}\rangle\\
& = &\langle b_n,a_n,((n-s)\times (n-s))\textrm{-minors of } M_{n-1}\rangle ~~\cap \\
&&\langle b_n, b_{n-1}, ((n-1-s)\times(n-1-s))\textrm{-minors of } M_{n-2},\\
&&~~~~~((n-s)\times (n-s))\textrm{-minors of } M_{n-1}\rangle\\
\end{array}\]
Let $J_2$ denote the second ideal in the intersection above. Since both $b_{n-1}\in J_2$ and the $((n-1-s)\times(n-1-s))\textrm{-minors of } M_{n-2}$ are in $J_2$, we see that (again by Lemma \ref{l;cij}) we needn't include the $((n-s)\times (n-s))\textrm{-minors of } M_{n-1}$ as part of the generating set of $J_2$. Thus,
 \[\begin{array}{lll}
\textrm{Lex}_{a_n}J_Y&\supseteq &\langle b_n,a_n,((n-s)\times (n-s))\textrm{-minors of } M_{n-1}\rangle ~~\cap \\
&&~~~\langle b_n, b_{n-1}, ((n-1-s)\times(n-1-s))\textrm{-minors of } M_{n-2}\rangle\\
\end{array}\]
and so $J(Y'')\subseteq \textrm{Lex}_{a_n}J_Y\subseteq \textrm{Lex}_{a_n}J(Y)$. Thus, $\textrm{Lex}_{a_n}(Y)\subseteq Y''$.

For any $s\geq 0$, let $C_1 = (s,0,n-s-1,+0)\times 0_{b_n}\times 0_{a_n}$ and $C_2 = (s,0,n-s-2,+1)\times 0_{b_n}\times \bA^1_{a_n}$ denote the two components of $Y''$. We now show that both $C_1\subseteq \textrm{Lex}_{a_n}Y$ and $C_2\subseteq \textrm{Lex}_{a_n}Y$, and thus that $Y'' = \textrm{Lex}_{a_n}Y$. Because $C_1$ and $C_2$ and both irreducible and of the same dimension as $Y$, it suffices to find points $q\in C_1\setminus C_2$ and $q'\in C_2\setminus C_1$ which are also in $\textrm{Lex}_{a_n}Y$.

Let $q$ be given by the coordinates $a_1 = \cdots = a_n = 0$, $b_1 = b_2 = \dots = b_{n-2} = 0$, $b_{n-1} = 1$, $b_n =0$. By induction, we know the defining ideals of $C_1$ and $C_2$. Using these ideals, we can check that $q\in C_1\setminus C_2$. The point $q$ corresponds to the colength-$n$ ideal $\langle x, y^n-y\rangle$. Thus, $q\in Y$. To see that $q\in \textrm{Lex}_{a_n}(Y)$, we construct a curve $\gamma(t)$, $0\leq t\leq 1$, in the total family of the degeneration such that $\gamma(1) = (a_1(1),\dots,a_n(1),b_1(1),\dots,b_n(1)) = q\in Y$ and $\gamma(0)\in \textrm{Lex}_{a_n}Y$. Define $\gamma(t)$ as follows: \[\gamma(t) = (a_1(1),\dots,a_{n-1}(1),ta_n(1),b_1(1),\dots,b_n(1)).\] Then $\gamma(0) = q$ and $q\in \textrm{Lex}_{a_n}Y$ as desired. Thus, $C_1\subseteq \textrm{Lex}_{a_n}Y$. 

Let $q'$ be given by the coordinates $a_1 = \cdots = a_{n-1} = 0$, $a_n = 1$, and $b_1 = \cdots = b_n = 0$. Then $q'\in C_2\setminus C_1$. The point $q'$ corresponds to the colength-$n$ ideal $\langle x-y^{n-1}, y^n\rangle$. Thus, $q'\in Y$. Notice that this means that $a_n\notin J(Y)$. Therefore, by Lemma \ref{l;lexinit}, $J(Y)$ has a Gr\"{o}bner basis of $T^2$-homogeneous polynomials of the form $\{g_1,\dots,g_r,a_ng_{r+1}+h_1,\dots,a_ng_{r+l}+h_l\}$ such that $a_n$ does not appear in any of $g_1,\dots,g_{r+l},h_1,\dots,h_l$, and $g_{r+1},\dots,g_{r+l}$ have no constant terms. Thus, $\textrm{Lex}_{a_n}J(Y) = \langle g_1,\dots,g_r,a_ng_{r+1},\dots,a_ng_{r+l}\rangle$ and we see that $q'\in \textrm{Lex}_{a_n}Y$. Therefore $C_2\subseteq \textrm{Lex}_{a_n}Y$.

\noindent \textbf{Suppose} $\mathbf{Y = (s,u,n-s-u,+0),~ 2\leq u<n,~ s<n-u}$: Then $J_Y$ is the ideal $\langle b_n,\dots$, $b_{n-u+1}$, $(M_n)_{n,n}$, $\dots,(M_n)_{n,n-u+1}$, $((n-s-u+1)\times (n-s-u+1))\textrm{-minors of } M_n\rangle$ and $Y'' = ((s,u-1,n-s-u,+0)\times 0_{b_n}\times 0_{a_n})\cup ((s,u,n-s-u-1,+1)\times 0_{b_n}\times \bA^1_{a_n}$.

Because $b_n,\dots,b_{n-u+1}\in J_Y$, we may refer to Lemma \ref{l;cij} to see that the generating set of $J_Y$ may be altered by replacing $(M_n)_{n,n},\dots,(M_n)_{n,n-u+2}$ with $(M_{n-1})_{n-1,n-1},\dots,(M_{n-1})_{n-1,n-u+1}$. Note that $a_n$ does not appear in any of $(M_{n-1})_{n-1,n-1},\dots,(M_{n-1})_{n-1,n-u+1}$. Thus,
\[\begin{array}{lll}
\textrm{Lex}_{a_n}J_Y&\supseteq&\langle b_n,\dots,b_{n-u+1},(M_{n-1})_{n-1,n-1},\dots,(M_{n-1})_{n-1,n-u+1},\textrm{Lex}_{a_n}(M_n)_{n,n-u+1}, \\
&&~~~~~~\{\textrm{Lex}_{a_n}m~|~m\in ((n-s-u+1)\times (n-s-u+1))\textrm{-minors of } M_n\}\rangle\\
&=&\langle b_n,\dots,b_{n-u+1},(M_{n-1})_{n-1,n-1},\dots,(M_{n-1})_{n-1,n-u+1},a_nb_{n-u}, \\
&&~~~~~~\{\textrm{Lex}_{a_n}m~|~m\in ((n-s-u+1)\times (n-s-u+1))\textrm{-minors of } M_n\}\rangle\\
\end{array}\]
We follow the previous case and replace the $((n-s-u+1)\times (n-s-u+1))\textrm{-minors of } M_n$ by a specific subset of them (utilizing the matrix $\tilde{A}$) to get
\[\begin{array}{lll}
\textrm{Lex}_{a_n}J_Y&\supseteq&\langle b_n,\dots,b_{n-u+1},(M_{n-1})_{n-1,n-1},\dots,(M_{n-1})_{n-1,n-u+1},a_nb_{n-u}, \\
&&~~~~~((n-s-u+1)\times (n-s-u+1))\textrm{-minors of } M_{n-1},\\
&&~~~~~\{a_nm'~|~m'\in ((n-u-s)\times(n-u-s))\textrm{-minors of } M_{n-2}\}\rangle\\
&=&\langle b_n,\dots,b_{n-u+1},(M_{n-1})_{n-1,n-1},\dots,(M_{n-1})_{n-1,n-u+1},a_n, \\
&&~~~~~((n-s-u+1)\times (n-s-u+1))\textrm{-minors of } M_{n-1}\rangle ~~\cap\\
&&\langle b_n,\dots,b_{n-u},(M_{n-1})_{n-1,n-1},\dots,(M_{n-1})_{n-1,n-u+1},\\
&&~~~~~((n-u-s)\times(n-u-s))\textrm{-minors of } M_{n-2}\rangle\\
\end{array}\]

Note that the final equality holds for the same reason as in the previous case: Let $J_2$ denote the second ideal in the intersection above. Since both $b_{n-1}\in J_2$ and the $((n-u-s)\times(n-u-s))\textrm{-minors of } M_{n-2}$ are in $J_2$, we may use Lemma \ref{l;cij} to see that the $((n-u-s+1)\times (n-u-s+1))\textrm{-minors of } M_{n-1}$ are also in $J_2$. 

Thus, $J(Y'')\subseteq \textrm{Lex}_{a_n}J_Y\subseteq \textrm{Lex}_{a_n}J(Y)$.

To show that $Y'' \subseteq \textrm{Lex}_{a_n}Y$, we copy the argument from the previous case. Let $C_1 = (s,u-1,n-s-u,+0)\times 0_{b_n}\times 0_{a_n}$ and $C_2 = (s,u,n-s-u-1,+1)\times 0_{b_n}\times \bA^1_{a_n}$. Let $q$ be given by the coordinates $a_1 = \cdots = a_n = 0$, $b_1= \cdots = b_{n-u-1} = 0$, $b_{n-u} = 1$, and $b_{n-u+1}=\cdots = b_n = 0$. Then $q\in C_1\setminus C_2$ and, by the same argument as in the previous case, $q\in \textrm{Lex}_{a_n}Y$. Let $q'$ be given by the coordinates $a_1 = \cdots = a_{n-1} = 0$, $a_n = 1$, and $b_1 = \cdots b_n = 0$. Then $q'\in C_2\setminus C_1$ and, by the same argument as in the previous case, $q'\in \textrm{Lex}_{a_n} Y$.

Thus, $J(Y'') = \textrm{Lex}_{a_n}J_Y = \textrm{Lex}_{a_n}J(Y)$.

\noindent \textbf{Suppose} $\mathbf{Y = (n-u,u,0,+0), u\geq 1}$: Then, $J_Y = \langle b_n,\dots,b_{n-u+1}, a_1,\dots,a_n\rangle$ and $Y'' = (n-u,u-1,0,+0)\times 0_{b_n}\times 0_{a_n}$. Notice that $J(Y'') = J_Y = \textrm{Lex}_{a_n}J_Y\subseteq \textrm{Lex}_{a_n}J(Y)$. By the irreducibility of $Y''$, $J(Y'') = \textrm{Lex}_{a_n}J_Y = \textrm{Lex}_{a_n}J(Y)$.  

\noindent \textbf{Suppose} $\mathbf{Y = (s,u,n-s-u,+1), u\geq 2}$: Then, $J_Y = \langle b_n,\dots,b_{n-u+1},(M_n)_{n,n},\dots,(M_n)_{n,n-u+2},$ $((n-s-u+1)\times (n-s-u+1))\textrm{-minors of }M_{n-1}\rangle$ and $Y'' = (s,u-1,n-s-u,+0)\times 0_{b_n}\times \bA^1_{a_n}$. Because $b_n,\dots,b_{n-u+1}\in J_Y$, we may refer to Lemma \ref{l;cij} to see that the generating set of $J_Y$ may be altered by replacing $(M_n)_{n,n},\dots,(M_n)_{n,n-u+2}$ with $(M_{n-1})_{n-1,n-1},\dots,(M_{n-1})_{n-1,n-u+1}$. As $a_n$ now appears nowhere in the generating set of $J_Y$, we see that $J(Y'') = J_Y =  \textrm{Lex}_{a_n} J_Y \subseteq J(Y)$. Thus, $\textrm{Lex}_{a_n}Y\subseteq Y''$. Again, since (i) $\textrm{Lex}_{a_n}(Y)\neq\emptyset$, (ii) $\textrm{dim}(Y) = \textrm{dim}(Y'')$, and (iii) $Y''$ is irreducible, we get the equality $\textrm{Lex}_{a_n}Y = Y''$. Thus, $J(Y'') = \textrm{Lex}_{a_n}J_Y = \textrm{Lex}_{a_n}J(Y)$.

\emph{So, $J(Y'') = \textrm{Lex}_{a_n}J_Y = \textrm{Lex}_{a_n}J(Y)$ for all compatibly split $Y\subseteq \{b_n = 0\}$. Furthermore, since $J_Y\subseteq J(Y)$, it follows that $J_Y = J(Y)$.}

Finally, let $Y$ be a compatibly split subvariety of $U_{\langle x,y^n\rangle}$ and let $Y''$ denote $\textrm{Revlex}_{b_n}Y$ \emph{as claimed in the statement of the proposition}. By step 2. (completed above), we know $J(Y'')$. We now show that $J(Y'') = \textrm{Revlex}_{b_n}J_Y = \textrm{Revlex}_{b_n}J(Y)$. As explained in the remark prior to this proof, we need only treat the cases where $Y\nsubseteq \{b_n=0\}$. That is, if $Y\subseteq \{b_n\}$ then $\textrm{Revlex}_{b_n}Y = Y$ and we have already shown that $J_Y = J(Y)$. 

So, suppose that $Y\nsubseteq \{b_n=0\}$. Then $\textrm{Revlex}_{b_n}Y = (Y\cap \{b_n=0\})\times \bA^1_{b_n}$. Because $Y\cap \{b_n=0\}$ is compatibly split, we can use Proposition \ref{p;allsplit} to very that 
\begin{itemize}
\item $\textrm{Revlex}_{b_n}(0,0,n,+0) = (0,0,n-1,+1)\times \bA^1_{b_n}$, 
\item $\textrm{Revlex}_{b_n}(s,0,n-s,+0) = [(s,0,n-s-1,+1)\cup (s-1,1,n-s,+0)]\times \bA^1_{b_n}$, for $s\geq 1$, and 
\item $\textrm{Revlex}_{b_n}(n,0,0,+0) = (n-1,1,0,+0)\times \Spec (k[b_n])$.
\end{itemize}
Therefore, to see that $J_Y = J(Y)$ for each $Y\nsubseteq \{b_n=0\}$, it remains to show that $J(Y'')\subseteq \textrm{Revlex}_{b_n}(J_Y)\subseteq \textrm{Revlex}_{b_n}(J(Y))$.

\noindent \textbf{Suppose} $\mathbf{Y = (0,0,n,+0)}$: Then $Y = U_{\langle x,y^n\rangle}$. Thus, $J(Y) = J_Y = \langle 0\rangle$ and the result is clear. 

\noindent \textbf{Suppose} $\mathbf{Y = (s,0,n-s,+0),~s\geq 1}$: Then $J_Y = \langle ((n-s+1)\times (n-s+1))\textrm{-minors of } M_n\rangle$ and $Y'' = [(s,0,n-s-1,+1)\cup (s-1,1,n-s,+0)]\times \bA^1_{b_n}$. 

Let $C_1 = (s,0,n-s-1,+1)$, $C_2 = (s-1,1,n-s,+0)$, and let $A$ denote the $(n\times (n-1))$-matrix with top $((n-1)\times(n-1))$-submatrix equal to $M_{n-1}$ and bottom row equal to the southwest $(1\times (n-1))$-submatrix of $M_n$. By Lemma \ref{l;cij}, it follows that
\[\begin{array}{lll}
J(C_1) &=& \langle b_n, ((n-s)\times (n-s))\textrm{-minors of}~M_{n-1}\rangle\\
J(C_2) &=& \langle b_n, (M_n)_{n,n}, ((n-s+1)\times(n-s+1))\textrm{-minors of}~A\rangle\\  
\end{array}\]
Notice that $b_n$ appears nowhere in the $((n-s)\times (n-s))\textrm{-minors of}~M_{n-1}$, $(M_n)_{n,n}$, or $((n-s+1)\times(n-s+1))\textrm{-minors of}~A$.

Furthermore,
\[\begin{array}{lll}
\textrm{Revlex}_{b_n} J_Y&\supseteq&\langle \{(M_n)_{n,n}m~|~m\in ((n-s)\times (n-s))\textrm{-minor of}~ M_{n-1}\},\\
&&~~~~~((n-s+1)\times(n-s+1))\textrm{-minors of}~ A\rangle\\
&=&\langle ((n-s)\times (n-s))\textrm{-minors of}~ M_{n-1}\rangle~~\cap\\
&&~~~~~\langle (M_n)_{n,n}, ((n-s+1)\times(n-s+1))\textrm{-minors of}~ A\rangle\\
&=&J(Y'')
\end{array}\]

\noindent \textbf{Suppose} $\mathbf{Y = (n,0,0,+0)}$: Then $J_Y = \langle a_1,\dots,a_n\rangle$ and $Y'' = (n-1,1,0,+0)\times \bA^1_{b_n}$. Notice that $J(Y'') = J_Y = \textrm{Revlex}_{b_n}J_Y\subseteq \textrm{Revlex}_{b_n}J(Y)$. 

\emph{In each of the above cases, $\textrm{Revlex}_{b_n}J_Y = \textrm{Revlex}_{b_n}J(Y)$ and so $J_Y = J(Y)$. This completes the proof.}
\end{proof}

\subsection{Some combinatorics and the geometric consequences}\label{s;combs}

Let $Y\subseteq U_{\langle x,y^n\rangle}$ be a compatibly split subvariety. Recall that, with respect to the weighting given in Proposition \ref{p;eqnRNC}, $\textrm{init}(Y)$ is a Stanley-Reisner scheme. Let $\Delta_Y$ denote the associated simplicial complex. In this subsection, we provide an explicit description of each $\Delta_Y$. We then use this description to show that for certain compatibly split $Y\subseteq U_{\langle x,y^n\rangle}$, $\Delta_Y$ is a vertex-decomposable ball. Therefore $\textrm{init}(Y)$ is Cohen-Macaulay, and by semicontinuity, so is $Y$. 

We begin with the $n=2$ case. In Figure \ref{fig;posetcomplex} we associate the compatibly split subvarieties of $U_{\langle x,y^2\rangle}$ to unions of faces of the $3$-simplex by first degenerating (with respect to the $\textrm{Revlex}_{b_2}$, $\textrm{Lex}_{a_2}$, $\textrm{Revlex}_{b_1}$, $\textrm{Lex}_{a_1}$ weighting) and then applying the Stanley-Reisner recipe. Note that the simplex in the figure has been ``unfolded'' to make it easier to label the faces.

\begin{figure}[!h]
\begin{center}
\includegraphics[scale = 0.325]{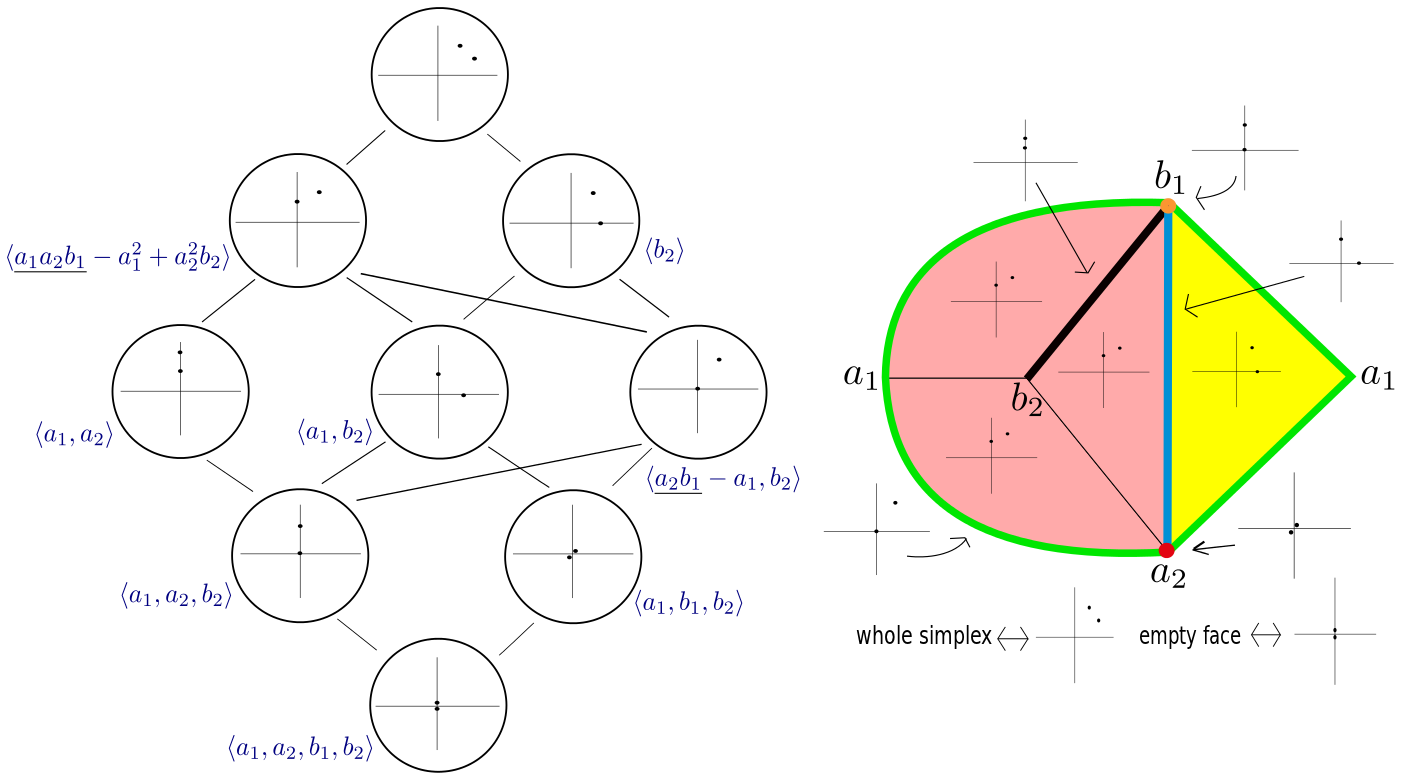}
\caption{Associating compatibly split subvarieties to simplicial complexes}
\label{fig;posetcomplex}
\end{center}
\end{figure}

We now proceed to the case when $n$ is arbitrary.

\begin{definition}\label{d;fullwords}
Fix some positive integer $n$.

Consider the following letters: $a$, $\hat{a}$, $\uparrow$. Consider the following segments of letters:
\begin{center}
$a\uparrow,~~~\hat{a}\uparrow,~~~aa\uparrow,~~~  aa,~~~  \hat{a}$
\end{center}

Define a \emph{full word} to be a word in segments $a$$\uparrow$, $\hat{a}$$\uparrow$, $aa$$\uparrow$, $aa$, $\hat{a}$ of the form
\begin{center}
(word in $a\uparrow,~\hat{a}\uparrow,~aa\uparrow$)$~~|~~$(word in $aa,~\hat{a}$)$~~|~~$($0$ or $1$ lone copy of $a$)
\end{center}
such that $\#a+\#\hat{a} = n$.

Let $Y$ be a subvariety of $U_{\langle x,y^n\rangle}$ of type $(s,u,t,+0)$ or $(s,u,t,+1)$. Define \emph{a full word associated to Y} to be a word in segments $a$$\uparrow$, $\hat{a}$$\uparrow$, $aa$$\uparrow$, $aa$, $\hat{a}$ of the form
\begin{center}
(word in $a\uparrow,~\hat{a}\uparrow,~aa\uparrow$)$~~|~~$(word in $aa,~\hat{a}$)$~~|~~$($a$ iff $Y$ is of type $(s,u,t,+1)$)
\end{center}
such that
\begin{itemize}
\item $\#a\uparrow+\#aa\uparrow+\#aa = t,~~\#\hat{a}\uparrow+\#aa\uparrow = s,~~\#aa+\#\hat{a} = u$, if $Y = (s,u,t,+0)$.
\item $\#a\uparrow+\#aa\uparrow+\#aa = t,~~\#\hat{a}\uparrow+\#aa\uparrow = s,~~\#aa+\#\hat{a} = u-1$, if $Y = (s,u,t,+1)$ and $u\geq 2$.
\item $\#a\uparrow+\#aa\uparrow+\#aa = t,~~\#\hat{a}\uparrow+\#aa\uparrow = s,~~\#aa+\#\hat{a} = 0$, if $Y = (s,0,t,+1)$.
 \end{itemize}
\end{definition}

\begin{remark}\label{r;fw}
The condition $\#\hat{a}+\#a = n$ ensures that every full word is really a full word associated to $Y$ for some compatibly split $Y\subseteq U_{\langle x,y^n\rangle}$. (Proof: Suppose that $w$ is a full word without a lone copy of $a$ at the end. Let $\#a\uparrow+\#aa\uparrow+\#aa = t,~~\#\hat{a}\uparrow+\#aa\uparrow = s,~~\#aa+\#\hat{a} = u$. Then, $n = \#\hat{a}+\#a = \#a\uparrow+\#\hat{a}\uparrow+2\cdot\#aa\uparrow+2\cdot\#aa+\#\hat{a} = s+u+t$. Thus, $w$ is a full word associated to $Y = (s,u,t,+0)$. A similar argument works when there is a lone copy of $a$ at the end of a full word.)
\end{remark}

\begin{proposition}\label{p;pipedreams}
Let $Y$ be a compatibly split subvariety of $U_{\langle x,y^n\rangle}$. The collection of full words associated to $Y$ are in one-to-one correspondence with the components of $\textrm{init}(Y)$. Thus, the full words associated to $Y$ are in one-to-one correspondence with the facets of $\Delta_Y$.

More precisely, let $FW_Y$ denote the set of full words associated to $Y$ and let $IN_Y$ denote the set of components of $\textrm{init}(Y)$. Then, there is a well-defined, bijective map \[m_Y: \{w_Y~|~w_Y\in FW_Y\}\rightarrow \{S\subseteq k[a_1,\dots,a_n,b_1,\dots,b_n]~|~\Spec(k[S])\in IN_Y\}\]
defined in the following way:
\begin{enumerate}
\item Let $w_Y$ denote a full word associated to $Y$. By construction, $\#a+\#\hat{a} = n$. From left to right, number the letters $a$ and $\hat{a}$ in $w_Y$ from $1$ to $n$. (Eg. $aa\uparrow\hat{a}\uparrow$ becomes $a_1a_2\uparrow\hat{a_3}\uparrow$.)
\item Replace each $\uparrow$ by $b_i$ such that $a_i$ or $\hat{a}_i$ (same $i$ as in $b_i$) now appears immediately to the left of $b_i$. (Eg. $a_1a_2\uparrow\hat{a_3}\uparrow$ becomes $a_1a_2b_2\hat{a_3}b_3$.)
\item Delete all $\hat{a}_i$ appearing in the new word. The set of letters in the resulting word is a subset $S\subseteq \{a_1,\dots,a_n,b_1,\dots,b_n\}$. (Eg. $a_1a_2b_2\hat{a_3}b_3$ becomes $a_1a_2b_2b_3$.) Define $m_Y(w_Y):=S$.
\end{enumerate}
\end{proposition}

 \begin{example}
Let $n=3$ and let $Y = (1,1,1,+0)$. The full words associated to $Y$ are: $aa$$\uparrow$$\hat{a}$, $\hat{a}$$\uparrow$$aa$, $\hat{a}$$\uparrow$$a$$\uparrow$$\hat{a}$, $a$$\uparrow$$\hat{a}$$\uparrow$$\hat{a}$. The facets of $\Delta_Y$ are given by $\{a_1,a_2,b_2\}$, $\{b_1,a_2,a_3\}$, $\{b_1,a_2,b_2\}$, and $\{a_1,b_1,b_2\}$. Notice that these subsets are the images of the full words associated to $Y$ under the map $m_Y$. 
\end{example}

We now prove Proposition \ref{p;pipedreams}.

\begin{proof}[Proof of Proposition \ref{p;pipedreams}]
It suffices to show that $m_Y$ is well-defined and bijective for each compatibly split subvariety $Y\subseteq U_{\langle x,y^n\rangle}$. We proceed by induction on $n$.

Suppose $n=1$. Then $Y$ is one of $(1,0,0,+0)$, $(0,1,0,+0)$, $(0,0,1,+0)$, or $(0,0,0,+1)$. In each case, there is exactly one full word associated to $Y$ and exactly one component of $\textrm{init}(Y) = Y$.

Now let $n$ be arbitrary and suppose that $Y$ is a compatibly split subvariety of $U_{\langle x,y^n\rangle}$.

\noindent\textbf{Suppose that} $\mathbf{Y = (s,u,t,+1),~ u\geq 2:}$ By Proposition \ref{degens}, \[\textrm{Lex}_{a_n}\textrm{Revlex}_{b_n}Y = (s,u-1,t,+0)\times 0_{b_n}\times \bA^1_{a_n}.\] Let $Y' = (s,u-1,t,+0)$. Let $S'\subseteq \{a_1,\dots,a_{n-1},b_1,\dots,b_{n-1}\}$ be such that $\Spec k[S']$ is a component of $\textrm{init}(Y')$. Then, $\Spec k[S]$, $S\subseteq \{a_1,\dots,a_n,b_1,\dots,b_n\}$, is a component of $\textrm{init}(Y)$ if and only if $S = S'\cup\{a_n\}$ for some $S'$. Also, by the definition of full words, $w_Y$ is a full word associated to $Y$ if an only if $w_Y = w_{Y'}a$ for some $w_{Y'}$, a full word associated to $Y'$. Thus, we have
\begin{center}
 \begin{xy}
 (0,20)*+{\{w_Y~|~w_Y~\textrm{a full word associated to}~Y\}}="a"; (93,20)*+{\{w_{Y'}a~|~w_{Y'}~\textrm{a full word associated to } Y'\}}="b";%
 (0,0)*+{\{S~|~\Spec(S)~\textrm{a component of }\textrm{init}(Y)\}}="c"; (93,0)*+{\{S'\cup{a_n}~|~\Spec(S')~\textrm{a component of }\textrm{init}(Y')\}}="d";%
 {\ar "a";"b"}?*!/_2mm/{\simeq};
 {\ar "a";"c"};{\ar "b";"d"};%
 {\ar "d";"c"}?*!/_2mm/{\simeq};%
\end{xy}
\end{center}


\noindent where $m_{Y'}$ determines the vertical map on the right. The diagram commutes if $m_Y$ is the vertical map on the left. Because $m_{Y'}$ is well defined and bijective (by induction), so is $m_Y$.

\noindent\textbf{Suppose that }$\mathbf{Y = (s,0,t,+1):}$ By Proposition \ref{degens}, \[\textrm{Lex}_{a_n}\textrm{Revlex}_{b_n}Y = (s,0,t,+0)\times 0_{b_n}\times \bA^1_{a_n}.\] By an identical argument as above, $m_Y$ is well defined and bijective.

\noindent\textbf{Suppose that }$\mathbf{Y = (s,u,t,+0),~ t,u\geq 1:}$ Then, \[\textrm{Lex}_{a_n}\textrm{Revlex}_{b_n}Y = ((s,u-1,t,+0)\times 0_{b_n}\times 0_{a_n})\cup (s,u,t-1,+1)\times 0_{b_n}\times \bA^1_{a_n}.\] Let $Y_1  = (s,u-1,t,+0)$ and $Y_2 = (s,u,t-1,+1)$. Then, $\Spec k[S]$, $S\subseteq \{a_1,\dots,a_n,b_1,\dots,b_n\}$, is a component of $\textrm{init}(Y)$ if and only if either (i) $S = S_1$ for some $S_1\subseteq \{a_1,\dots,a_{n-1},b_1,\dots,b_{n-1}\}$ where $\Spec k[S_1]$ is a component of $\textrm{init}(Y_1)$, or (ii) $S = S_2\cup\{a_n\}$ for some $S_2\subseteq \{a_1,\dots,$ $a_{n-1},$ $b_1,\dots,b_{n-1}\}$ where $\Spec k[S_2]$ is a component of $\textrm{init}(Y_2)$. 

Now, a full word associated to $Y$, $w_Y$, either ends with $\hat{a}$ or $aa$. By the definition of full words, $w_Y$ is a full word ending in $\hat{a}$ if an only if $w_Y = w_{Y_1}\hat{a}$ for some $w_{Y_1}$, a full word associated to $Y_1$. Similarly, $w_Y$ is a full word ending in $aa$ if and only if $w_Y = w_{Y_2}a$ for some $w_{Y_2}$, a full word associated to $Y_2$.
Letting $FW_Y$, $FW_{Y_1}$, and $FW_{Y_2}$ denote the set of full words associated to $Y$, $Y_1$, and $Y_2$, and letting $IN_Y$, $IN_{Y_1}$, and $IN_{Y_2}$ denote the set of components of $\textrm{init}(Y)$, $\textrm{init}(Y_1)$, and $\textrm{init}(Y_2)$, we get a diagram
\begin{center}
 \begin{xy}
 (0,20)*+{\{w_Y~|~w_Y\in FW_Y\}}="a"; (95,20)*+{\{w_{Y_1}\hat{a}~|~w_{Y_1}\in FW_{Y_1}\}\cup\{w_{Y_2}a~|~w_{Y_2}\in FW_{Y_2}\}}="b";%
 (0,0)*+{\{S~|~\Spec(S)\in IN_Y\}}="c"; (95,0)*+{\{S_1~|~\Spec(S_1)\in IN_{Y_1}\}\cup\{S_2\cup\{a_n\}~|~\Spec(S_2)\in IN_{Y_2}\}}="d";%
 {\ar "a";"b"};
 {\ar "a";"c"};{\ar "b";"d"};%
 {\ar "d";"c"};%
\end{xy}
\end{center}
where the top and bottom maps are bijections and the right vertical map is given by $m_{Y_1}$ on the first set and $m_{Y_2}$ on the second set. If the left vertical map is given by $m_{Y_2}$, then the diagram commutes. By induction, both $m_{Y_1}$ and $m_{Y_2}$ are well defined and bijective. Thus, so is $m_Y$.

\noindent\textbf{Suppose that} $\mathbf{Y = (s,1,t,+0),~ s<n-1,}$ $\mathbf{Y = (n-u,u,0,+0),}$ \textbf{or} $\mathbf{Y = (0,n,0,+0):}$ Then $m_Y$ is well defined and bijective by a nearly identical argument to the one above.

\noindent\textbf{Suppose} $\mathbf{Y = (s,0,t,+0),}$ $\mathbf{s,t>0:}$ Then, \[\textrm{Revlex}_{b_n}Y = (s,0,t-1,+1)\times\bA^1_{b_n}\cup (s-1,1,t,+0)\times \bA^1_{b_n}.\] Let $Y_1 = (s,0,t-1,+1)$ and let $Y_2 = (s-1,1,t,+0)$. Then, $\Spec k[S]$, $S\subseteq \{a_1,\dots,a_n,b_1,\dots,b_n\}$, is a component of $\textrm{init}(Y)$ if and only if either (i) $S = S_1\cup\{b_n\}$ for some $S_1\subseteq \{a_1,\dots,a_{n-1},$ $b_1,\dots,b_{n-1}\}$ where $\Spec k[S_1]$ is a component of $\textrm{init}(Y_1)$, or (ii) $S = S_2\cup\{b_n\}$ for some $S_2\subseteq \{a_1,\dots,a_{n-1},$ $b_1,\dots,b_{n-1}\}$ where $\Spec k[S_2]$ is a component of $\textrm{init}(Y_2)$. 

Now, a full word $w_Y$ either ends with $a\uparrow$, $\hat{a}\uparrow$, or $aa\uparrow$. Notice that $w_Y$ is a full word ending in $a\uparrow$ if an only if $w_Y = w_{Y_1}\uparrow$ for some $w_{Y_1}$, a full word associated to $Y_1$. Similarly, (i) $w_Y$ is a full word ending in $\hat{a}\uparrow$ if and only if $w_Y = w_{Y_2}\uparrow$ for some $w_{Y_2}$, a full word associated to $Y_2$ ending in $\hat{a}$ and (ii) $w_Y$ is a full word ending in $aa\uparrow$ if and only if $w_Y = w_{Y_2}\uparrow$ for some $w_{Y_2}$, a full word associated to $Y_2$ ending in $aa$. As every full word of $Y_2$ ends in either $\hat{a}$ or $aa$, there is a one-to-one correspondence between full words of $Y$ ending in one of $\hat{a}\uparrow$ or $aa\uparrow$ and full words of $Y_2$.

Letting $FW_Y$, $FW_{Y_1}$, and $FW_{Y_2}$ denote the set of full words associated to $Y$, $Y_1$, and $Y_2$, and letting $IN_Y$, $IN_{Y_1}$, and $IN_{Y_2}$ denote the set of components of $\textrm{init}(Y)$, $\textrm{init}(Y_1)$, and $\textrm{init}(Y_2)$, we get a diagram 
\begin{center}
 \begin{xy}
 (0,20)*+{\{w_Y~|~w_Y\in FW_Y\}}="a"; (91,20)*+{\{w_{Y_1}\uparrow~|~w_{Y_1}\in FW_{Y_1}\}\cup\{w_{Y_2}\uparrow~|~w_{Y_2}\in FW_{Y_2}\}}="b";%
 (0,0)*+{\{S~|~\Spec(S)\in IN_Y\}}="c"; (91,0)*+{\{S_1\cup\{b_n\}~|~\Spec(S_1)\in IN_{Y_1}\}\cup\{S_2\cup\{b_n\}~|~\Spec(S_2)\in IN_{Y_2}\}}="d";%
 {\ar "a";"b"};
 {\ar "a";"c"};{\ar "b";"d"};%
 {\ar "d";"c"};%
\end{xy}
\end{center}
where the top and bottom maps are bijections and the right vertical map is given by $m_{Y_1}$ on the first set and $m_{Y_2}$ on the second set. If the left vertical map is given by $m_{Y_2}$, then the diagram commutes. By the relevant previous cases, both $m_{Y_1}$ and $m_{Y_2}$ are well defined and bijective. Thus, so is $m_Y$.

\noindent\textbf{Suppose} $\mathbf{Y}$ \textbf{is one of} $\mathbf{(n,0,0,+0)}$ \textbf{or} $\mathbf{(0,0,n,+0):}$ $m_Y$ is well defined and bijective by a similar argument to the previous case.
\end{proof}

\begin{remarks}
\begin{enumerate}
\item Let $AW$ denote the set of all strings in the letters $\hat{a}$, $a$, $\uparrow$ such that (i) $\#\hat{a}+\#a = n$ and (ii) the first letter in the string is not $\uparrow$, and (iii) given two consecutive letters, they are not both $\uparrow$. Then, we may apply steps 1, 2, and 3 in the definition of $m_Y$ to assign a unique subset $S\subseteq \{a_1,\dots,a_n,b_1,\dots,b_n\}$ to any $w\in AW$. Let $m$ denote this assignment, \[m: AW \rightarrow \{S~|~S\subseteq \{a_1,\dots,a_n,b_1,\dots,b_n\}\}.\] Then $m$ is a bijection. Indeed, there is a unique way to undo steps 3, 2, and 1 appearing in the definition of $m_Y$.
\item Let $w\in AW$. Using the map from \cite[Theorem 2]{K}), $m(w)$ maps to a unique minimal compatibly split $Y\subseteq U_{\langle x,y^n\rangle}$. If $m(w)$ maps to $Y$ and $w$ is not a full word associated to $Y$, we say that $w$ is a \emph{partial word associated to $Y$}.
\item In light of Proposition \ref{p;pipedreams}, full words may be thought of as an analogy of the reduced pipe dreams that appear in the matrix Schubert variety setting (see \cite[Chapter 15]{MS}). Partial words may be thought of as an analogy of non-reduced pipe dreams.
\end{enumerate}
\end{remarks}

Using Definition \ref{d;fullwords}, we can determine which compatibly split $Y\subseteq U_{\langle x,y^n\rangle}$ is associated to a given full word. However, it may not be obvious which $Y$ is associated to a given partial word. The next proposition shows us how to do this.

\begin{proposition}
Let $w\in AW$. Then, either $w$ does not contain any copies of $\uparrow$, or $w$ can be written in the form $w = w_1\uparrow w_2$ where $w_1$ is a string in $a$, $\hat{a}$, and $\uparrow$, and $w_2$ is a string in $a$ and $\hat{a}$. In addition, $w\in AW$ is a partial word if and only if at least one of the following holds:
\begin{enumerate}
\item $w_1$ contains at least one of the following subwords: $aaa$, $\hat{a}a$, $a\hat{a}$, or $\hat{a}\hat{a}$.
\item $w_2$ contains a subword of the form $a\cdots a\hat{a}$ where the number of copies of $a$ appearing before the $\hat{a}$ is odd.
\end{enumerate}

We may fill up a partial word associated to $Y$ (in a non-unique way) to obtain a full word associated to $Y$. This can be done using the following procedure:
\begin{enumerate}
\item Write $w = w_1\uparrow w_2$ if $w$ contains some $\uparrow$ and proceed to step 2. If $w$ does not contain any $\uparrow$, set $w=w_2$ and proceed to step 3.
\item Consider the subword $(w_1\uparrow)$. Within this subword, move from left to right inserting a copy of $\uparrow$ in the first spot from the left where $w$ fails to be a full word. Denote the resulting word by $(\tilde{w_1}\uparrow)$. Eg. Suppose $(w_1\uparrow) = a\hat{a}aaa\hat{a}\uparrow$. Move left to right inserting copies of $\uparrow$ in the following way: \[a\hat{a}aaa\hat{a}\uparrow ~\mapsto~ a\uparrow\hat{a}aaa\hat{a}\uparrow ~\mapsto~ a\uparrow\hat{a}\uparrow aaa\hat{a}\uparrow ~\mapsto~ a\uparrow\hat{a}\uparrow aa\uparrow a\hat{a}\uparrow ~\mapsto~ a\uparrow\hat{a}\uparrow aa\uparrow a\uparrow \hat{a}\uparrow\]
\item Consider the subword $w_2$. Within this subword, move from left to right changing $\hat{a}$ to $a$ in the first spot from the left where $w$ fails to be a full word. Denote the resulting word by $\tilde{w_2}$. Eg. Suppose $w_2 = a\hat{a}aa\hat{a}a\hat{a}\hat{a}$. Move left to right changing $\hat{a}$ to $a$ as follows: \[a\hat{a}aa\hat{a}a\hat{a}\hat{a}~\mapsto~aaaa\hat{a}a\hat{a}\hat{a}~\mapsto~aaaa\hat{a}aa\hat{a}\]
\item Replace $w$ by $\tilde{w} = \tilde{w_1}\uparrow\tilde{w_2}$ if $w$ contains some $\uparrow$. Replace $w$ by $\tilde{w_2}$ if $w$ does not contain any $\uparrow$. 
\end{enumerate}
\end{proposition}

\begin{proof}
The first part of the proposition is clear by the definition of full and partial words. We now show that the stated procedure for filling up a partial word associated to $Y$ produces a full word associated to $Y$. 

Let $w$ be a partial word associated to $Y$. By construction, $\tilde{w}$ is a full word, and thus it is necessarily associated to some compatibly split $Y'\subseteq U_{\langle x,y^n\rangle}$ (by Remark \ref{r;fw}). We show that $Y' = Y$.

By construction, $m(w)\subseteq m(\tilde{w})$ where $m$ is the map described in the remarks prior to the proposition. So, to show that $\tilde{w}$ is a full word associated to $Y$, we show that there is no $S$ such that $m(w)\subsetneq S\subsetneq m(\tilde{w})$ and $m^{-1}(S)$ is a full word. Suppose that $S$ has the property that $m(w)\subsetneq S\subsetneq m(\tilde{w})$. Then, $m^{-1}(S)$ must be equal to $\tilde{w}$ transformed by either (i) removing at least one $\uparrow$ in $\tilde{w_1}$ that had been added to $w_1$, or (ii) taking at least one $a$ in $\tilde{w_2}$ that came from the letter $\hat{a}$ in $w_2$ and changing it back to the letter $\hat{a}$. But then, $m^{-1}(S)$ is not a full word.
\end{proof}

\begin{remark}
Let $w$ denote a partial word associated to a compatibly split $Y\subseteq U_{\langle x,y^n\rangle}$. Then $m(w)$ is a face of $\Delta_Y$, the simplicial complex associated to $\textrm{init}(Y)$. Suppose that $m(w)$ is contained in $k$ facets of $\Delta_Y$, then there are $k$ different full words associated to $Y$, $w_1,\dots,w_k$, such that $m(w)\subseteq m(w_i)$, $1\leq i\leq k$.
\end{remark}

We now use the full words to study the simplicial complex $\Delta_Y$ associated to the Stanley-Reisner scheme $\textrm{init}(Y)$. In particular, we show that if a stratum representative of $Y$ can be chosen to either (i) have no points in $\bA^2_k\setminus \{(0,0)\}$ or (ii) have at most one point on the punctured $y$-axis, then $\Delta_Y$ is a vertex-decomposable ball. 

\begin{proposition}\label{balls}
Let $Y$ be a compatibly split subvariety of $U_{\langle x,y^n\rangle}$. Let $\Delta_Y$ denote the simplicial complex associated to the Stanley-Reisner scheme $\textrm{init}(Y)$.
\begin{enumerate}
\item If $Y = (s,u,0,+0)$, $u\geq 0$, or $Y = (s,u,0,+1)$, $u\geq 2$, or $Y = (n-1,0,0,+1)$, then $\Delta_Y$ is a simplex.
\item If $Y = (0,0,n,+0)$ or $Y = (0,0,n-1,+1)$, then $\Delta_Y$ is a simplex.
\item If $Y = (0,u,n-u,+0)$, $u\geq 1$, or $Y = (0,u,n-u,+1)$, $u\geq 2$, then $\Delta_Y$ is a vertex-decomposable ball. Furthermore, $\partial Y$, the union of codimension-$1$ compatibly split subvarieties of $Y$, degenerates to the Stanley-Reisner scheme of $\partial\Delta_Y$, the boundary sphere of $\Delta_Y$. 
\item If $Y = (1,u,n-u-1,+0)$, $u\geq 0$, or $Y = (1,u,n-u-1,+1)$, $u\geq 2$, or $Y = (1,0,n-2,+1)$, then $\Delta_Y$ is a vertex-decomposable ball.
\end{enumerate}
\end{proposition}

\begin{proof}
Let $J(Y)\subseteq k[a_1,\dots,a_n,b_1,\dots,b_n]$ denote the ideal of $Y\subseteq U_{\langle x,y^n\rangle}$. 

By Proposition \ref{degens}, if $Y = (s,u,0,+0)$, $u\geq 0$, then $J(Y) = \langle b_n,\dots,b_{n-u+1},a_1,\dots,a_n\rangle$. If $Y = (s,u,0,+1)$, $u\geq 2$, then $J(Y) = \langle b_n,\dots,b_{n-u+1},a_1,\dots,a_{n-1}\rangle$. If $Y = (n-1,0,0,+1)$, then $J(Y) = \langle b_n,a_1,\dots,a_{n-1}\rangle$. Thus, in each of these cases, $\textrm{init}(Y) = Y$ and $\Delta_Y$ is a simplex.

Similarly, when $Y = (0,0,n,+0)$ or $Y = (0,0,n-1,+1)$, we get that $\textrm{init}(Y) = Y$ and $\Delta_Y$ is a simplex. 

We now consider the case $Y = (0,u,n-u,+0)$, $u\geq 1$ or $Y = (0,u,n-u,+1)$, $u\geq 2$. We proceed by induction on $n$ to see that $\Delta_Y$ is a vertex-decomposable ball. When $n=1,2$, the result holds (see Figure \ref{fig;posetcomplex}). So, let $n$ be arbitrary. 

\noindent\textbf{Suppose that} $\mathbf{Y = (0,u,n-u,+0)}$ \textbf{and} $\mathbf{u=1}:$ By Proposition \ref{degens}, \[\textrm{Lex}_{a_n}(0,1,n-1,+0) = [(0,0,n-1,+0)\times 0_{b_n}\times 0_{a_n}]\cup [(0,0,n-2,+1)\times 0_{b_n}\times \bA^1_{a_n}].\] 
Let $Y_1 = (0,0,n-1,+0)$ and let $Y_2 = (0,0,n-2,+1)$. The only full word associated to $Y_1$ is $a\uparrow\cdots a\uparrow$ where the segment $(a\uparrow)$ repeats $n-1$ times, and the only full word associated to $Y_2$ is $a\uparrow\cdots a\uparrow a$ where the segment $(a\uparrow)$ repeats $n-2$ times. Using the association of words to simplices given in Proposition \ref{p;pipedreams}, we can directly see that $\Delta_Y$ is a vertex-decomposable ball. 

\noindent\textbf{Suppose that} $\mathbf{Y = (0,u,n-u,+0)}$ \textbf{and} $\mathbf{u \geq 2:}$ By Proposition \ref{degens}, \[\textrm{Lex}_{a_n}(0,u,n-u,+0) = [(0,u-1,n-u,+0)\times 0_{b_n}\times 0_{a_n})]\cup [(0,u,n-u-1,+1)\times 0_{b_n}\times \bA^1_{a_n}].\] 
Let $Y_1 = (0,u-1,n-u,+0)$ and let $Y_2 = (0,u,n-u-1,+1)$. Now vertex decompose $\Delta_Y$ with respect to $a_n$. From above, we see that $\textrm{del}(a_n) = \Delta_{Y_1}$ and $\textrm{link}(a_n) = \Delta_{Y_2}$.

By induction, each of $\Delta_{Y_1}$ and $\Delta_{Y_2}$ are vertex-decomposable balls. So, to get that $\Delta_Y$ is a vertex-decomposable ball, it suffices to show that $\textrm{link}(a_n)$ is contained inside of the boundary sphere of $\textrm{del}(a_n)$. To do this, we show that each facet of $\textrm{link}(a_n)$ is contained in exactly one facet of $\textrm{del}(a_n)$. 

\begin{itemize}
\item If $w_{lk}$ is a full word associated to a facet of $\textrm{link}(a_n)$, then $w_{lk}$ can be written as $w_{lk} = (w_{lk}')a$, where the number of segments of letters of type $(a\uparrow)$, $(\hat{a}\uparrow)$, $(aa\uparrow)$, $(aa)$, and $(\hat{a})$ in $w_l'$ satisfies the conditions $\#(\hat{a}\uparrow) = \# (aa\uparrow) = 0$, $\#(a\uparrow)+\#(aa)=n-u-1$, and $\#(aa)+\#(\hat{a}) = u-1$.
\item If $w_d$ is a full word associated to a facet of $\textrm{del}(a_n)$, then the number of segments of letters of type $(a\uparrow)$, $(\hat{a}\uparrow)$, $(aa\uparrow)$, $(aa)$, and $(\hat{a})$ in $w_d$ satisfies the conditions $\#(\hat{a}\uparrow) = \# (aa\uparrow) = 0$, $\#(a\uparrow)+\#(aa) = n-u$, and $\#(aa)+\#(\hat{a}) = u-1$.
\end{itemize}

Therefore, either (i) $w_{lk}$ has one fewer $(a\uparrow)$ and the same number of $(aa)$ and $(\hat{a})$ as $w_d$ or (ii) $w_{lk}$ has one fewer $(aa)$, one extra $(\hat{a})$, and the same number of $(a\uparrow)$ as $w_d$. 

Now, let $m$ denote the map that assigns simplices to full words and let $w_{lk}$ be given. Write $w_{lk}$ as $w_{lk} = w_{lk,1}\uparrow w_{lk,2}$. Suppose first that there are no $\hat{a}$ in $w_{lk,2}$. Then, the only way to obtain a full word $w_d$ associated to a facet in $\textrm{del}(a_n)$ such that $m(w_{lk})\subseteq m(w_d)$ is to insert a copy of $\uparrow$ after the \emph{first} occurance of the letter $a$ in $w_{lk,2}$.

Now suppose that $w_{lk,2}$ contains at least one $\hat{a}$. Changing the \emph{last} occurance of $\hat{a}$ to $a$ yields a word $w_d$ associated to a facet of $\textrm{del}(a_n)$ such that $m(w_{lk})\subseteq m(w_d)$. Furthermore, noting (i) and (ii) above, we can check that no other change to $w_{lk}$ will yield a full word associated to a facet of $\textrm{del}(a_n)$ such that $m(w_{lk})\subseteq m(w_d)$. Indeed, changing any other appearance of $\hat{a}$ to $a$ will necessarily yield a partial word since there will be a string $a\cdots a\hat{a}$ with an odd number of the letter $a$ appearing before $\hat{a}$. In addition, it is not permissible to add a copy of $\uparrow$ to $w_{lk}$ as this will again cause there to be an odd number of the letter $a$ in a string of the form $a\cdots a\hat{a}$.

Thus, $\textrm{link}(a_n)$ is contained in the boundary of $\textrm{del}(a_n)$ and it follows that $\Delta_Y$ is a vertex-decomposable ball.

\noindent\textbf{Suppose that} $\mathbf{Y = (0,u,n-u,+1),}$ $\mathbf{u\geq 2:}$ Then $\textrm{Lex}_{a_n} = (0,u-1,n-u,+0)\times 0_{b_n}\times \bA^1_{a_n}$ and so $\Delta_Y$ is a cone on $\Delta_{(0,u-1,n-u,+0)}$. As $\Delta_{(0,u-1,n-u,+0)}$ is a vertex-decomposable ball, so is $\Delta_Y$.


\noindent\textbf{Suppose that} $\mathbf{Y = (1,0,n-1,+0):}$ Then, \[\Delta_Y = \{S\cup\{b_n\}~|~S\subseteq \{a_1,b_1,\dots,a_{n-1},b_{n-1},a_n\},~|S|=2n-2\}.\]
To see that $\Delta_Y$ is a vertex-decomposable ball, we show that \[X = \{S\cup\{x_{m}\}~|~S\subseteq \{x_1,\dots,x_{m-1}\},~|S|=m-2\}\] is a vertex-decomposable ball. To do so, we proceed by induction on $m$. When $m=1$, the result is trivial. Now let $m$ be arbitrary and vertex decompose with respect to $x_{m-1}$. Then $\textrm{del}(x_n)$ is a single simplex $\{x_1,\dots,x_{m-2},x_m\}$ and so is a vertex-decomposable ball. Also, \[\textrm{link}(x_{m-1}) = \{S\cup\{x_{m}\}~|~S\subseteq \{x_1,\dots,x_{m-2}\},~|S|=m-3\}\] and so is a vertex-decomposable ball by induction. It's clear that $\textrm{link}(x_{m-1})$ is contained in the boundary sphere of $\textrm{del}(x_n)$. Thus, $X$ is a vertex-decomposable ball and so $\Delta_Y$ is too.  



Next, consider the case $Y = (1,u,n-u-1,+0)$, $u\geq 1$, or $Y = (1,u,n-u-1,+1)$, $u\geq 2$, or $Y = (1,0,n-2,+1)$. Proceed by induction on $n$. When $n=1$ or $n=2$, $\Delta_Y$ is a vertex-decomposable ball (see Figure \ref{fig;posetcomplex}).

\noindent\textbf{Suppose that} $\mathbf{Y = (1,1,n-2,+0),}$ $\mathbf{n\geq 3:}$ By Proposition \ref{degens}, \[\textrm{Lex}_{a_n}(Y) = [(1,0,n-2,+0)\times 0_{b_n}\times 0_{a_n}]\cup [(1,0,n-3,+1)\times 0_{b_n}\times \bA^1_{a_n}].\] Let $Y_1 = (1,0,n-2,+0)$ and let $Y_2 = (1,0,n-3,+1)$. Now vertex decompose $\Delta_Y$ with respect to $a_n$. From above, we see that $\textrm{del}(a_n) = \Delta_{Y_1}$ and $\textrm{link}(a_n) = \Delta_{Y_2}$. By the previous case, each of $\Delta_{Y_1}$ and $\Delta_{Y_2}$ are vertex-decomposable balls. As each facet of $F\in \textrm{link}(a_n)$ is contained in exactly one facet of $\textrm{del}(a_n)$ (i.e. $F\cup \{b_{n-1}\}$), we see that $\textrm{link}(a_n)$ is contained in the boundary of $\textrm{del}(a_n)$ and so $\Delta_Y$ is a vertex-decomposable ball.

\noindent\textbf{Suppose that} $\mathbf{Y = (1,u,n-u-1,+0),}$ $\mathbf{u\geq 2:}$ By Proposition \ref{degens}, \[\textrm{Lex}_{a_n}(Y) = [(1,u-1,n-u-1,+0)\times 0_{b_n}\times 0_{a_n}]\cup [(1,u,n-u-2,+1)\times 0_{b_n}\times \bA^1_{a_n}].\] Let $Y_1 = (1,u-1,n-u-1,+0)$ and let $Y_2 = (1,u,n-u-2,+1)$. By induction, each of $\Delta_{Y_1}$ and $\Delta_{Y_2}$ are vertex-decomposable balls. So, to get that $\Delta_Y$ is a vertex-decomposable ball, it again suffices to show that $\textrm{link}(a_n)$ is contained inside of the boundary sphere of $\textrm{del}(a_n)$. This argument is nearly identical to the one given when $Y = (0,u,n-u,+0)$, $u\geq 2$, so we do not repeat it. 

\noindent\textbf{Suppose that} $\mathbf{Y = (1,u,n-u-1,+1),}$ $\mathbf{u\geq 2},$ \textbf{or} $\mathbf{Y = (1,0,n-2,+1):}$ As before, $\Delta_Y$ is the cone on some $\Delta_{Y'}$ which is already known to be a vertex-decomposable ball. Thus, $\Delta_Y$ is a vertex-decomposable ball.

To complete the proof, it remains to show that $\partial Y$ degenerates to the Stanley-Reisner scheme of the boundary sphere of $\Delta_Y$ when $Y = (0,u,n-u,+0)$, $u\geq 1$, or $Y = (0,u,n-u,+1)$, $u\geq 2$. Proceed by induction on $n$. When $n=1,2$, the result is either trivial or seen to be true in Figure \ref{fig;posetcomplex}.

Now let $Y = (0,u,n-u,+0)$, $u\geq 2$. First note that \[\partial Y = (1,u,n-u-1,+0)\cup (0,u+1,n-u-1,+1).\] Let $C = (1,u,n-u-1,+0)$ and let $D = (0,u+1,n-u-1,+1)$. Then, 
\begin{enumerate}
\item $\textrm{Lex}_{a_n}(Y) = [(0,u-1,n-u,+0)\times 0_{b_n}\times 0_{a_n}]\cup[(0,u,n-u-1,+1)\times 0_{b_n}\times \bA^1_{a_n}]$. Let $Y_1$ and $Y_2$ denote the two components of $\textrm{Lex}_{a_n}(Y)$.
\item $\textrm{Lex}_{a_n}(C) = [(1,u-1,n-u-1,+0)\times 0_{b_n}\times 0_{a_n}]\cup [(1,u,n-u-2,+1)\times 0_{b_n}\times \bA^1_{a_n}]$. Let $C_1$ and $C_2$ denote the two components of the union. 
\item $\textrm{Lex}_{a_n}(D) = (0,u,n-u-1,+0)\times 0_{b_n}\times \bA^1_{a_n}$. 
\item By 1., 2., and 3., $(\partial Y_1) \setminus (Y_1\cap Y_2) = C_1$ and $(\partial Y_2)\setminus (Y_1\cap Y_2) = C_2\cup\textrm{Lex}_{a_n}D$.
\end{enumerate}
By induction, $(\partial Y_1) \setminus (Y_1\cap Y_2)$ and $(\partial Y_2)\setminus (Y_1\cap Y_2)$ degenerate to Stanley-Reisner schemes of subsets of the boundary spheres of $\Delta_{Y_1}$ and $\Delta_{Y_2}$. By item 4., $\textrm{Lex}_{a_n}C\cup \textrm{Lex}_{a_n}D$ degenerates to the Stanley-Reisner scheme of the boundary sphere of $\Delta_Y$.

The remaining (very similar) cases are left to the reader.
\end{proof}

\begin{example}
As an example of the remaining case, consider $Y = (2,0,1,+0)\subseteq U_{\langle x,y^3\rangle}$. The simplicial complex associated to $\textrm{init}(Y)$ has the following facets: $\{a_2,a_3,b_1,b_3\}$, $\{a_3,b_1,b_2,b_3\}$, $\{a_1,b_1,b_2,b_3\}$, $\{a_1,a_2,b_2,b_3\}$, $\{a_2,b_1,b_2,b_3\}$.  Therefore $\Delta_Y$ is not homeomorphic to a ball (see Figure \ref{fig;badcomplex}). 
\begin{figure}[!h]
\begin{center}
\includegraphics[scale = 0.6]{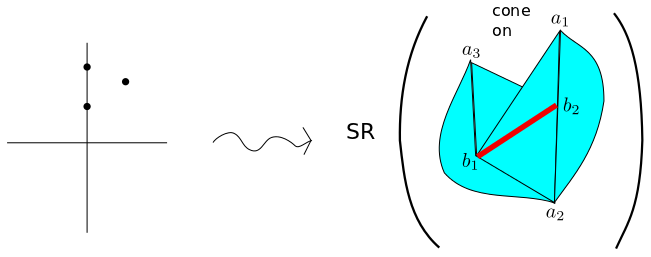}
\caption{The simplicial complex associated to $\textrm{init}(2,0,1,+0)$. Note that the red edge is the simplicial complex associated to $(3,0,0,+0)$, the non-R1 locus of $(2,0,1,+0)$.}
\label{fig;badcomplex}
\end{center}
\end{figure}
\end{example}

We obtain the following consequence of Proposition \ref{balls}.

\begin{corollary}
The compatibly split subvarieties of $U_{\langle x,y^n\rangle}$ listed in Proposition \ref{balls} are Cohen-Macaulay.
\end{corollary}

\begin{proof}
Each $Y$ described in Proposition \ref{balls} degenerates onto the Stanley-Reisner scheme of a shellable ball. Thus, $\textrm{init}(Y)$ is Cohen-Macaulay. By semicontinuity (of local cohomology), so is $Y$. 
\end{proof}



\renewcommand\bibname{\centerline{BIBLIOGRAPHY}}
\bibliographystyle{amsalpha}	
\bibliography{jenna}
\addcontentsline{toc}{section}{{\bf Bibliography}}
\backmatter

\end{document}